\documentclass{amsart}
\usepackage{amsmath}
\usepackage[psamsfonts]{amssymb}
\usepackage[all]{xy}
\usepackage{graphicx}
\usepackage[T1]{fontenc}
\usepackage[bitstream-charter]{mathdesign}
\usepackage{hyperref}
\newtheorem{theorem}{Theorem}[section]
\newtheorem{prop}[theorem]{Proposition}
\newtheorem{lemma}[theorem]{Lemma}
\newtheorem{cor}[theorem]{Corollary}

\newtheorem{ex}[theorem]{Example}
\theoremstyle{remark}

\newtheorem{remark}[theorem]{Remark}

\def\co{\colon\thinspace}

\def\ep{\epsilon}

\def\R{\mathbb{R}}
\def\Z{\mathbb{Z}}

\def\C{\mathbb{C}}

\title{On Certain Lagrangian Submanifolds of $S^2\times S^2$ and $\C P^n$}
\author{Joel Oakley\and Michael Usher}
\address{Department of Mathematics\\University of Georgia\\Athens, GA 30602}
\email{\texttt{joakley@math.uga.edu}}
\email{\texttt{usher@math.uga.edu}}
\begin{document}
\begin{abstract}  We consider various constructions of monotone Lagrangian submanifolds of $\C P^n, S^2\times S^2$, and  quadric hypersurfaces of $\C P^n$.  In $S^2\times S^2$ and $\C P^2$ we show that several different known constructions of exotic monotone tori yield results that are Hamiltonian isotopic to each other, in particular answering a question of Wu by showing that the monotone fiber of a toric degeneration model of $\C P^2$ is Hamiltonian isotopic to the Chekanov torus.  Generalizing our constructions to higher dimensions leads us to consider monotone Lagrangian submanifolds (typically not tori) of quadrics and of $\C P^n$ which can be understood either in terms of the geodesic flow on $T^*S^n$ or in terms of the Biran circle bundle construction.  Unlike previously-known monotone Lagrangian submanifolds of closed simply connected symplectic manifolds, many of our higher-dimensional Lagrangian submanifolds are provably displaceable.
\end{abstract}

\maketitle

\section{Introduction}
A basic, though generally unattainable, goal in symplectic topology is the classification of the Lagrangian submanifolds $L$ of a given symplectic manifold $(M,\omega)$.  Such submanifolds have a pair of classical invariants, both homomorphisms defined on the group $\pi_2(M,L)$: the area homomorphism $I_{\omega}\co \pi_2(M,L)\to\R$ which  maps the homotopy class of a disk $u$ to the area $\int_{D^2}u^*\omega$; and the Maslov homomorphism $I_{\mu}\co \pi_2(M,L)\to \Z$ which maps the homotopy class of $u$ to the Maslov index of the loop of Lagrangian subspaces given by $u|_{\partial D^2}^{*}TL$ with respect to a symplectic trivialization of $u^{*}TM$.  The Lagrangian submanifold $L$ is called \emph{monotone} if there is a constant $\lambda>0$ such that $I_{\omega}=\lambda I_{\mu}$.  This is a rather restrictive hypothesis---for instance it requires that the $\omega$-area of any two-sphere in $M$ be $\frac{\lambda}{2}$ times its Chern number, so many symplectic manifolds have no monotone Lagrangian submanifolds and in those that do the parameter $\lambda$ is often (\emph{e.g.} if $M$ is simply-connected and compact) determined by the ambient manifold.  A more modest goal than the classification of all Lagrangian submanifolds is to understand the monotone Lagrangian submanifolds of a given symplectic manifold---in particular the specialization to the monotone case restricts the classical invariants in a natural way.

The notion of monotonicity was originally introduced with the somewhat technical goal of facilitating the construction of Floer homology \cite{Oh95}, though over time it has become clearer that monotonicity is a geometrically significant property.  To give a simple special case, on the sphere $S^2$ the lines of latitude (which are the fibers of the moment map for the standard Hamiltonian $S^1$-action on $S^2$) are all Lagrangian submanifolds. Among these, the  only one which is monotone is the equator; the equator is also the only line of latitude $L$ which is \emph{nondisplaceable} in the sense that for every Hamiltonian diffeomorphism $\phi$ of $S^2$ we have $\phi(L)\cap L\neq\varnothing$.  A considerable generalization of this has been proven by independently by Cho and Entov-Polterovich \cite{Cho} \cite{EPrigid}: for any monotone toric symplectic manifold the unique fiber of the moment map which is monotone is also nondisplaceable; conversely it is often (though not always) the case that all other fibers of the moment map are displaceable \cite{Mprobes}.

This paper concerns monotone Lagrangian submanifolds in complex projective space $\C P^n$ and in quadric hypersurfaces thereof.  Our normalization convention throughout will be to treat $\C P^n$ as the coisotropic reduction of the sphere of radius $\sqrt{2}$ in $\C^{n+1}$, which then induces a symplectic form on the quadric in $\C P^n$ by restriction; we will emphasize these conventions by denoting the relevant manifolds as $\C P^{n}(\sqrt{2})$ and $Q_n(\sqrt{2})$.  The four-real-dimensional cases of $\C P^2(\sqrt{2})$ and $Q_3(\sqrt{2})$ will receive special attention. Note that $Q_3(\sqrt{2})$ is symplectomorphic to $S^2\times S^2$, where the latter is equipped with a symplectic form giving area $2\pi$ to each factor.  

\subsection{Results in dimension four} Each of  $\C P^{2}(\sqrt{2})$ and $S^2\times S^2$ contains a standard  monotone Lagrangian torus, in each case sometimes called the Clifford torus: in the first case one views $\C P^{2}(\sqrt{2})$ as a compactification of the ball of radius $\sqrt{2}$ in $\C^2$ and considers the torus given as the product of the circles of radius $\sqrt{2/3}$ in each factor, while in the second case one simply takes the product of the equators in the two $S^2$ factors of $S^2\times S^2$.  For some time it was unknown whether  the Clifford tori were the only monotone Lagrangian tori in these manifolds up to Hamiltonian isotopy; however by now there are a variety of constructions of monotone Lagrangian tori which are not Hamiltonian isotopic to the Clifford tori.  One of the goals of this paper is to clarify the relationship between several of these constructions.

In particular, here are sketches of five ways of constructing a monotone Lagrangian torus in $S^2\times S^2$ (more specific details will be provided later in the paper):
\begin{itemize}\item Begin with a suitably-scaled version, denoted in this paper by $\mathcal{P}_{0,1}^{1/2}$, of the torus in $T^{*}S^2$ considered in \cite{AF}, and consider the image $T_{AF}$ of $\mathcal{P}_{0,1}^{1/2}$ under a symplectic embedding of  $D_{1}^{*}S^2$ into $S^2\times S^2$ where $D_{1}^{*}S^2$ is the unit disk bundle in $T^*S^2$.
\item As in \cite{FOOO}, begin with the fiber over $(1/2,1/2)$ of the toric orbifold whose moment polygon is the triangle with vertices $(0,0)$, $(0,1)$, and $(2,0)$, and then desingularize the orbifold to obtain a manifold symplectomorphic to $S^2\times S^2$ containing a monotone Lagrangian torus $T_{FOOO}$.
\item Let $T_{CS}$ denote the image of the ``twist torus'' $\Theta_1$ in $B^2(\sqrt{2})\times B^2(\sqrt{2})\subset \R^4$ from \cite{CS} under the standard dense symplectic embedding of $B^2(\sqrt{2})\times B^2(\sqrt{2})$ into $S^2\times S^2$.
\item As in \cite[Section 3]{Ga}, Consider the polarization of $S^2\times S^2$ given by the diagonal $\Delta$, and let $T_{BC}$ be the result of the Biran circle bundle construction \cite{Bir} using an equatorial circle in $\Delta$.
\item As in \cite{EPrigid}, where $e_1$ is the first standard basis vector in $\R^3$ let \[ T_{EP} = \left\{(v,w)\in S^2\times S^2\left|\thinspace(v+w)\cdot e_1=0,\,v\cdot w=-1/2\right.\right\} \]
\end{itemize}

We prove:
\begin{theorem}\label{mains2s2} All of the five tori $T_{AF}$, $T_{FOOO}$, $T_{CS}$, $T_{BC}$, and $T_{EP}$ are Hamiltonian isotopic\footnote{Since some of these tori are constructed not in $S^2\times S^2$ but rather in some other manifold $M$ which is symplectomorphic to $S^2\times S^2$, whether a Hamiltonian isotopy exists between, \emph{e.g.}, $T_{FOOO}$ and $T_{EP}$ might seem to depend on the choice of symplectomorphism from $M$ to $S^2\times S^2$.  However since every symplectomorphism of $S^2\times S^2$ can be written as the composition of a Hamiltonian diffeomorphism and the diffeomorphism which switches the two factors of $S^2\times S^2$ \cite[0.3.C]{Gr}, and since $T_{EP}$ is invariant under this latter diffeomorphism, there is in fact no such dependence.} to each other.
\end{theorem}

The fact that $T_{CS}$ and $T_{BC}$ are Hamiltonian isotopic is proven in \cite{Ga}, though we give a somewhat different argument.  The equivalences of the other tori seem to have been widely expected, and in some cases may be known to a small number of experts, but we have not seen full proofs in the literature.  Theorem \ref{mains2s2} is proven in several parts, as Propositions \ref{FOOOtoEP}, \ref{CStoEP}, and \ref{bcep}.

Many of the tori listed above have analogues which are monotone Lagrangian tori in $\C P^{2}(\sqrt{2})$, namely:
\begin{itemize}\item A suitably-scaled version of the torus in $T^*S^2$ from \cite{AF} descends under the natural quotient map to a torus $\underline{\mathcal{P}}_{0,1}^{1/3}\subset T^{*}\R P^2$, which then includes into $\C P^2(\sqrt{2})$ as a monotone torus $T_{AF}^{P}$ under a Darboux-Weinstein neighborhood of $\R P^2$
\item As in \cite{Wu}, begin with the fiber over $(1/3,1/3)$ in the toric orbifold whose moment polygon has vertices $(0,0)$, $(0,1/2)$, and $(2,0)$, 
and then smooth the singular point to obtain a manifold symplectomorphic to $\C P^2(\sqrt{2})$ containing a monotone Lagrangian torus $T_W$.
\item Let $T_{CS}^{P}$ denote the image under the standard dense symplectic embedding $B^4(\sqrt{2})\hookrightarrow \C P^2(\sqrt{2})$ of the twist torus $\Theta_1$ from \cite{CS}.
\item As in \cite[Section 6.4]{BC09}, consider the polarization of $\C P^{2}(\sqrt{2})$ given by the quadric hypersurface $Q_2(\sqrt{2})$, which is symplectomorphic to $S^2$, and let $T_{BC}^{P}$ be the result of the Biran circle bundle construction using an equatorial circle in $Q_2(\sqrt{2})$.
\end{itemize}

Similarly to Theorem \ref{mains2s2}, we have:
\begin{theorem}\label{maincp2} All of the tori $T_{AF}^{P}$, $T_W$, $T_{CS}^{P}$, and $T_{BC}^{P}$ are Hamiltonian isotopic\footnote{As has been known since \cite{Gr}, the symplectomorphism group of $\C P^2$ is equal to its Hamiltonian diffeomorphism group, so a similar remark applies here as in the previous footnote.} to the following torus in $\C P^2(\sqrt{2})$:\[ L_{0,1}^{P} = \left\{[z_0:z_1:z_2]\in \C P^2(\sqrt{2})\left|\left|\sum z_{j}^{2}\right|=\frac{4\sqrt{2}}{3},\,Im(\bar{z}_1z_2) = 0  \right.\right\}\]
\end{theorem}

Similarly to the case of $S^2\times S^2$ the tori $T_{CS}^{P}$ and $T_{BC}^{P}$ were proven to be Hamiltonian isotopic in \cite{Ga}; however the other equivalences do not appear to have been known, and in particular the question of whether $T_{W}$ is Hamiltonian isotopic to $T_{CS}^{P}$ is asked in \cite[Remark 1.3]{Wu}.  Theorem \ref{maincp2} is proven in Section \ref{cp2eq} and Proposition \ref{cpncorr}.

It has been known for some time that the tori in Theorems \ref{mains2s2} and \ref{maincp2} are not Hamiltonian isotopic to the respective Clifford tori (see, \emph{e.g.}, \cite[Theorem 1]{CS}).  Recently, Vianna \cite{Vi} has constructed a  monotone Lagrangian torus in $\C P^2$ which is Hamiltonian isotopic neither to the Clifford torus nor to $T_{CS}^{P}$.

\subsection{Results in higher dimensions}  Some of the tori in Theorems \ref{mains2s2} and \ref{maincp2} have natural generalizations to higher-dimensional monotone Lagrangian submanifolds.  Higher-dimensional  versions of the twist tori $T_{CS}$ are discussed in \cite{CS}, yielding many mutually inequivalent, non-displaceable monotone Lagrangian tori in $(S^2)^n$ and $\C P^n$.  Here we will instead consider certain generalizations of $T_{AF}$ and $T_{BC}$, yielding monotone Lagrangian submanifolds (typically not tori) in $\C P^n(\sqrt{2})$ and in the quadrics $Q_{n}(\sqrt{2})$.

More specifically, in Section \ref{afintro} we introduce, for any pair of natural numbers $k\leq m$ with $m\geq 1$ and any positive real number $r$, a submanifold $\mathcal{P}_{k,m}^{r}\subset T^{*}S^{k+m+1}$ as the union of the cotangent lifts of speed-$r$ geodesics connecting the linked spheres $S^k\times\{\vec{0}\}$ and $\{\vec{0}\}\times S^m$ in $S^{k+m+1}$. The torus $\mathcal{P}_{0,1}^{1}$ is the subject of \cite{AF}, and more generally a remark near the end of \cite{AF} discusses $\mathcal{P}_{0,m}^{1}$.  To obtain a Lagrangian submanifold in a closed manifold there are two natural ways to proceed: one can perform a symplectic cut \cite{L} on $T^*S^{k+m+1}$ to obtain a manifold symplectomorphic to $Q_{k+m+2}(\sqrt{2})$; or one can instead first take the quotient of $T^{*}S^{k+m+1}$ by the antipodal involution, yielding $T^{*}\R P^{k+m+1}$, and then perform a symplectic cut to obtain a manifold symplectomorphic to $\C P^{k+m+1}(\sqrt{2})$.  In either case, we will find a unique value of $r$ for which the image of $\mathcal{P}_{k,m}^{r}$ under the indicated operations on the ambient manifold is monotone, giving rise to monotone Lagrangian submanifolds $L_{k,m}^{Q}\subset Q_{k+m+2}(\sqrt{2})$ and $L_{k,m}^{P}\subset \C P^{k+m+1}(\sqrt{2})$.  (Strictly speaking our presentation will not explicitly use the symplectic cut, though it is not difficult to see that it can be cast in these terms; rather we will directly construct  symplectomorphisms from the open disk bundles $D_{1}^{*} S^{k+m+1}$ and $D_{1}^{*}\R P^{k+m+1}$ to the complements of symplectic hypersurfaces in $Q_{k+m+2}(\sqrt{2})$ and $\C P^{k+m+1}(\sqrt{2})$, respectively.) Up to diffeomorphism one has $L_{k,m}^{Q}\cong \frac{S^1\times S^k\times S^m}{\Z/2\Z}$ and $L_{k,m}^{P}\cong S^1\times\left(\frac{S^k\times S^m}{\Z/2\Z}\right)$ where in each case $\Z/2\Z$ acts simultaneously by the antipodal map on each factor.  In particular $L_{0,m}^{P}$ and $L_{0,m}^{Q}$ are each diffeomorphic to $S^1\times S^m$, while $L_{1,1}^{Q}$ and $L_{1,1}^{P}$ are each three-dimensional tori.

Another way of constructing monotone Lagrangian submanifolds of $Q_{n+1}(\sqrt{2})$ or $\C P^{n}(\sqrt{2})$ proceeds as follows.  Each of these manifolds is a K\"ahler manifold containing the quadric $Q_n(\sqrt{2})$ as an ample divisor, and so a given monotone Lagrangian submanifold $\Lambda\subset Q_{n}(\sqrt{2})$ induces by the Biran circle bundle construction \cite[Section 4.1]{Bir}  monotone Lagrangian submanifolds $\Lambda^Q\subset Q_{n+1}(\sqrt{2})$ and $\Lambda^P\subset \C P^{n}(\sqrt{2})$.  (To recall this construction briefly, one identifies an open dense subset of $Q_{n+1}(\sqrt{2})$ or $\C P^{n}(\sqrt{2})$ with a symplectic disk bundle over $Q_n(\sqrt{2})$, and then sets $\Lambda^Q$ or $\Lambda^P$ equal to the restriction of a circle bundle of appropriate radius in this disk bundle to $\Lambda$.)
A natural choice for the submanifold $\Lambda\subset Q_{n}(\sqrt{2})$ is the following, where $k$ and $m$ are natural numbers with $k\leq m$ and $n=k+m+1$: \[ 
\mathbb{S}_{k,m} = \left\{[ix:y]\in \C P^{n}(\sqrt{2}) | x\in S^k\subset\R^{k+1},\,y\in S^m\subset \R^{m+1}\right\} \]  The following, proven in Propositions \ref{maincorr} and \ref{cpncorr}, is a generalization of the equivalences of $T_{AF}$ with $T_{BC}$ and of $T_{AF}^{P}$ with $T_{BC}^{P}$:

\begin{theorem}\label{gendimcorr}
Whenever $0\leq k\leq m$ and $m\geq 1$ we have \[ L_{k,m}^{Q} = (\mathbb{S}_{k,m})^{Q}\qquad\mbox{and}\qquad L_{k,m}^{P} = (\mathbb{S}_{k,m})^{P}\]
\end{theorem} 

In the course of proving Theorem \ref{gendimcorr} we give very explicit formulas for $L^{Q}_{k,m}$ and $L^{P}_{k,m}$   (similar to the formula for $L_{0,1}^{P}$ in Theorem \ref{maincp2}) as subsets of $Q_{k+m+2}(\sqrt{2})$  and $\C P^{k+m+1}(\sqrt{2})$; see Propositions \ref{lqkmprop} and \ref{cpnchar}.

In view of Theorems \ref{mains2s2} and \ref{maincp2} and \cite[Theorem 2]{CS}, $L_{0,1}^{Q}$ and  $L_{0,1}^{P}$ are both nondisplaceable.  Also, we show in Proposition \ref{pnd} that for any $k,m,r$ the submanifold $\mathcal{P}_{k,m}^{r}\subset T^{*}S^{k+m+1}$ is nondisplaceable, generalizing the main result of \cite{AF}; hence so too is its quotient $\underline{\mathcal{P}}_{k,m}^{r}\subset T^{*}\R P^{k+m+1}$.  Since $L^{Q}_{k,m}$ and $L^{P}_{k,m}$ are formed from $\mathcal{P}_{k,m}^{r}$ and  $\underline{\mathcal{P}}_{k,m}^{r}$ by performing symplectic cuts on their respective ambient manifolds, the following is somewhat surprising: 

\begin{theorem}\label{maindisp} $\,\,$\begin{itemize} \item[(i)] If $m \geq 2$ then $L_{0,m}^{Q}\subset Q_{m+2}(\sqrt{2})$ is displaceable.
\item[(ii)] If $k+m\geq 3$ then $L_{k,m}^{P}\subset\C P^{k+m+1}(\sqrt{2})$ is displaceable.\end{itemize}
\end{theorem}

These appear to be the first examples of displaceable \emph{monotone} Lagrangian submanifolds in any simply-connected compact symplectic manifold.  In some non-simply-connected ambient manifolds there are simple examples of displaceable monotone Lagrangians such as a small contractible circle on a $2$-torus; however if instead of requiring monotonicity one requires homological monotonicity (\emph{i.e.}, that the area and Maslov homomorphisms are proportional on $H_2(M,L)$ and not just on $\pi_2(M,L)$---if $M$ is simply-connected this is equivalent to monotonicity) then such trivial examples do not arise and we have not found any other examples in the literature.  Theorem \ref{maindisp} is proven below as Propositions \ref{qdisp} and \ref{highdisp}.

The requirement that $k+m\geq 3$ in Theorem \ref{maindisp} (ii) means that that result does not apply to the monotone Lagrangian submanifolds $L_{0,2}^{P}$ or $L_{1,1}^{P}$ in $\C P^{3}(\sqrt{2})$, which are diffeomorphic to $S^1\times S^2$ and $S^1\times S^1\times S^1$, respectively.  In fact, in Corollary \ref{l11} we prove by a Floer homology computation that $L_{1,1}^{P}$ is nondisplaceable.  It would be interesting to know whether $L_{0,2}^{P}$, or any of the $L_{k,m}^{Q}$ with $k\geq 1$, is displaceable.  In the case of $L_{0,2}^{P}$ we show in Corollary \ref{floervanish} that its Floer homology (with arbitrarily-twisted coefficients) is trivial.

The constructions of \cite{CS} yield four non-Hamiltonian-isotopic, nondisplaceable twist tori in $\C P^3(\sqrt{2})$ (including the Clifford torus).  Based on communication with F. Schlenk  concerning invariants of these tori, together with our own computations for $L_{1,1}^{P}$ in Propositions \ref{onlyclasses} and \ref{fiveclasses}, it appears that the torus $L_{1,1}^{P}$ is not Hamiltonian isotopic to any of the twist tori from \cite{CS}.

 Incidentally, just as in \cite{Wu} where the corresponding result is proven for the torus denoted $T_{W}\subset \C P^{2}(\sqrt{2})$ in Theorem \ref{maincp2}, the Floer homology computation in Corollary \ref{l11} implies that $\R P^3$ is not a stem in $\C P^3$ in the sense of \cite{EPstates} (since $L_{1,1}^{P}$ is disjoint from $\R P^3$).  

 \subsection{Outline of the paper and additional remarks} In Section \ref{sectionS2S2}, we give more detailed descriptions of the tori $T_{AF}$, $T_{FOOO}$, $T_{CS}$, and $T_{EP}$ from Theorem \ref{mains2s2} and establish their equivalences.  In Section \ref{sectionCP2}, we first establish a symplectic identification of the unit disk cotangent bundle $D_1^*\R P^n$ with a dense subset of $\C P^n(\sqrt{2})$, and we then describe the tori $T_{AF}^P$, $T_{CS}^P$, and $T_W$ from Theorem \ref{maincp2} and establish their equivalences.  

In Section \ref{sectionBiran}, we give explicit descriptions of embeddings of symplectic disk bundles \cite{B} in three special cases: the disk bundle over the diagonal $\Delta$ in $S^2\times S^2$; the disk bundle over the quadric $Q_n(\sqrt{2})$ in the quadric $Q_{n+1}(\sqrt{2})$; and the disk bundle over the quadric $Q_n(\sqrt{2})$ in $\C P^n(\sqrt{2})$.  In the case of the disk bundle over $\Delta\subset S^2\times S^2$, we observe that the Biran circle bundle construction \cite{Bir} over an equator in the diagonal is equal to the torus $T_{EP}$ (finishing the proof of Theorem \ref{mains2s2}).  Finally, in Proposition \ref{gendisp} we establish a general criterion, in terms of topological data, for the displaceability of monotone Lagrangian submanifolds obtained by the Biran circle bundle construction.

Section \ref{afintro} defines and discusses in detail the submanifolds $\mathcal{P}_{k,m}^{r}\subset T^*S^{k+m+1}$, which as mentioned earlier generalize the torus $\mathcal{P}_{0,1}^{r}$ considered in \cite{AF}; most of the section is devoted to proving Proposition \ref{pnd}, asserting that these submanifolds are monotone and nondisplaceable.  This nondisplaceability result is not used elsewhere in the paper, but provides an interesting contrast to Theorem \ref{maindisp}.

In Section \ref{quadsubs} we map the submanifolds $\mathcal{P}_{k,m}^{r}\subset T^{*}S^{k+m+1}$ into the quadric $Q_{k+m+2}(\sqrt{2})$ by means of a Darboux--Weinstein neighborhood of the sphere $\mathbb{S}_{0,k+m+1}$, whose image is the complement of the hyperplane section $Q_{k+m+1}(\sqrt{2})$.  We find a unique value of $r$ with the property that the image of $\mathcal{P}_{k,m}^{r}$ is monotone, and show that the resulting monotone Lagrangian submanifold $L_{k,m}^{Q}\subset Q_{k+m+2}(\sqrt{2})$ coincides with the Biran circle bundle construction over $\mathbb{S}_{k,m}\subset Q_{k+m+1}(\sqrt{2})$ using the disk bundle constructed earlier in Section \ref{sectionBiran}.  Moreover, using the very explicit nature of our symplectomorphisms, we give a concrete presentation of $L_{k,m}^{Q}$ as a submanifold of $Q_{k+m+2}(\sqrt{2})$, which allows us to see that $L_{0,m}^{Q}$ is displaceable for $m\geq 2$ by a direct construction.  (Alternatively, this displaceability could be proven using Proposition \ref{gendisp}.)

Section \ref{cpnsubs} carries out a similar procedure for the submanifolds $\underline{\mathcal{P}}_{k,m}^{r}\subset T^{*}\R P^{k+m+1}$ obtained from $\mathcal{P}_{k,m}^{r}\subset T^* S^{k+m+1}$ by quotient projection, mapping them into $\C P^{k+m+1}(\sqrt{2})$ by means of the Darboux--Weinstein neighborhood from Section \ref{sectionCP2}.  Once again this is shown to result in a monotone Lagrangian resulting submanifold $L_{k,m}^{P}$ for a unique value of $r$, and we show that $L_{k,m}^{P}$ coincides with the Biran circle bundle construction over $\mathbb{S}_{k,m}$.  In Proposition \ref{highdisp}, an explicit formula for $L_{k,m}^{P}$ allows us to show  that it is displaceable when $k+m\geq 3$.

Finally, in Section \ref{cp3sect} we consider the cases of the Lagrangian submanifolds $L_{0,2}^{P}$ and $L_{1,1}^{P}$ of $\C P^{3}(\sqrt{2})$, which fall just outside the reach of Proposition \ref{highdisp}.  In fact, by an approach similar to that used in \cite{Au},\cite{CS} to address a torus which by Theorem \ref{maincp2} is equivalent to $L_{0,1}^{P}$, we find that $L_{1,1}^{P}$ has nonvanishing Floer homology for an appropriately twisted coefficient system, and hence is nondisplaceable.  The Floer homology of $L_{0,2}^{P}$, on the other hand, vanishes, and so its displaceability or nondisplaceability remains an interesting open question.
 
We will end this introduction by providing some additional perspective on the submanifolds $L_{k,m}^{P}$ and $L_{k,m}^{Q}$ and on our approach to proving the various equivalences in Theorems \ref{mains2s2}, \ref{maincp2}, and  \ref{gendimcorr}.  In each case, an important ingredient is our construction of very explicit symplectomorphisms from dense neighborhoods of the Lagrangian submanifolds $\overline{\Delta}\subset S^2\times S^2$ (the antidiagonal), $\mathbb{S}_{0,k+m+1}\subset Q_{k+m+2}(\sqrt{2})$, or $\R P^{k+m+1}\subset\C P^{k+m+1}(\sqrt{2})$ to the appropriate disk cotangent bundles, and likewise of explicit symplectomorphisms from dense neighborhoods of $\Delta\subset S^2\times S^2$, $Q_{k+m+1}(\sqrt{2})\subset Q_{k+m+2}(\sqrt{2})$, or $Q_{k+m+1}(\sqrt{2})\subset\C P^{k+m+1}(\sqrt{2})$ to appropriate symplectic disk bundles.  An organizing principle behind our constructions of these symplectomorphisms is that they should respect the natural symmetry group of the pair $(M,\Sigma)$ where $M$ is the ambient manifold and $\Sigma$ is the relevant codimension-two symplectic submanifold: for both $(M,\Sigma)=(Q_{k+m+2}(\sqrt{2}),Q_{k+m+1}(\sqrt{2}))$ and $(M,\Sigma) = (\C P^{k+m+1}(\sqrt{2}),Q_{k+m+1}(\sqrt{2}))$ this symmetry group is the orthogonal group $O(k+m+2)$, and likewise for $(S^2\times S^2,\Delta)$ the symmetry group is $O(3)$, acting diagonally. 

 In each case the relevant orthogonal group acts in Hamiltonian fashion on $M$ (preserving $\Sigma$), on the appropriate disk cotangent bundle, and on the appropriate symplectic disk bundle, and our symplectomorphisms are designed to be equivariant with respect to these group actions.   A notable feature of these Hamiltonian $O(k+m+2)$-actions is that in all  cases the norms of their moment maps generate  Hamiltonian $S^1$-actions (except at the zero locus of the moment map, where the norm of the moment map fails to be differentiable) which commute with the $O(k+m+2)$-action---this is most easily seen when one works in the appropriate disk cotangent bundle, where the norm of the moment map is just the norm of the momentum, which has Hamiltonian flow given by the unit-speed geodesic flow.  Thus in each case we have an $S^1\times O(k+m+2)$-action on an open dense subset of $M$.  The Lagrangian submanifolds $L_{k,m}^{Q}$ and $L_{k,m}^{P}$ can each be characterized as particular orbits of the subgroup $S^1\times SO(k+1)\times SO(m+1)\leq S^1\times O(k+m+2)$ where $SO(k+1)\times SO(m+1)$ acts block-diagonally, and likewise the torus $T_{EP}\subset S^2\times S^2$ can be seen as an orbit of the torus $S^1\times SO(1)\times SO(2)\leq S^1\times SO(3)$; indeed the action of this torus can be seen as a concrete realization of the toric action on $(S^2\times S^2)\setminus\overline{\Delta}$ in \cite{FOOO}.  When one maps $L_{k,m}^{Q}$, $L_{k,m}^{P}$, or $T_{EP}$ to a cotangent bundle by our equivariant symplectomorphism, one finds that the $S^1$ factor of $S^1\times SO(k+1)\times SO(m+1)$ acts by the geodesic flow, while when one maps $L_{k,m}^{Q}$, $L_{k,m}^{P}$, or $T_{EP}$ to the appropriate sympelctic disk bundle one finds that the $S^1$ acts by rotation of the disk fibers, consistently with our submanifolds being obtained by the Biran circle bundle construction.
 
\subsection*{Acknowledgements}  We would like to thank Felix Schlenk for his interest in this work and for useful information about the twist tori, and Weiwei Wu for helpful conversations about $L_{1,1}^{P}$.  The first author was partially supported by an NSF VIGRE grant (DMS-0738586) and the second author by NSF Grant DMS-1105700.

\section{Lagrangian Submanifolds of $S^2\times S^2$}\label{sectionS2S2}

This section will establish the equivalences of all of the tori in Theorem \ref{mains2s2} except for $T_{BC}$.  We begin by recalling descriptions of the monotone tori $T_{AF}$, $T_{FOOO}$, $T_{CS}$, and $T_{EP}$, from \cite{AF}, \cite{FOOO},\cite{CS},\cite{EPrigid}, respectively, with some minor modifications mostly regarding our normalizations.  The simplest description is that given in \cite{EPrigid}, in which $S^2\times S^2$ is viewed as an embedded submanifold of $\R^3\times \R^3$ in the usual way: \[ S^2\times S^2 = \left\{(v,w)\in \R^3\times \R^3\left|\thinspace |v|=|w|=1\right.\right\}, \] and, where $e_1$ is the first standard basis vector in $\R^3$, the torus $T_{EP}$ is described explicitly as \[ T_{EP} = \left\{(v,w)\in S^2\times S^2\left|\thinspace (v+w)\cdot e_1 = 0, v\cdot w = -1/2 \right.\right\}. \]  Consistently with the conventions mentioned at the start of the introduction we will take the symplectic form on $S^2\times S^2$ to be $\Omega_{S^2\times S^2}=\left( \frac{1}{2}\omega_{\mbox{\scriptsize std}}\right)\oplus\left(\frac{1}{2}\omega_{\mbox{\scriptsize std}}\right)$, where $(S^2,\omega_{\mbox{\scriptsize std}})$ has area $4\pi$.  In particular, the sphere $S^2\times \{\mbox{point}\}$ has area $2\pi$ in our conventions.

In \cite{CS}, one begins with a curve $\Gamma$ enclosing an area of $\frac{\pi}{2}$ and contained in the open upper half disk $\mathbb{H}(\sqrt{2})=\{z\in \C\mid Im(z)>0,|z|<\sqrt{2}\}$.  The curve $\Delta_{\Gamma} = \{(z,z)\mid z\in \Gamma\}$ then lies in the diagonal of $B^2(\sqrt{2})\times B^2(\sqrt{2})$, the product of open disks of radius $\sqrt{2}$, and one then considers the torus $\Theta_{CS}$ in $B^2(\sqrt{2})\times B^2(\sqrt{2})$ given as the orbit of $\Delta_\Gamma$ under the circle action \[ e^{it}\cdot (z_1,z_2) = (e^{it}z_1,e^{-it}z_2). \]  More explicitly, we have \[ \Theta_{CS} = \{(e^{it}z,e^{-it}z)\mid z\in \Gamma,t\in[0,2\pi]\}. \]  Finally, one symplectically embeds $B^2(\sqrt{2})\times B^2(\sqrt{2})$ in $S^2\times S^2$ and defines the torus $T_{CS}$ to be the image of $\Theta_{CS}$ under such an embedding.

In \cite{FOOO}, one begins with a symplectic toric orbifold that is denoted $F_2(0)$ and whose moment polytope is \[ \Delta_{FOOO}=\left\{(x,y)\in \R^2\left|\thinspace 0\leq x \leq 2, 0 \leq y \leq 1-\frac{1}{2}x\right.\right\} \] with exactly one singular point sitting over the point $(0,1)\in \Delta_{FOOO}$.  Then, by replacing a neighborhood of the singular point with a neighborhood of the zero-section of the cotangent bundle $T^{*}S^2$, one obtains a manifold denoted $\hat{F}_2(0)$ that is shown to be symplectomorphic to $S^2\times S^2$.  The monotone Lagrangian torus $T_{FOOO}$ is then described as the image of the fiber over the point $(1/2,1/2)\in \Delta_{FOOO}$ under a symplectomorphism $\hat{F}_2(0)\to S^2\times S^2$.

Finally to construct $T_{AF}$, where $\Delta\subset S^2\times S^2$ is the diagonal, one begins with a symplectomorphism $\Phi_2\co \left( S^2\times S^2 \right)\setminus\Delta \to D_{1}^{*}S^2$ where $D_{1}^{*}S^2\subset T^*S^2$ is the open unit disk bundle (an explicit choice of $\Phi_2$ will be given below in Lemma \ref{cotprod}), and then puts \[ T_{AF}=\Phi_{2}^{-1}\left(\left\{(p,q)\in D_{1}^{*}S^2\left|\thinspace |p|=\frac{1}{2},\,(p\times q)\cdot e_1 = 0\right.\right\}\right). \] Here we use the standard Riemannian metric to identify $T^*S^2$  with $TS^2$, which can then be viewed a submanifold of $\R^3\times \R^3$ as follows \[ T^{*}S^2\cong TS^2 = \{ (p,q)\in \R^3\times \R^3\mid q\cdot p = 0,|q|=1 \}. \]  Under this identification, the canonical one form on $T^{*} S^2$ is $\lambda = p_1 dq_1 + p_2 dq_2 + p_3 dq_3$, and we consider $T^{*} S^2$ with symplectic form $d\lambda$.

\begin{prop}\label{FOOOtoEP}
$T_{AF}$ is equal to $T_{EP}$, and there is a symplectomorphism $S^2\times S^2 \to S^2 \times S^2$ taking $T_{FOOO}$ to $T_{EP}$.
\end{prop}

The focus of the proof of this proposition will be on the second statement; the fact that $T_{AF}=T_{EP}$ will be observed along the way.

While $T_{EP}$ is given very explicitly as a submanifold of $S^2\times S^2$, the same cannot be said of $T_{FOOO}$.  Rather, $T_{FOOO}$ is given as a submanifold of a symplectic manifold that is denoted $\hat{F}_2(0)$ in \cite{FOOO}, and this ambient manifold is then proven to be symplectomorphic to $S^2\times S^2$ in a way that makes it hard to extract what the image of $T_{FOOO}$ under the symplectomorphism might be.  Thus, most of our task will consist of giving a construction of the manifold $\hat{F}_2(0)$ which allows it to be symplectically identified with $S^2\times S^2$ in a very explicit way.  In fact, once our construction is finished it will follow almost immediately that $T_{FOOO}$ is mapped to $T_{EP}$ by our symplectomorphism.

\begin{lemma}\label{cotball} Where $B^4(2)$ is the open ball of radius $2$ in the quaternions $\mathcal{H}\cong \C^2\cong \R^4$, where $\R^3$ is identified with the pure imaginary quaternions, and where $0_{S^2}\subset T^{*}S^2$ is the zero-section, the map $\varphi_1\co B^4(2)\setminus\{0\}\to D_{1}^{*}S^2\setminus 0_{S^2}$ defined by \[ \varphi_1(\xi) = \left(-\frac{\xi^*k\xi}{4},\frac{\xi^*j\xi}{|\xi|^2}\right) \] is a symplectic double cover with $\varphi_1(\xi_1)=\varphi_1(\xi_2)$ if and only if $\xi_1=\pm \xi_2$.  Moreover, where $f_{T^*S^2}(p,q) = |p|$ and $g_{T^*S^2}(p,q) = (p\times q)\cdot e_1$, we have \[ 
f_{T^{*}S^2}\circ \varphi_1 (z_1+z_2j) = \frac{1}{4}(|z_1|^2+|z_2|^2)\quad \mbox{and}\quad g_{T^{*}S^2}\circ \varphi_1(z_1+z_2j) = \frac{1}{4}(|z_1|^2-|z_2|^2) \] for $z_1,z_2\in \C$ with $0<|z_1|^2+|z_2|^2<4$.
\end{lemma}

\begin{proof}
First, writing $\xi=z_1+z_2j$, we observe that
\begin{align*}
g_{T^{*}S^2}\circ \varphi_1 (z_1+z_2j) &= g_{T^{*}S^2}\circ \varphi_1 (\xi) = \left( -\frac{\xi^* k \xi}{4} \times \frac{\xi^* j \xi}{|\xi|^2} \right)\cdot e_1 =\left( \frac{1}{4} \xi^* i \xi\right)\cdot e_1 \\
& = \left(\frac{1}{4} \left( (|z_1|^2 - |z_2|^2)i - Re(2\bar{z}_1 z_2)j + Im(2\bar{z}_1 z_2)k \right) \right)\cdot e_1 \\
&= \frac{1}{4}(|z_1|^2-|z_2|^2)
\end{align*}
and also that \[ f_{T^{*}S^2}\circ \varphi_1 (z_1+z_2j) = f_{T^{*}S^2}\circ \varphi_1 (\xi) = \left| -\frac{\xi^* k \xi}{4}\right| = \frac{1}{4}|\xi|^2 = \frac{1}{4}(|z_1|^2+|z_2|^2), \] which proves the second statement of the lemma and also makes clear that $\varphi_1$ has an appropriate codomain.

We then observe that $\varphi_1(-\xi)=\varphi_1(\xi)$, and we claim also that $\varphi_1(\xi_1)=\varphi_1(\xi_2)$ only if $\xi_1=\pm \xi_2$.  Indeed if $\varphi_1(\xi_1)=\varphi_1(\xi_2)$, then it follows that $|\xi_1|=|\xi_2|$ and also that
\begin{align*}
\left(\frac{\xi_1}{|\xi_1|} \right)^* j \left(\frac{\xi_1}{|\xi_1|} \right) &= \left(\frac{\xi_2}{|\xi_2|} \right)^* j \left(\frac{\xi_2}{|\xi_2|} \right), \\
\left(\frac{\xi_1}{|\xi_1|} \right)^* k \left(\frac{\xi_1}{|\xi_1|} \right) &= \left(\frac{\xi_2}{|\xi_2|} \right)^* k \left(\frac{\xi_2}{|\xi_2|} \right), \\
\left(\frac{\xi_1}{|\xi_1|} \right)^* i \left(\frac{\xi_1}{|\xi_1|} \right) &= \left(\frac{\xi_2}{|\xi_2|} \right)^* i \left(\frac{\xi_2}{|\xi_2|} \right).
\end{align*}
Then, writing $\mathcal{S}$ for the group of unit quaternions, it is well known that the map $\mathcal{S} \to SO(3)$ given by $\xi \mapsto \begin{pmatrix} (\xi^* i \xi) & (\xi^* j \xi) & (\xi^* k \xi) \end{pmatrix}$ is a surjective Lie group homomorphism with kernel $\{\pm 1\}$; thus, it must be the case that $\xi_1/|\xi_1| = \pm \xi_2/|\xi_2|$ and hence that $\xi_1=\pm \xi_2$.  Moreover, the surjectivity of this Lie group homomorphism, when paired with the observation that $\left| -\frac{\xi^* k \xi}{4}\right| = \frac{1}{4}|\xi|^2$, implies that $\varphi_1$ is surjective.  A routine computation shows that \[ \varphi_1^* \lambda  = -\frac{y_1}{2}dx_1 + \frac{x_1}{2}dy_1 - \frac{y_2}{2}dx_2 + \frac{x_2}{2}dy_2, \] from which it follows that \[ \varphi_1^*(d\lambda) = d\left(\varphi_1^* \lambda\right) = d\left(-\frac{y_1}{2}dx_1 + \frac{x_1}{2}dy_1 - \frac{y_2}{2}dx_2 + \frac{x_2}{2}dy_2\right) = dx_1\wedge dy_1 + dx_2\wedge dy_2. \] Then, since any symplectic map is an immersion, it follows that $\varphi_1$ is a symplectic double cover as claimed.
\end{proof}

\begin{lemma}\label{cotprod}  Where $\Delta\subset S^2\times S^2$ is the diagonal, the map $\Phi_2\co (S^2\times S^2)\setminus\Delta\to D_{1}^{*}S^2$ defined by \[ \Phi_2(v,w) = \left(\frac{v\times w}{|v-w|},\frac{v-w}{|v-w|} \right)\] is a symplectomorphism.  Moreover, where $f_{T^*S^2}$ and $g_{T^*S^2}$ are as in Lemma \ref{cotball}, we have \[ f_{T^{*}S^2}\circ \Phi_2 (v,w) = \frac{1}{2}|v+w|\qquad \mbox{and}\qquad g_{T^{*}S^2}\circ \Phi_2(v,w) = \frac{1}{2}(v+w)\cdot e_1. \]
\end{lemma}

\begin{proof}
First, we observe that 
\begin{align*}
g_{T^{*}S^2}\circ \Phi_2(v,w) &= \left(\frac{v\times w}{|v-w|}\times \frac{v-w}{|v-w|} \right) \cdot e_1 = \left( \frac{(v\times w)\times v - (v\times w)\times w}{|v-w|^2} \right)\cdot e_1 \\
&= \left( \frac{-(v\cdot w) v  + w + v - (v\cdot w) w}{|v-w|^2} \right)\cdot e_1 \\
&= \left( \frac{(v+w)(1-v\cdot w)}{2-2v\cdot w} \right) \cdot e_1 = \frac{1}{2}(v+w)\cdot e_1,
\end{align*}
and the relationship
\begin{equation}\label{normseqn}
|v-w|^2|v+w|^2 = 4|v\times w|^2 \mbox{ for }(v,w)\in S^2\times S^2 
\end{equation}
makes clear that $f_{T^{*}S^2}\circ \Phi_2 (v,w) = \frac{1}{2}|v+w|$.  Thus, we have proved the second statement of the lemma (which also makes clear that $\Phi_2$ has an appropriate codomain).

To see that $\Phi_2$ is a symplectomorphism, we observe that the vector fields
\begin{align*}
X_1(v,w) &= (v\times w, w\times v) &X_2(v,w)= (v\times (v\times w), w\times (w\times v))  \\
X_3(v,w) &= (w\times v, w\times v) &X_4(v,w)= (v\times (w\times v), w\times (w\times v))  
\end{align*}
give a basis for $T_{(v,w)}\left((S^2\times S^2)\setminus \Delta\right)$ at each point $(v,w)$ not in the anti-diagonal $\overline{\Delta}=\{(v,w)\in S^2\times S^2\mid v=-w\}$.  We then compute that $\Omega_{S^2\times S^2}$ evaluates on pairs as follows:
\begin{align*}
\Omega_{S^2\times S^2}(X_1,X_2)&=\Omega_{S^2\times S^2}(X_3,X_4) = |v\times w|^2,\\
\Omega_{S^2\times S^2}(X_1,X_3)&=\Omega_{S^2\times S^2}(X_1,X_4)=\Omega_{S^2\times S^2}(X_2,X_3)=\Omega_{S^2\times S^2}(X_2,X_4)=0.
\end{align*}
Then, using the coordinate free formula for the exterior derivative of a one form, we will verify that $\Phi_2^* d\lambda$ evaluates on pairs in an identical manner to $\Omega_{S^2\times S^2}$.  To that end, computing the commutators of the vector fields $X_1$, $X_2$, $X_3$, and $X_4$, one finds the following relationships:
\begin{align*}
[X_1,X_2]&= \frac{1}{2}|v-w|^2 X_1 = -[X_3,X_4], \\
[X_1,X_3]&= -2 X_4, \\
[X_1,X_4]&= \frac{1}{2}|v+w|^2 X_3 = [X_2,X_3],\\
[X_2,X_4]&=(2v\cdot w) X_4.
\end{align*}
Moreover, another computation shows that $\Phi_2^*\lambda(X_1)=\frac{1}{2}|v+w|^2$ while $\Phi_2^* \lambda(X_j) = 0$ for $j\not=1$ (note that \eqref{normseqn} was used here to obtain the simplified form given for $\Phi_2^*\lambda (X_1)$), and yet another computation reveals that $X_j\left(\Phi_2^*\lambda(X_1)\right)=0$ for $j\not=2$ (since the quantity $\Phi_2^*\lambda(X_1)=\frac{1}{2}|v+w|^2$ is preserved under the flows of $X_1$, $X_3$, and $X_4$) while $X_2\left(\Phi_2^*\lambda(X_1)\right)=-2|v\times w|^2$.  It then follows from \eqref{normseqn} that
\begin{align*}
d \Phi_2^* \lambda (X_1,X_2) &= X_1\left(\Phi_2^*\lambda (X_2)\right) - X_2\left(\Phi_2^*\lambda (X_1)\right) - \Phi_2^* \lambda \left( [X_1,X_2] \right) \\ 
&= 2|v\times w|^2 - \Phi_2^* \lambda \left( \frac{1}{2}|v-w|^2 X_1 \right) \\
&= 2|v\times w|^2 - \frac{1}{4}|v-w|^2|v+w|^2 = |v\times w|^2,\\
d \Phi_2^* \lambda (X_1,X_3) &= X_1\left(\Phi_2^*\lambda (X_3)\right) - X_3\left(\Phi_2^*\lambda (X_1)\right) - \Phi_2^* \lambda \left( [X_1,X_3] \right) \\
&= - \Phi_2^* \lambda \left( -2 X_4 \right) = 0, \\
d \Phi_2^* \lambda (X_1,X_4) &= X_1\left(\Phi_2^*\lambda (X_4)\right) - X_4\left(\Phi_2^*\lambda (X_1)\right) - \Phi_2^* \lambda \left( [X_1,X_4] \right) \\
&= - \Phi_2^* \lambda \left( \frac{1}{2}|v+w|^2 X_3 \right) = 0, \\
d \Phi_2^* \lambda (X_2,X_3) &= X_2\left(\Phi_2^*\lambda (X_3)\right) - X_3\left(\Phi_2^*\lambda (X_2)\right) - \Phi_2^* \lambda \left( [X_2,X_3] \right) \\
&= - \Phi_2^* \lambda \left( \frac{1}{2}|v+w|^2 X_3 \right) = 0, \\
d \Phi_2^* \lambda (X_2,X_4) &= X_2\left(\Phi_2^*\lambda (X_4)\right) - X_4\left(\Phi_2^*\lambda (X_2)\right) - \Phi_2^* \lambda \left( [X_2,X_4] \right) \\
&= - \Phi_2^* \lambda \left( (2v\cdot w) X_4 \right) = 0, \\
d \Phi_2^* \lambda (X_3,X_4) &= X_3\left(\Phi_2^*\lambda (X_4)\right) - X_4\left(\Phi_2^*\lambda (X_3)\right) - \Phi_2^* \lambda \left( [X_3,X_4] \right) \\
&= -\Phi_2^* \lambda \left(-\frac{1}{2}|v-w|^2 X_1\right) = \frac{1}{4}|v-w|^2|v+w|^2=|v\times w|^2,
\end{align*}
and then (by continuity along the anti-diagonal $\overline{\Delta}$ where the vector fields $X_j$ vanish) we see that $\Phi_2^* d\lambda = \Omega_{S^2\times S^2}$ as required.  Finally, to see that $\Phi_2$ is bijective, a routine check (using \eqref{normseqn} and the fact that $4-|v+w|^2=|v-w|^2$ for $(v,w)\in S^2\times S^2$) reveals that \[ \Phi_2^{-1}(p,q)=\left(\sqrt{1-|p|^2}\thinspace q-q\times p,-\sqrt{1-|p|^2}\thinspace q-q\times p\right). \]
\end{proof}

With Lemmas \ref{cotball} and \ref{cotprod} proved, we are now ready to give a construction of the manifold $\hat{F}_2(0)$ and prove Proposition \ref{FOOOtoEP}.

\begin{proof}[Proof of Proposition \ref{FOOOtoEP}]

First, the fact that $T_{AF}=T_{EP}$ follows immediately from the definitions and from the computations of $f_{T^*S^2}\circ\Phi_2$ and $g_{T^*S^2}\circ \Phi_2$ in Lemma \ref{cotprod}, since for $(v,w)\in S^2\times S^2$ we have $|v+w| = \sqrt{2+2v\cdot w}$.

Since the preimage of the zero-section $0_{S^2}$ under the map $\Phi_2$ is the anti-diagonal $\overline{\Delta}\subset S^2\times S^2$, it follows from Lemmas \ref{cotball} and \ref{cotprod} that the map $\Phi_{2}^{-1}\circ \varphi_1\co B^{4}(2)\setminus\{0\}\to (S^2\times S^2)\setminus(\overline{\Delta}\cup\Delta)$ descends to a symplectomorphism \[ A\co \frac{B^{4}(2)\setminus\{0\}}{\pm 1}\to (S^2\times S^2)\setminus(\overline{\Delta}\cup\Delta) \] which pulls back the function $(v,w)\mapsto \frac{1}{2}|v+w|$ to the function $[(z_1,z_2)]\mapsto 
\frac{1}{4}(|z_1|^2+|z_2|^2)$ and pulls back the function $(v,w)\mapsto \frac{1}{2}(v+w)\cdot e_1$ to the function $[(z_1,z_2)]\mapsto 
\frac{1}{4}(|z_1|^2-|z_2|^2)$.

Consequently we may introduce the symplectic $4$-orbifold \[ \mathcal{O} = \frac{(B^4(2)/\pm 1)\coprod \left( (S^2\times S^2)\setminus \overline{\Delta}\right)}{[(z_1,z_2)]\sim A([(z_1,z_2)])\mbox{ for }(z_1,z_2)\neq (0,0)} \]  since the fact that $A$ is a symplectomorphism shows that the the symplectic forms on $B^4(2)$ and on $(S^2\times S^2)\setminus \overline{\Delta}$ coincide on their overlap in $\mathcal{O}$.
Moreover we have well-defined functions $F\co \mathcal{O}\to\R$ and $G\co \mathcal{O}\to\R$ defined by \begin{align*} F([(z_1,z_2)])&=\frac{1}{4}(|z_1|^2+|z_2|^2) & G([(z_1,z_2)]) = \frac{1}{4}(|z_1|^2-|z_2|^2)  \\
F(v,w) &= \frac{1}{2}|v+w| & G(v,w) = \frac{1}{2}(v+w)\cdot e_1  \end{align*}
for $(z_1,z_2)\in B^4(2)$ and $(v,w)\in (S^2\times S^2)\setminus\overline{\Delta}$

One easily verifies that the map $(F+G,1-F)\co \mathcal{O}\to \R^2$ is a moment map for a symplectic toric action on the symplectic orbifold $\mathcal{O}$, with image equal to the polytope $\Delta_{FOOO}$.  The classification of toric orbifolds from \cite{LT} therefore implies that $\mathcal{O}$ is equivariantly symplectomorphic to the orbifold $F_{2}(0)$ from \cite[Section 3]{FOOO} (as $\mathcal{O}$ and $F_2(0)$ have identical moment polytopes and both have only one singular point, located at the preimage of $(0,1)$ under the moment map); accordingly we hereinafter implicitly identify $F_2(0)$ with $\mathcal{O}$.  The manifold $\hat{F}_{2}(0)$ from \cite{FOOO} is then constructed by removing a neighborhood $\mathcal{U}$ of the unique singular point $[(0,0)]$ of $\mathcal{O}$ and gluing in its place a neighborhood $\mathcal{N}$ of $0_{S^2}$ in the cotangent bundle $T^*S^2$, using a symplectomorphism between $\mathcal{U}\setminus\{[(0,0)]\}$ and $\mathcal{N}\setminus 0_{S^2}$.  While a particular choice of this symplectomorphism is not specified in \cite{FOOO}, we have already constructed one that will serve the purpose, namely the map $\Phi_1\co (B^4(2)\setminus\{0\})/\pm 1\to D_{1}^{*}S^2\setminus 0_{S^2}$ induced on the quotient by the map $\varphi_1$ from Lemma \ref{cotball}.  This gives a symplectomorphism between the manifold $\hat{F}_{2}(0)$ from \cite{FOOO} and the manifold \[ \frac{D_{1}^{*}S^2\coprod \left( (S^2\times S^2)\setminus \overline{\Delta}\right)}{(p,q)\sim \Phi_{2}^{-1}(p,q)\mbox{ for }(p,q)\in D_{1}^{*}S^2\setminus 0_{S^2}}. \]  But of course the map $\Phi_{2}^{-1}$ then induces a symplectomorphism between this latter manifold and $S^2\times S^2$.

There is an obvious continuous map $\Pi\co \hat{F}_{2}(0)\to F_{2}(0)$ which maps the zero-section $0_{S^2}$ to the singular point $[(0,0)]$ and coincides with $\Phi_{1}^{-1}$ on $D_{1}^{*}S^2\setminus 0_{S^2}\subset \hat{F}_{2}(0)$ and with the identity on $(S^2\times S^2)\setminus \overline{\Delta}\subset \hat{F}_{2}(0)$; the monotone Lagrangian torus $T_{FOOO}$ is the preimage of the point $(1/2,1/2)$ under the pulled-back moment map $((F+G)\circ \Pi,(1-F)\circ\Pi)\co \hat{F}_2(0)\to \R^2$.  In view of the expressions for the functions $F,G$ on $(S^2\times S^2)\setminus \overline{\Delta}$, it follows that $T_{FOOO}$ is taken by our symplectomorphism $\hat{F}_2(0)\to S^2\times S^2$ to \[ \left\{(v,w)\in S^2\times S^2\left|\thinspace \frac{1}{2}|v+w|+\frac{1}{2}(v+w)\cdot e_1 = \frac{1}{2},\thinspace 1-\frac{1}{2}|v+w|=\frac{1}{2} \right.\right\},\]
which is obviously equal to the Entov-Polterovich torus \[ T_{EP}=\{(v,w)\in S^2\times S^2 \mid (v+w)\cdot e_1=0,\, v\cdot w = -1/2\}. \]
\end{proof}

\begin{prop}\label{CStoEP}
There is a symplectomorphism $S^2\times S^2 \to S^2 \times S^2$ taking $T_{CS}$ to $T_{EP}$.
\end{prop}

\begin{proof}
We recall that $T_{CS}$ is defined as $\psi\times \psi(\Theta_{CS})$, where 
\begin{align*}
\psi\co\left(B^2(\sqrt{2}),\omega_{\C^4}\right)&\to \left(S^2\setminus\{-e_1\},\frac{1}{2}\omega_{\mbox{\scriptsize std}}\right) \\
re^{i\theta}&\mapsto \left( 1-r^2 , r\cos\theta\sqrt{2-r^2} , r\sin\theta\sqrt{2-r^2} \right)^\top
\end{align*}
is a symplectomorphism (shown by a standard computation) and \[ \Theta_{CS} = \{(e^{it}z,e^{-it}z)\mid z\in \Gamma,t\in[0,2\pi]\}\subset B^2(\sqrt{2})\times B^2(\sqrt{2}) \] for a curve $\Gamma\subset\mathbb{H}(\sqrt{2})$ enclosing area $\frac{\pi}{2}$ (the Hamiltonian isotopy class of $T_{CS}$ is easily seen to be independent of the particular choice of $\Gamma$).  Alternatively, $\Theta_{CS}$ is given as the orbit of the curve $\Delta_\Gamma=\{(z,z)\mid z\in \Gamma\}$ under the circle action \[ e^{it}\cdot (z_1,z_2) = (e^{it}z_1,e^{-it}z_2).\]  Another simple computation shows that \[ \psi\left(e^{it} r e^{i\theta}\right) =  R_t\thinspace\psi(re^{i\theta}), \mbox{ where } R_t = \begin{pmatrix} 1 & 0 & 0 \\ 0 & \cos t & -\sin t \\ 0 & \sin t & \cos t \end{pmatrix} \in SO(3), \] from which it follows that \[ T_{CS} = \psi\times \psi(\Theta_{CS}) = \{ \left( R_t\thinspace\psi(z), R_{-t}\thinspace \psi(z) \right)\mid t\in [0,2\pi ],z\in \Gamma\}. \]  In other words, $T_{CS}$ is the orbit of the curve $\psi\times \psi\left(\Delta_{\Gamma}\right)$ under the following circle action, denoted $\rho_{CS}$, on $S^2\times S^2$: \[ \rho_{CS}(e^{it})\cdot(v,w) = \left( R_t\thinspace v, R_{-t}\thinspace w\right). \] 

On the other hand, if we consider the smooth embedded curve $C\subset S^2\times S^2$ parametrized by
\begin{align*}
[0,2\pi] & \to S^2\times S^2 \\
s &\mapsto \left( \left( -\frac{\sqrt{3}}{2}\sin(s) , -\frac{\sqrt{3}}{2}\cos(s) , \frac{1}{2}  \right)^\top , \left( \frac{\sqrt{3}}{2}\sin(s) , \frac{\sqrt{3}}{2}\cos(s) , \frac{1}{2}  \right)^\top \right),
\end{align*}
then we claim that the torus $T_{EP}$ is the orbit of $C$ under the following action of the circle on $S^2\times S^2$: \[ \rho_{EP}(e^{it})\cdot (v,w) = \left( R_t\thinspace v, R_t\thinspace  w\right). \]  Indeed, we note that $T_{EP}$ is the regular level set $(F_1,F_2)^{-1}\left(0,-\frac{1}{2}\right)$ for the $\R^2$-valued function $(F_1,F_2)\co(v,w) \mapsto \left( -(v+w)\cdot e_1, v\cdot w \right)$.  The Hamiltonian vector fields associated to $F_1$ and $F_2$ are $X_{F_1}(v,w) = \left( e_1\times v, e_1 \times w \right)$ and $X_{F_2}(v,w) = \left( v\times w, w\times v\right)$, respectively.  We then observe that the curve $C$ is the orbit of the point $\left( \left( 0,-\frac{\sqrt{3}}{2},\frac{1}{2}\right)^\top , \left( 0,\frac{\sqrt{3}}{2},\frac{1}{2}\right)^\top \right)\in T_{EP}$ under the Hamiltonian flow for $F_2$, and thus the torus $T_{EP}$ is exactly the orbit of $C$ under the Hamiltonian flow for $F_1$.  Noting that the Hamiltonian flow for $F_1$ gives the circle action $\rho_{EP}$, we see that $T_{EP}$ is the orbit of $C$ under the action $\rho_{EP}$ as claimed.

Next, we use the observation of Gadbled in \cite{Ga} that the actions $\rho_{EP}$ and $\rho_{CS}$ are conjugate in $SO(3)\times SO(3)$.  Indeed a simple computation shows that \[ \left(R_t, R_t\right) = \left(\mathcal{Q}_1, \mathcal{Q}_2\right)^{-1} \left(R_t, R_{-t}\right) \left(\mathcal{Q}_1, \mathcal{Q}_2\right) \] for $\mathcal{Q}_1$ the identity and  $\mathcal{Q}_2 = \left( \begin{smallmatrix} -1 & 0 & 0 \\ 0 & -1 & 0 \\ 0 & 0 & 1 \end{smallmatrix}\right)$.  Hence, it follows that 
\begin{equation}\label{gadbledeqn}
(\mathcal{Q}_1\, \mathcal{Q}_2)\left(\rho_{EP}(e^{it})\cdot(v,w) \right) = \rho_{CS}(e^{it})\cdot \left( (\mathcal{Q}_1,\mathcal{Q}_2)(v,w) \right),
\end{equation}
and we define $T_{EP}^\prime$ to be the orbit of the curve $(\mathcal{Q}_1, \mathcal{Q}_2)(C)$ under the action of $\rho_{CS}$.  Where $\Gamma^\prime\subset S^2$ is the curve parametrized by $s\mapsto \left(-\frac{\sqrt{3}}{2}\sin(s), -\frac{\sqrt{3}}{2}\cos(s), \frac{1}{2} \right)^\top $, we observe that $(\mathcal{Q}_1, \mathcal{Q}_2)(C)$ is the curve $\Delta_{\Gamma^\prime}=\{(v,v)\in S^2\times S^2\mid v\in \Gamma^\prime\}$ in the diagonal of $S^2\times S^2$.  Where $D^2\subset \C$ is the closed disk of radius $1$, we observe that the disk $D^\prime\subset S^2\setminus\{-e_1\}$ parametrized by 
\begin{align*}
g\co D^2&\to S^2\setminus\{-e_1\} \\
re^{i\phi}&\mapsto \left( \frac{\sqrt{3}}{2}r\sin(\phi), -\frac{\sqrt{3}}{2}r\cos(\phi), \sqrt{1-\frac{3}{4}r^2} \right)^\top 
\end{align*}
has boundary $\Gamma^\prime$, and a routine computation shows that \[ g^* \left(\frac{1}{2}\omega_{\mbox{\scriptsize std}}\right) = \frac{3r}{4\sqrt{4-3r^2}}\thinspace dr\wedge d\phi. \]  Thus, $D^\prime$ has area \[ \int_{D^\prime} \frac{1}{2}\omega_{\mbox{\scriptsize std}} = \int_{D^2} g^* \left( \frac{1}{2}\omega_{\mbox{\scriptsize std}} \right) = \int_{0}^{2\pi} \int_{0}^{1} \frac{3r}{4\sqrt{4-3r^2}}\thinspace dr\thinspace d\phi = \int_{0}^{2\pi} \frac{1}{4} \thinspace d\phi = \frac{\pi}{2}, \] which means that $\Gamma^\prime$ encloses a domain of area $\frac{\pi}{2}$ in $S^2\setminus\{-e_1\}$.  It then follows that the curve $\psi^{-1}(\Gamma^\prime)$ encloses an area of $\frac{\pi}{2}$ since $\psi$ is a symplectomorphism, and it is not difficult to see that $\psi^{-1}(\Gamma^\prime)$ also lies in $\mathbb{H}(\sqrt{2})$ since $\psi$ maps $\mathbb{H}(\sqrt{2})$ to the hemisphere $\{v\in S^2\mid v_3>0\}$.

Finally, taking the curve $\Gamma$ in Chekanov and Schlenk's construction to be the curve $\psi^{-1}(\Gamma^\prime)$, the corresponding torus $T_{CS}\subset S^2\times S^2$ is exactly the orbit of the curve \[ \psi\times \psi\left(\Delta_\Gamma\right) = \Delta_{\Gamma^\prime} = (\mathcal{Q}_1,\mathcal{Q}_2)(C) \] under the action of $\rho_{CS}$; in other words, $T_{CS}=T_{EP}^\prime$.  Now, by \eqref{gadbledeqn} and the fact that $T_{EP}$ is the orbit of $C$ under the action $\rho_{EP}$, it is clear that $T_{EP}^\prime$ is nothing more than the image of $T_{EP}$ under the map $(\mathcal{Q}_1,\mathcal{Q}_2)$, and thus $T_{CS} = T_{EP}^\prime = \left(\mathcal{Q}_1, \mathcal{Q}_2\right)\left(T_{EP}\right)$.  Since $(\mathcal{Q}_1,\mathcal{Q}_2)$ is a symplectomorphism (in fact a Hamiltonian diffeomorphism), the desired result has been obtained. 
\end{proof}

\section{Lagrangian Submanifolds of $\C P^2$}\label{sectionCP2}

\subsection{A dense Darboux-Weinstein neighborhood of $\R P^n\subset \C P^n$}

Just as the symplectomorphism $\Phi_2\co (S^2\times S^2)\setminus\Delta\to D_{1}^{*}S^n$ played a prominent role in the proof of Proposition \ref{FOOOtoEP}, in comparing the various Lagrangian tori in Theorem \ref{maincp2} it will be crucial to have a symplectic identification of the disk bundle $D_{1}^{*}\R P^2$ with a dense subset of $\C P^2(\sqrt{2})$.  In fact, it will be no more difficult to construct a version of this for $D_{1}^{*}\R P^{n}$ for arbitrary $n$, and this will be useful later in the construction of the manifolds $L_{k,m}^{P}\subset \C P^{k+m+1}(\sqrt{2})$.

We view the symplectic manifold $T^*\R P^n$ as the quotient of $T^*S^n$ by the antipodal involution $(p,q)\mapsto (-p,-q)$.  
For $r>0$, the radius-$r$ disk bundle $D_{r}^{*}\R P^n$ then consists of pairs $[(p,q)]$ with $|p|<r$.

Define a function $f\co [0,1)\to \R$ by $f(0)=\frac{1}{2}$ and, for $0<x<1$, \begin{equation}\label{fdef} f(x)=\frac{1-\sqrt{1-x^2}}{x^2}. \end{equation} It is easy to check that $f$ is $C^{\infty}$ on $[0,1)$, and also that $f$ satisfies the identity \begin{equation}\label{fidentity} x^2f(x)+\frac{1}{f(x)} = 2    \end{equation} for all $x\in [0,1)$.   

\begin{lemma}\label{rpncpn} Where $\C P^{n}(\sqrt{2})$ denotes the coisotropic reduction of the sphere of radius $\sqrt{2}$ in $\C^{n+1}$ and where $f$ is the function from (\ref{fdef}), the map \begin{align*} \Psi^P\co D_{1}^{*}\R P^{n} &\to \C P^{n}(\sqrt{2}) \\ [(p,q)]&\mapsto \left[\sqrt{f(|p|)}p+\frac{i}{\sqrt{f(|p|)}}q \right] \end{align*} is a symplectomorphism to its image, which is equal to the complement of the quadric $Q_{n}(\sqrt{2}) = \left\{[z_0:\cdots:z_n]\in \C P^{n}(\sqrt{2})\left|\sum z_{j}^{2}=0\right.\right.\}$.
\end{lemma}

\begin{proof}
We must first show that $\Psi^P$ is well-defined, which evidently will be true provided that, for any $(p,q)\in D_{1}^{*}S^n$, the element  
$\sqrt{f(|p|)}p+\frac{i}{\sqrt{f(|p|)}}$ of $\C^{n+1}$ lies on the sphere of radius $\sqrt{2}$ (and hence has a well-defined projection to $\C P^{n}(\sqrt{2})$; clearly this projection would then be invariant under the antipodal involution of $D_{1}^{*}S^n$).  For $j=0,\ldots,n$ define $z_j\co D_{1}^{*}S^n  \to \C$ by \[ z_j(p,q) = \sqrt{f(|p|)}p_j+\frac{iq_j}{\sqrt{f(|p|)}} \]  We then have, for $(p,q)\in D_{1}^{*}S^n$ (so that $|q|=1$, $|p|<1$, $p\cdot q=0$) \[\sum_j |z_j(p,q)|^2 = f(|p|)\sum p_{j}^{2}+\frac{1}{f(|p|)}\sum q_{j}^{2} = |p|^2f(|p|)+\frac{1}{f(|p|)} =2\] where the last equality follows from (\ref{fidentity}).  So the map $\tilde{\Psi}^P\co D_{1}^{*}S^n\to \C^{n+1}$ defined by \[ \tilde{\Psi}^{P}(p,q) = \sqrt{f(|p|)}p+\frac{i}{\sqrt{f(|p|)}} \] indeed takes values in the sphere of radius $\sqrt{2}$, and so $\Psi^P$ is well-defined as a map to $\C P^{n}(\sqrt{2})$.

Next we claim that $\Psi^P$ has image contained in $\C P^{n}(\sqrt{2})\setminus Q_{n}(\sqrt{2})$.  Indeed for $(p,q)\in D_{1}^{*}S^n$ we see that \begin{align} \sum z_{j}(p,q)^2 &= f(|p|)\sum p_{j}^{2}-\frac{1}{f(|p|)}\sum q_{j}^{2} +2i\sum p_jq_j \nonumber
\\&=|p|^2f(|p|) - \frac{1}{f(|p|)} = 2|p|^2f(|p|)-2 = -2\sqrt{1-|p|^2} < 0 \label{sumz2} \end{align} (since $|p|<1$), proving the claim.

Suppose that $(p,q),(p',q')\in D_{1}^{*}S^n$ have the property that $\Psi^P([(p,q)]) = \Psi^P([(p',q')])$; we will show that $(p',q') = \pm(p,q)$.  Write \[ (z_0,\ldots,z_n)=\tilde{\Psi}^{P}(p,q) \qquad (z'_0,\ldots,z'_n) = \tilde{\Psi}^{P}(p',q') \]  Then for some $e^{i\theta}\in S^1$ we have $(z_0,\ldots,z_n) = e^{i\theta}(z'_0,\ldots,z'_n)$, but the calculation in the previous paragraph shows that $\sum z_{j}^{2}$ and $\sum (z'_{j})^{2}$ are both negative numbers, which forces $e^{2i\theta}$ to be $1$ and hence $e^{i\theta}=\pm 1$; moreover since $|\sum z_{j}^{2}|=|\sum (z'_{j})^{2}|$ we will have $|p|=|p'|$.  So the assumption that $\Psi^P([(p,q)])=\Psi^P([(p',q')])$  implies that \[ f(|p|)p+\frac{i}{f(|p|)}q = \pm\left( f(|p|)p'+\frac{i}{f(|p|)}q'\right),\] which indeed forces $(p,q)=\pm(p',q')$ and hence  $[(p,q)]=[(p',q')]$.  Thus our map $\Psi^{P}\co D_{1}^{*}\R P^n\to \C P^{n}(\sqrt{2})\setminus Q_n(\sqrt{2})$ is injective.

We now show that the image of $\Psi^{P}$ is all of $\C P^{n}(\sqrt{2})\setminus Q_n(\sqrt{2})$.  Any element of  $\C P^{n}(\sqrt{2})\setminus Q_n(\sqrt{2})$ can be represented in the form $[u+iv]$ where $u,v\in \R^{n+1}$ have $|u|^2+|v|^2=2$ and $\sum (u_j+iv_j)^2<0$.  The latter condition amounts to the statements that $|u|<|v|$ and that $u\cdot v=0$.  
Now from the definition of $f$ one finds that \begin{align*} f(|u||v|)|u|^2|v|^2 &= 1-\sqrt{1-|u|^2|v|^2} = 1-\sqrt{1-|u|^2(2-|u|^2)}
\\&=1-\sqrt{(1-|u|^2)^2} = |u|^2\end{align*} and hence that $ \sqrt{f(|u||v|)}|v| = 1$.  Thus we have \[ \Psi^P\left(\left[|v|u,\frac{v}{|v|}\right]\right) = \left[\sqrt{f(|u||v|)}|v|u+\frac{i}{\sqrt{f(|u||v|)}}\frac{v}{|v|}\right] = [u+iv] \] (The fact that $|v||u|<1$ follows from the facts that $0<(|v|-|u|)^2$ and $|v|^2+|u|^2=2$.)  So indeed all points of $\C P^{n}(\sqrt{2})\setminus Q_n(\sqrt{2})$ lie in the image of $\Psi^P$.

Finally we show that $\Psi^P$ is a symplectomorphism to its image.  Given what we have already done, this will follow once we show that where $\alpha=\sum_{i=0}^{n}x_jdy_j\in \Omega^1(\C^{n+1})$ and $\lambda\in \Omega^1(D_{1}^{*}S^n)$ is the canonical one-form, we have $(\tilde{\Psi}^{P})^{*}\alpha=\lambda$.  Now for $(p,q)\in D_{1}^{*}S^n$, \begin{align*}\left((\tilde{\Psi}^{P})^{*}\alpha\right)_{(p,q)} &= \sum_j\sqrt{f(|p|)}p_jd\left(\frac{q_j}{\sqrt{f(|p|)}}\right) \\&= \sum_jp_jdq_j +\left(\sum_j p_jq_j\right)\sqrt{f(|p|)}d\left(\frac{1}{\sqrt{f(|p|)}}\right) = \sum_j p_jdq_j \end{align*} since the fact that $(p,q)\in D_{1}^{*}S^n$ implies that $p\cdot q=0$.  So indeed $\tilde{\Psi}^{P}$ pulls back $\alpha$ to the canonical one-form $\lambda$, in view of which $\Psi^P\co D_{1}^{*}\R P^n\to \C P^{n}(\sqrt{2})\setminus Q_n(\sqrt{2})$ pulls back the Fubini-Study form to the standard symplectic form on $D_{1}^{*}\R P^n$.  In particular this implies that $\Psi^P$ is an immersion (and so also a submersion by dimensional considerations).  So since $\tilde{\Psi}^P$ maps $D_{1}^{*}\R P^n$ bijectively to the open subset $\C P^{n}(\sqrt{2})\setminus Q_n(\sqrt{2})$ of $\C P^{n}(\sqrt{2})$ it follows that $\Psi^P$ is a symplectomorphism to this subset.\end{proof}

\begin{remark}\label{pullbacks}
It follows from (\ref{sumz2}) that $\Psi^P$ pulls back the function $H\co [z_0:\cdots:z_n]\mapsto \frac{1}{4}\sqrt{4-|\sum z_{j}^{2}|^2}$ on $\C P^{n}(\sqrt{2})$ to 
$(p,q)\mapsto \frac{1}{4}\sqrt{4-4(1-|p|^2)} = \frac{|p|}{2}$.  In particular since $(p,q)\mapsto \frac{|p|}{2}$ generates a Hamiltonian $S^1$-action on $D_{1}^{*}\R P^{n}\setminus 0_{\R P^n}$ (given by the geodesic flow) it follows that $H$ generates a Hamiltonian $S^1$-action on $\C P^{n}\setminus \R P^n$.  (The function $H$ fails to be smooth along $\R P^n$.)   Meanwhile for $0\leq i<j\leq n$ the Hamiltonian $G_{ij}\co [z_0:\cdots:z_n]\mapsto Im(\bar{z}_iz_j)$ pulls back by $\Psi^P$ to $[(p,q)]\mapsto p_{i+1}q_{j+1}-p_{j+1}q_{i+1}$.  It is easy to see that $H$ Poisson-commutes with each of the functions $G_{ij}$.

When $n=2$, we thus have a Hamiltonian torus action on $\C P^{n}(\sqrt{2})\setminus \R P^n$ with moment map given by $(H,G_{12})$.
\end{remark}

\subsection{Equivalences in $\mathbb{C}P^2$}\label{cp2eq}

The torus $T_{AF}^{P}$ from Theorem \ref{maincp2} is, by definition, the image under $\Psi^P\co D_{1}^{*}\R P^{2}\to\C P^{2}(\sqrt{2})\setminus Q_2(\sqrt{2})$ of the torus \[ \underline{\mathcal{P}}_{0,1}^{1/3} = \left\{[(p,q)]\in T^{*}\R P^{2} \left|\thinspace |p| = 1/3,\, (p\times q)\cdot e_1=0\right.\right\} \]  By Remark \ref{pullbacks}, writing $\Psi^P([(p,q)]) = [z_0:z_1:z_2]$, we have $(p\times q)\cdot e_1=0$ (\emph{i.e.}, $p_2q_3-p_3q_2=0$) if and only if $Im(\bar{z}_1z_2)=0$, and $|p|=\frac{1}{3}$ if and only if $\sqrt{4-|\sum z_{j}^{2}|^2}=\frac{2}{3}$, \emph{i.e.} if and only if $|\sum z_{j}^{2}|=\frac{4\sqrt{2}}{3}$.  This proves that $T_{AF}^{P}$ is equal to the torus $L_{0,1}^{P}$ that is defined in Theorem \ref{maincp2}.

We now recall the constructions of the tori $T_{CS}^{P}$ and $T_{W}$ as presented in \cite{CS} and \cite{Wu} with some minor modifications pertaining to normalization.  In \cite{CS}, one begins with a curve $\Gamma_P$ enclosing an area of $\frac{\pi}{3}$ and contained in the open upper half disk $\mathbb{H}(1)=\{z\in \C\mid Im(z)>0,|z|<1\}$.  The curve $\Delta_{\Gamma_P} = \{(z,z)\mid z\in \Gamma_P\}$ then lies in the diagonal of $B^2(1)\times B^2(1)$, which is itself contained in the 4-dimensional Euclidean ball $B^4(\sqrt{2})$; then one considers the torus $\Theta_{CS}^P$ in $B^4(\sqrt{2})$ given as the orbit of $\Delta_{\Gamma_P}$ under the circle action \[ e^{it}\cdot (z_1,z_2) = (e^{it}z_1,e^{-it}z_2). \]  More explicitly, we have \[ \Theta_{CS}^P = \{(e^{it}z,e^{-it}z)\mid z\in \Gamma_P,t\in[0,2\pi]\}. \]  Finally, one symplectically embeds $B^4(\sqrt{2})$ in $\C P^2(\sqrt{2})$ and defines the torus $T_{CS}^P$ to be the image of $\Theta_{CS}^P$ under such an embedding.

In \cite{Wu}, one begins with a symplectic toric orbifold that is denoted $F_4(0)$ and whose moment polytope is \[ \Delta_{W}=\left\{(x,y)\in \R^2\left| \thinspace 0\leq x \leq 2, 0 \leq y \leq \frac{1}{2}-\frac{1}{4}x\right.\right\} \] with exactly one singular point sitting over the point $(0,1/2)\in \Delta_{W}$.  Then, by replacing a neighborhood of the singular point with a neighborhood of the zero-section of the cotangent bundle $T^{*}\R P^2$, one obtains a manifold denoted $\hat{F}_4(0)$ that is symplectomorphic to $\C P^2(\sqrt{2})$.  The monotone Lagrangian torus $T_{W}$ is then described as the image of the fiber over the point $(1/3,1/3)\in \Delta_{W}$ under a symplectomorphism $\hat{F}_4(0)\to \C P^2(\sqrt{2})$.

%Finally, we give another description of a monotone Lagrangian torus in $\C P^2(\sqrt{2})$, which is given quite explicitly as \[ L_{0,1}^P = \left\{ \left[\sqrt{2-|w|^2}:wy\right] \left|\thinspace w\in \mathbb{H}(\sqrt{2}), \left| w^2 + 2 - |w|^2 \right|^2=\frac{32}{9}, y\in S^1 \right.\right\}. \]  That $L_{0,1}^P$ is a monotone Lagrangian torus follows from Propositions ?? and ?? (combined with the observation that the discussion preceding the statement of Proposition ?? can be run with the roles of $v$ and $w$ interchanged -- hence $L_{0,1}^P$ is exactly the torus $\Psi^P\left(\underline{\mathscr{P}}_{0,1}^{1/3}\right)$ given in the statement of Proposition ?? with $r=1/3$).

%\begin{theorem}\label{CP2}Given any two of the tori $T_{W}$, $T_{CS}^P$, and $L_{0,1}^P$, there is a symplectomorphism $\C P^2(\sqrt{2})\to \C P^2(\sqrt{2})$ taking one to the other.
%\end{theorem}

%Again, we prove this theorem in two propositions:

\begin{prop}\label{WtoL01P}
There is a symplectomorphism $\C P^2(\sqrt{2})\to \C P^2(\sqrt{2})$ taking $T_{W}$ to $L_{0,1}^P$.
\end{prop}

Just as in the case of $T_{FOOO}\subset S^2\times S^2$, most of our task will consist of giving a very explicit symplectomorphism of the manifold $\hat{F}_4(0)$ with $\C P^2(\sqrt{2})$; once this is achieved we will see almost immediately that our symplectomorphism maps $T_W$ to $L_{0,1}^{P}$

%While $L_{0,1}^P$ is given very explicitly as a submanifold of $\C P^2(\sqrt{2})$, the same cannot be said of $T_{W}$.  Similarly to the case of $T_{FOOO}\subset S^2\times S^2$, the torus $T_{W}$ is given as a submanifold of an abstract symplectic manifold that is denoted $\hat{F}_4(0)$ and that is symplectomorphic to $\C P^2(\sqrt{2})$, but it is not clear from the description in \cite{Wu} what the image of $T_{W}$ under such a symplectomorphism might be.  Thus, most of our task will consist of giving a construction of the manifold $\hat{F}_4(0)$ which allows it to be symplectically identified with $\C P^2(\sqrt{2})$ in a very explicit way.  In fact, once our construction is finished it will follow almost immediately that $T_{W}$ is mapped to $L_{0,1}^P$ by our symplectomorphism.  

%As an intermediary in our construction, we will use the cotangent bundle $T^{*}\R P^2$, which we view as the quotient of $T^{*} S^2$ under the $\Z/2\Z$ action induced by the antipodal map $S^2\to S^2$.  More explicitly, we have \[ T^{*} \R P^2 = \frac{T^{*} S^2}{(p,q)\sim (-p,-q)} \] with the symplectic form on $T^{*}\R P^2$ being the obvious one induced by the quotient map.  Writing $D_{1}^{*} \R P^2$ for the open unit disk cotangent bundle and $0_{\R P^2}$ for the zero-section, we now prove a couple of lemmas that will be used in proving \ref{WtoL01P}.

As above we identify $T^*\R P^2$ as the quotient of $T^*S^2$ by the antipodal involution.  We denote the zero section of $T^*\R P^2$ by $0_{\R P^2}$.

\begin{lemma}\label{cotball2} Where $\langle i \rangle$ is the multiplicative subgroup of $\C^*$ generated by $i=\sqrt{-1}$, the map $\varphi_1\co B^4(2)\setminus\{0\}\to D_{1}^{*}S^2\setminus 0_{S^2}$ from Lemma \ref{cotball} descends to a symplectomorphism $\underline{\Phi}_1\co (B^4(2)\setminus\{0\})/\langle i \rangle \to D_{1}^{*} \R P^2 \setminus 0_{\R P^2}$.  Moreover, where $f_{T^*\R P^2}([(p,q)]) = \frac{1}{2}|p|$ and $g_{T^*\R P^2}([(p,q)]) = (p\times q)\cdot e_1$, we have \[ 
f_{T^{*}\R P^2}\circ \underline{\Phi}_1 ([z_1+z_2j]) = \frac{1}{8}(|z_1|^2+|z_2|^2) \mbox{ and } g_{T^{*}\R P^2}\circ \underline{\Phi}_1([z_1+z_2j]) = \frac{1}{4}(|z_1|^2-|z_2|^2) \] for $z_1,z_2\in \C$ with $0<|z_1|^2+|z_2|^2<4$.
\end{lemma}

\begin{proof}
By Lemma \ref{cotball}, it is clear that the map $\varphi_1$ descends to a symplectomorphism $\Phi_1\co (B^4(2)\setminus\{0\})/\pm 1\to D_{1}^{*}S^2\setminus 0_{S^2}$, and the observation that
\begin{align*}
\varphi_1\left( i\xi \right) &= \left( -\frac{(i\xi)^\ast k (i\xi)}{4},\frac{(i\xi)^\ast j (i\xi)}{|i\xi|^2} \right) = \left(  -\frac{\xi^\ast i^\ast k i\xi}{4},\frac{\xi^\ast i^\ast j i\xi}{|\xi|^2} \right) \\
&= \left(  \frac{\xi^\ast i^\ast i k \xi}{4},-\frac{\xi^\ast i^\ast i j \xi}{|\xi|^2} \right) = \left(  \frac{\xi^\ast k \xi}{4},-\frac{\xi^\ast j \xi}{|\xi|^2} \right)
\end{align*}
shows that $\Phi_1$ is equivariant with respect to the $\Z/2\Z$ actions on $(B^4(2)\setminus\{0\})/\pm 1$ and $D_{1}^{*}S^2\setminus 0_{S^2}$ given by multiplication by $i$ and the induced action of the antipodal map $S^2\to S^2$, respectively.  Since the quotient of $(B^4(2)\setminus\{0\})/\pm 1$ by the $\Z/2\Z$ action is obviously symplectomorphic to $(B^4(2)\setminus\{0\})/\langle i \rangle$, it follows that $\varphi_1$ descends to a symplectomorphism $\underline{\Phi}_1\co (B^4(2)\setminus\{0\})/\langle i \rangle \to D_{1}^{*} \R P^2 \setminus 0_{\R P^2}$ exactly as claimed.  The computations of the maps $f_{T^{*}\R P^2}\circ \underline{\Phi}_1$ and $g_{T^{*}\R P^2}\circ \underline{\Phi}_1$ are nearly identical to the analogous ones in the proof of Lemma \ref{cotball}.
\end{proof}

%\begin{lemma}\label{cotcp2} Where $f_{T^*\R P^2}$ and $g_{T^*\R P^2}$ are as in Lemma \ref{cotball2}, there is a symplectomorphism $\Psi^P:D_{1}^{*} \R P^2 \to \C P^2(\sqrt{2})\setminus Q_2(\sqrt{2})$ satisfying \[ f_{T^{*}\R P^2}\circ (\Psi^P)^{-1}([z_0:z_1:z_2]) = \frac{1}{4}\sqrt{4- \left| \sum z_j^2 \right|^2} \] and \[ g_{T^{*}\R P^2}\circ (\Psi^P)^{-1}([z_0:z_1:z_2]) = Im(\bar{z}_1z_2) \] for $[z_0:z_1:z_2]\in \C P^2(\sqrt{2})\setminus Q_2(\sqrt{2})$.
%\end{lemma}

%\begin{proof}
%This is nothing more than a special case of the map $\Psi^P$ discussed immediately after the proof of Proposition ??.
%\end{proof}

We may now give a construction of the manifold $\hat{F}_4(0)$ and prove Proposition \ref{WtoL01P}.

\begin{proof}[Proof of Proposition \ref{WtoL01P}]
Since the preimage of the zero-section $0_{\R P^2}$ under the map $(\Psi^P)^{-1}$ is $\R P^2\subset \C P^2(\sqrt{2})$, it follows from Lemmas \ref{rpncpn} and \ref{cotball2} and Remark \ref{pullbacks}  that the map \[ \Psi^P\circ \underline{\Phi}_1\co \frac{B^{4}(2)\setminus \{0\} }{\langle i \rangle}\to \C P^2(\sqrt{2})\setminus (\R P^2\cup Q_2(\sqrt{2})) \] is a symplectomorphism which pulls back the function $[z_0:z_1:z_2]\mapsto Im(\bar{z}_1z_2)$ to the function $[(z_1,z_2)]\mapsto 
\frac{1}{4}(|z_1|^2-|z_2|^2)$ and pulls back the function $[z_0:z_1:z_2]\mapsto \frac{1}{4}\sqrt{4- \left| \sum z_j^2 \right|^2}$ to the function $[(z_1,z_2)]\mapsto 
\frac{1}{8}(|z_1|^2+|z_2|^2)$.

Consequently we may introduce the symplectic $4$-orbifold \[ \underline{\mathcal{O}} = \frac{(B^4(2)/\langle i \rangle)\coprod \left( \C P^2(\sqrt{2})\setminus \R P^2 \right)}{[(z_1,z_2)]\sim \Psi^P\circ \underline{\Phi}_1([(z_1,z_2)])\mbox{ for }(z_1,z_2)\neq (0,0)} \]  since the fact that $\Psi^P\circ \underline{\Phi}_1$ is a symplectomorphism shows that the the symplectic forms on $B^4(2)$ and on $\C P^2(\sqrt{2})\setminus \R P^2$ coincide on their overlap in $\underline{\mathcal{O}}$.
Moreover we have well-defined functions $\underline{F}\co \underline{\mathcal{O}}\to\R$ and $\underline{G}\co \underline{\mathcal{O}}\to\R$ defined by \begin{align*} \underline{F}([(z_1,z_2)])&=\frac{1}{8}(|z_1|^2+|z_2|^2)\qquad & \underline{G}([(z_1,z_2)]) = \frac{1}{4}(|z_1|^2-|z_2|^2)  \\
\underline{F}([z_0:z_1:z_2]) & = \frac{1}{4}\sqrt{4- \left| \sum z_j^2 \right|^2} \qquad & \underline{G}([z_0:z_1:z_2]) = Im(\bar{z}_1z_2)  \end{align*}
for $(z_1,z_2)\in B^4(2)$ and $[z_0:z_1:z_2]\in \C P^2(\sqrt{2})\setminus \R P^2$.

One easily verifies that the map $(2\underline{F}+\underline{G},1/2-\underline{F})\co \underline{\mathcal{O}}\to \R^2$ is a moment map for a symplectic toric action on the symplectic orbifold $\underline{\mathcal{O}}$, with image equal to the polytope $\Delta_{W}$.  The classification of toric orbifolds from \cite{LT} therefore implies that $\underline{\mathcal{O}}$ is equivariantly symplectomorphic to the orbifold $F_{4}(0)$ (as $\underline{\mathcal{O}}$ and $F_4(0)$ have identical moment polytopes and both have only one singular point, located at the preimage of $(0,1/2)$ under the moment map); accordingly we hereinafter implicitly identify $F_4(0)$ with $\underline{\mathcal{O}}$.  The manifold $\hat{F}_{4}(0)$ is then constructed by removing a neighborhood $\underline{\mathcal{U}}$ of the unique singular point $[(0,0)]$ of $\underline{\mathcal{O}}$ and gluing in its place a neighborhood $\underline{\mathcal{N}}$ of $0_{\R P^2}$ in the cotangent bundle $T^*\R P^2$, using a symplectomorphism between $\underline{\mathcal{U}}\setminus\{[(0,0)]\}$ and $\underline{\mathcal{N}}\setminus 0_{\R P^2}$.  While a particular choice of this symplectomorphism is not specified in \cite{Wu}, we have already constructed one that will serve the purpose, namely the map $\underline{\Phi}_1\co (B^4(2)\setminus\{0\})/\langle i \rangle \to D_{1}^{*}\R P^2\setminus 0_{R P^2}$ from Lemma \ref{cotball2}.  This gives a symplectomorphism between the manifold $\hat{F}_{4}(0)$ and the manifold \[ \frac{D_{1}^{*}\R P^2\coprod \left( \C P^2(\sqrt{2})\setminus \R P^2 \right)}{[(p,q)]\sim \Psi^P([(p,q)])\mbox{ for }[(p,q)]\in D_{1}^{*}\R P^2\setminus 0_{\R P^2}}. \]  But of course the map $\Psi^P$ then induces a symplectomorphism between this latter manifold and $\C P^2(\sqrt{2})$.  

There is an obvious continuous map $\underline{\Pi}\co \hat{F}_{4}(0)\to F_{4}(0)$ which maps the zero-section $0_{\R P^2}$ to the singular point $[(0,0)]$ and coincides with $\underline{\Phi}_{1}^{-1}$ on $D_{1}^{*}\R P^2\setminus 0_{\R P^2}\subset \hat{F}_{4}(0)$ and with the identity on $\C P^2(\sqrt{2})\setminus \R P^2 \subset \hat{F}_{4}(0)$; the monotone Lagrangian torus $T_{W}$ is the preimage of the point $(1/3,1/3)$ under the pulled-back moment map $((2\underline{F}+\underline{G})\circ \underline{\Pi},(1/2-\underline{F})\circ\underline{\Pi})\co \hat{F}_4(0)\to \R^2$.  In view of the expressions for the functions $\underline{F},\underline{G}$ on $\C P^2(\sqrt{2})\setminus \R P^2$, it follows that $T_{W}$ is taken by our symplectomorphism $\hat{F}_4(0)\to \C P^2(\sqrt{2})$ to \[ \left\{ [z_0:z_1:z_2] \left|\thinspace \frac{1}{2}\sqrt{4- \left| \sum z_j^2 \right|^2} + Im(\bar{z}_1z_2) = \frac{1}{3},\thinspace \frac{1}{2}-\frac{1}{4}\sqrt{4- \left| \sum z_j^2 \right|^2}=\frac{1}{3} \right.\right\},\]
\emph{i.e.} to \[ \left\{ [z_0:z_1:z_2] \left|\thinspace \left|\sum z_{j}^{2}\right| = \frac{4\sqrt{2}}{3},\,  Im(\bar{z}_1z_2) = 0 \right.\right\},\] which is precisely $L_{0,1}^{P}$.
\end{proof}

Before proving the equivalence of $T_{CS}^{P}$ and $L_{0,1}^{P}$ we require an area computation:

\begin{lemma}\label{arealemma} For $0<\alpha<1$ the area enclosed by the curve $\left\{w\in \mathbb{H}(\sqrt{2})\left| |w^2+2-|w|^2|^2=4(1-\alpha^2)\right.\right\}$ is $\pi(1-\alpha)$.
\end{lemma}

\begin{proof} Where $D(\sqrt{2})$ is the open disk in $\C$ of radius $\sqrt{2}$ define $\tilde{h}_1\co D(\sqrt{2})\to \C^2$ by \[ \tilde{h}_1(\zeta) = \left( i\left(\sqrt{1-|\zeta|^2/2}+\zeta/\sqrt{2}\right) , \sqrt{1-|\zeta|^2/2}-\zeta/\sqrt{2}\right) \]  It is easy to check that $\tilde{h}_1$ has image contained in the sphere of radius $\sqrt{2}$ and pulls back the form $-y_1dx_1+x_2dy_2$ to a primitive for the standard symplectic form on $D(\sqrt{2})\subset \C$.  Hence $\tilde{h}_1$ descends to an isosymplectic map $h_1\co D(\sqrt{2})\to \C P^1(\sqrt{2})$ defined by 
\[ h_1(\zeta) = \left[ i\left(\sqrt{1-|\zeta|^2/2}+\zeta/\sqrt{2}\right) : \sqrt{1-|\zeta|^2/2}-\zeta/\sqrt{2}\right] \]
which moreover is easily seen to be an embedding (for instance  composing it with the unitary transformation $\frac{1}{\sqrt{2}}\left(\begin{array}{cc} i & 1 \\ -i & 1\end{array}\right)$ yields a map which is obviously injective).
$h_1$ pulls back the function $[z_0:z_1]\mapsto |z_{0}^{2}+z_{1}^{2}|^2$ to $\zeta\mapsto 4|\zeta|^2(2-|\zeta|^2)$, and it pulls back the function $[z_0:z_1]\mapsto Im(z_0\bar{z}_1)$ to the function $\zeta\mapsto 1-|\zeta|^2$.  Thus $h_1$ restricts to the unit disk $D(1)$ as a symplectomorphism to the subset $\{[z_0:z_1]\in \C P^1(\sqrt{2}):Im(z_0\bar{z}_1)>0\}$.

Meanwhile another symplectomorphism whose image is this latter subset is given by the map $h_2\co \mathbb{H}(\sqrt{2})\to \{[z_0:z_1]\in \C P^1(\sqrt{2}):Im(z_0\bar{z}_1)>0\}$ defined by $h_2(w)=[w:\sqrt{2-|w|^2}]$.  The map $h_2$ pulls back the function 
$[z_0:z_1]\mapsto |z_{0}^{2}+z_{1}^{2}|^2$ to $w\mapsto |w^2+2-|w|^2|^2$.  

We thus obtain a symplectomorphism $h_{2}^{-1}\circ h_1\co D(1)\to \mathbb{H}(\sqrt{2})$ which sends the disk of any given radius $r<1$ to the region \[ \{w\in \mathbb{H}(\sqrt{2})\left|\thinspace |w^2+2-|w|^2|^2<4r^2(2-r^2)\right.\}  \]  Thus the region in the statement of the lemma has the same area as the disk of radius equal to the number $r\in (0,1)$ obeying $4r^2(2-r^2)=4(1-\alpha^2)$.  But this value of $r$ is precisely $r=\sqrt{1-\alpha}$ and so the lemma follows. 
\end{proof}  

\begin{prop}\label{CStoL01P}
There is a symplectomorphism $\C P^2(\sqrt{2})\to \C P^2(\sqrt{2})$ taking $T_{CS}^P$ to $L_{0,1}^P$.
\end{prop}

\begin{proof}
Recall that $T_{CS}^P$ is defined as $\psi_P\left(\Theta_{CS}^P\right)$, where
\begin{align*}
\psi_P:B^4(\sqrt{2})&\to \C P^2(\sqrt{2}) \\
(z_1,z_2)&\mapsto [\sqrt{2-|z_1|^2 - |z_2|^2}:z_1:z_2]
\end{align*}
is a symplectomorphism (as a routine computation shows) and \[ \Theta_{CS}^P = \{(e^{it}z,e^{-it}z)\mid z\in \Gamma_P, t\in [0,2\pi] \} \subset B^2(1)\times B^2(1) \subset B^4(\sqrt{2}) \] for a curve $\Gamma_P\subset \mathbb{H}(1)$ enclosing area $\frac{\pi}{3}$ (the Hamiltonian isotopy class of $T_{CS}^{P}$ is easily seen to be independent of the particular choice of $\Gamma_P$).  Alternatively, $\Theta_{CS}^P$ is given as the orbit of the curve $\Delta_{\Gamma_P}=\{(z,z)\mid z\in \Gamma_P\}$ under the circle action $\rho$ given by \[ \rho(e^{it})\cdot (z_1,z_2) = \begin{pmatrix} e^{it} & 0 \\ 0 & e^{-it} \end{pmatrix} \begin{pmatrix} z_1 \\ z_2 \end{pmatrix}. \] 

We now consider points $[z_0:z_1:z_2]$ of $L_{0,1}^{P}$.  Since such points have $|\sum z_{j}^{2}|=4\sqrt{2}/3<2$, not all of $z_0,z_1,z_2$ can be real multiples of each other.  On the other hand since $Im(\bar{z}_1z_2)=0$, $z_1$ and $z_2$ \emph{are} real multiples of each other.  So $z_0$ must be nonzero, and by modifying $(z_0,z_1,z_2)$ in its fiber under the Hopf map we may assume that $z_0>0$, and then we can uniquely write $(z_1,z_2)=(w\cos t,w\sin t)$ where $w\in \mathbb{H}(\sqrt{2})$ and $(\cos t,\sin t)\in S^1$.  We will then have $z_0=\sqrt{2-|w|^2}$, and the fact that $|\sum z_{j}^{2}|=4\sqrt{2}/3$ amounts to the statement that $\left|w^2+2-|w|^2\right|^2=32/9$.

Thus, where \[ C_P = \left\{[\sqrt{2-|w|^2}:w:0]\left|\thinspace w\in \mathbb{H}(\sqrt{2}), \left|w^2+2-|w|^2\right|^2=\frac{32}{9} \right.\right\}, \] $L_{0,1}^P$ is the orbit of the curve $C_P$ under the circle action \[ e^{it}\cdot[z_0:z_1:z_2] = [z_0:z_1 \cos t + z_2 \sin t: z_1(-\sin t) + z_2 \cos t]. \]  Moreover, $L_{0,1}^P$ is contained in the image of $\psi_P$, and we observe that $\psi_P^{-1}(L_{0,1}^P)$ is the orbit of the curve $C_{0,1} \times \{0\} \subset B^4(\sqrt{2})$, where \[ C_{0,1}=\left\{w\in \mathbb{H}(\sqrt{2}) \left|\thinspace \left|w^2+2-|w|^2\right|^2=\frac{32}{9} \right.\right\}, \] under the    circle action $\rho_{0,1}$ given by \[ \rho_{0,1}(e^{it})\cdot (z_1,z_2) = \begin{pmatrix} \cos t & \sin t \\ -\sin t & \cos t \end{pmatrix} \begin{pmatrix} z_1 \\ z_2 \end{pmatrix}. \]

We then observe that the circle actions $\rho$ and $\rho_{0,1}$ are conjugate in $U(2)$ since \[ \begin{pmatrix} \cos t & \sin t \\ -\sin t &  \cos t \end{pmatrix} = \mathcal{Q}^{-1} \begin{pmatrix} e^{it} & 0 \\ 0 &  e^{-it} \end{pmatrix} \mathcal{Q} \quad\mbox{for}\quad \mathcal{Q} = \begin{pmatrix} \frac{1}{\sqrt{2}} & \frac{-i}{\sqrt{2}} \\ \frac{1}{\sqrt{2}} & \frac{i}{\sqrt{2}} \end{pmatrix}\in U(2). \]  Hence, it follows that 
\begin{equation}\label{cp2conjugacyeqn}
\mathcal{Q}\left( \rho_{0,1}(e^{it}) \cdot \begin{pmatrix} z_1 \\ z_2 \end{pmatrix} \right) =  \rho(e^{it})\cdot \left( \mathcal{Q} \begin{pmatrix} z_1 \\ z_2 \end{pmatrix} \right).
\end{equation}
 Noting that $\mathcal{Q} \begin{pmatrix} z_1 \\ 0 \end{pmatrix} = \frac{1}{\sqrt{2}} \begin{pmatrix} z_1 \\ z_1 \end{pmatrix}$, we observe that $\mathcal{Q}\left(C_{0,1} \times \{0\}\right)$ is a curve $\Delta_{\tilde{C}}=\{(z,z)\mid z\in \tilde{C}\}$ in the diagonal of $B^2(1)\times B^2(1)$, where $\tilde{C}\subset \mathbb{H}(1)$ is a curve that we claim encloses an area of $\frac{\pi}{3}$.  Indeed, since $\mathcal{Q}$ is a symplectomorphism, the area in the diagonal enclosed by $\Delta_{\tilde{C}}$ is equal to that enclosed by $C_{0,1}\times \{0\}$.  According to Lemma \ref{arealemma} with $\alpha=1/3$, the area enclosed by $C_{0,1}\times \{0\}$ and hence by $\Delta_{\tilde{C}}$ is $\frac{2\pi}{3}$; and since the symplectic form on the diagonal is twice the standard symplectic form on $\C$, it follows that $\tilde{C}$ encloses an area of $\frac{\pi}{3}$ as claimed.  Taking the curve $\Gamma_P$ in Chekanov and Schlenk's construction to be the curve $\tilde{C}$, it follows from \eqref{cp2conjugacyeqn} that $\Theta_{CS}^P = \mathcal{Q}\left( \psi_P^{-1}\left(L_{0,1}^P\right)\right)$.  Finally, it follows that $T_{CS}^P$ is the image of $L_{0,1}^P$ under the symplectomorphism which restricts to the (dense) image of $\psi_P$ as $\psi_P\circ\mathcal{Q}\circ \psi_{P}^{-1}$.
\end{proof}

\section{The Biran circle bundle construction}\label{sectionBiran}

We recall the general description of standard symplectic disk bundles from \cite{B}, which we present with a few minor modifications pertaining to normalization.  Let $(\Sigma,\omega_{\Sigma})$ be a symplectic manifold,  let $\pi_P\co P\to \Sigma$ be a principal $S^1$-bundle with Chern class $[\omega_{\Sigma}]/\tau$ for some $\tau>0$, and let $\beta\in \Omega^1(P)$ be a connection $1$-form on $P$ with curvature equal to $2\pi i\omega_{\Sigma}/\tau$, normalized so that $\beta$ evaluates on the vector field generating the $S^1$-action as $\frac{1}{2\pi}$.  Thus we will have $-\tau d\beta = \pi_{P}^{*}(\omega_{\Sigma})$.  

The \emph{standard symplectic disk bundle} to $\Sigma$ associated to the pair $(P,\beta)$ is the symplectic manifold $(\mathcal{D}_{\tau}(P),\Omega)$ defined as follows.  The smooth manifold $\mathcal{D}_{\tau}(P)$ is the quotient \[ \mathcal{D}_{\tau}(P)=\frac{P\times D(\sqrt{\tau/\pi})}{(e^{i\theta}\cdot w,\zeta)\sim (w,e^{i\theta}\zeta)} \] where $D(\sqrt{\tau/\pi})$ denotes the disk of radius $\sqrt{\tau/\pi}$ in $\C$, and the symplectic form $\Omega$ is defined by the property that, where $pr\co P\times D(\sqrt{\tau/\pi})\to \mathcal{D}_{\tau}(P)$ is the quotient projection and where $\omega_{\C}=\frac{i}{2}d\zeta\wedge d\bar{\zeta}$ is the standard symplectic form on $\C$, \[ pr^*\Omega = d\left((\pi|\zeta|^2-\tau)\beta\right)+\omega_{\C} \]  Thus $\Sigma$ naturally embeds into $\mathcal{D}_{\tau}(P)$ as the ``zero section'' $\{[(w,0)]|w\in P\}$ with $\Omega|_{\Sigma}=\omega_{\Sigma}$, and the projection $\pi\co [(w,\zeta)]\mapsto[(w,0)]$ gives $\mathcal{D}_{\tau}(P)$ the structure of a fiber bundle whose fibers are symplectic disks of area $\tau$.
If $\Lambda\subset \Sigma$ is any monotone Lagrangian submanifold, then for any $r\in (0,\sqrt{\tau/\pi})$ the submanifold \[ \Lambda_{(r)} = \left\{[(w,\zeta)]\in \mathcal{D}_{\tau}(P)\left|\thinspace |\zeta|=r,\,\pi([(w,\zeta)])\in\Lambda\right.\right\} \] will be a monotone Lagrangian submanifold of $\mathcal{D}_{\tau}(P)$.

The main theorem of \cite{B} asserts that, if $(\Sigma,\omega_{\Sigma})$ is a complex hypersurface of a K\"ahler manifold $(M,\omega)$ which is Poincar\'e dual to the cohomology class $\frac{1}{\tau}[\omega]$, then $(\mathcal{D}_{\tau}(P),\Omega)$ symplectically embeds into $M$ as the complement of an isotropic CW complex. There will typically (always, if $\dim \Sigma\geq 4$; see \cite[Proposition 6.4.1]{BC09}) be a unique value of $r$ with the property that the embedding maps $\Lambda_{(r)}$ to a monotone Lagrangian submanifold $L$ of $M$; $L$ is then called the \emph{Biran circle bundle construction} associated to $\Lambda$.

Consistently with this, we will now give explicit such embeddings for three special cases: where $\Sigma$ is the diagonal $\Delta$ in $M=S^2\times S^2$; where $\Sigma$ is the quadric $Q_{n}(\sqrt{2})$ and $M=Q_{n+1}(\sqrt{2})$; and where $\Sigma$ is $Q_{n}(\sqrt{2})$ and $M=\C P^n(\sqrt{2})$.  (Under a suitable symplectomorphism the first of these cases can be viewed as the special case of the second where $n=2$, but we will handle it separately in order to make the proof of the last equivalence in Theorem \ref{mains2s2} clearer.)

\subsection{$S^2\times S^2$}

To begin with the case of $\Delta\subset S^2\times S^2$, we identify $\Delta$ with $S^2$ in the obvious way and we use for the principal $S^1$-bundle $P_{\Delta}$ the unit circle bundle in the tangent bundle to $S^2$, which may be identified as the following subset of $\R^3\times \R^3$: \[ P_{\Delta}=\left\{(x,y)\in \R^3\times\R^3\left|\thinspace |x|=|y|=1,\,x\cdot y=0\right.\right\}\] with bundle projection given by $(x,y)\mapsto x$ and circle action given by $e^{it}\cdot(x,y) = (x,(\cos t)y+(\sin t)x\times y)$.  The $1$-form $\beta\in \Omega^1(P_{\Delta})$ defined by $\beta_{(x,y)}(a,b) = b\cdot(x\times y)$ is then a connection $1$-form for $P_{\Delta}$.  A routine computation shows that, for $(a_1,b_1),(a_2,b_2)\in T_{(x,y)}P_{\Delta}$, one has \[ d\beta_{(x,y)}((a_1,b_1),(a_2,b_2)) = \frac{1}{2\pi}\left(2x\cdot(b_1\times b_2)-y\cdot\left(a_1\times b_2 + b_1\times a_2\right)   \right) \] and that moreover $d\beta$ is $-\frac{1}{2\pi}$ times the pullback of the standard symplectic form $(\omega_{\mbox{\scriptsize std}})_x(a_1,a_2)=x\cdot(a_1\times a_2)$ on $S^2$.  (Note that since our convention is to take a symplectic form $\Omega_{S^2\times S^2}$ on $S^2\times S^2$ giving area $2\pi$ to each factor, the restriction $\Omega_{S^2\times S^2}\big|_{\Delta}$ will coincide with $\omega_{\mbox{\scriptsize std}}$ under the obvious identification of $\Delta$ with $S^2$.)  

Thus \cite{B} implies the existence of a symplectomorphism from $\mathcal{D}_{2\pi}(P_{\Delta}) = \frac{P_{\Delta}\times D(\sqrt{2})}{S^1}$ to the complement of an isotropic CW complex in $S^2\times S^2$; we will now explicitly construct such a symplectomorphism.  Now the map $[(x,y),re^{i\theta}]\mapsto (x,(r\cos\theta)y+(r\sin\theta)x\times y)$ gives a diffeomorphism of $\mathcal{D}_{2\pi}(P_{\Delta})$ to the radius-$\sqrt{2}$ disk bundle \[ D_{\sqrt{2}}S^2 = \left\{(x,y)\in \R^3\times \R^3\left||x|=1,\,|y|<\sqrt{2},\,x\cdot y=0\right.\right\} \] in $TS^2$, and in terms of this diffeomorphism the Biran symplectic form $\Omega$  restricts to the complement of the zero section as the exterior derivative of the form $\eta\in \Omega^1(D_{\sqrt{2}}S^2\setminus 0_{S^2})$ given by \[ \eta_{(x,y)}(a,b)  = \left(\frac{1}{2}-\frac{1}{|y|^2}\right)b\cdot(x\times y). \]

Now define \[ \Theta_{\Delta}\co D_{\sqrt{2}}S^2\to \left( S^2\times S^2 \right)\setminus \overline{\Delta} \] by \[ \Theta_{\Delta}(x,y) = \left(\left(1-\frac{|y|^2}{2}\right)x+\sqrt{1-\frac{|y|^2}{4}}y, \left(1-\frac{|y|^2}{2}\right)x-\sqrt{1-\frac{|y|^2}{4}}y \right) \]
One easily sees (using the identity $|v-w|^2=(2-|v+w|)(2+|v+w|)$ for $(v,w)\in S^2\times S^2$) that the smooth map $\left( S^2\times S^2 \right)\setminus \overline{\Delta}\to D_{\sqrt{2}}S^2$ defined by \[ (v,w)\mapsto \left(\frac{v+w}{|v+w|},\frac{v-w}{\sqrt{2+|v+w|}}\right) \] is an inverse to $\Theta_{\Delta}$, so $\Theta_{\Delta}$ is a diffeomorphism to $\left( S^2\times S^2 \right)\setminus \overline{\Delta}$.  Evidently $\Theta_{\Delta}$ maps the zero section of $D_{\sqrt{2}}S^2$ to $\Delta\subset S^2\times S^2$.  To see that $\Theta_{\Delta}$ is a symplectomorphism we consider its composition (after restricting to the complement of the zero-section) with our previously-defined symplectomorphism $\Phi_2\co \left( S^2\times S^2 \right)\setminus \Delta \to D_{1}^{*}S^2$, given by $\Phi_2(v,w) = \left(\frac{v\times w}{|v-w|},\frac{v-w}{|v-w|}\right)$.  We find that, for $(x,y)\in D_{\sqrt{2}}S^2$ with $y\neq 0$, \[ \Phi_2\circ \Theta_{\Delta}(x,y) = \left( (|y|^2/2-1)\left(x\times\frac{y}{|y|}\right), \frac{y}{|y|}\right).\]  Where $\lambda$ is the canonical one-form on $T^*S^2$ we see that \[ \left((\Phi_2\circ\Theta_{\Delta})^{*}\lambda\right)_{(x,y)}(a,b) =  (|y|^2/2-1)\left(x\times\frac{y}{|y|}\right)\cdot \frac{b}{|y|} = \left(\frac{1}{2}-\frac{1}{|y|^2}\right)b\cdot(x\times y) \] which, as noted earlier, is a primitive for the Biran symplectic form $\Omega$ on the complement of the zero-section in $D_{\sqrt{2}}S^2$.  Since $\Phi_2$ is a symplectomorphism, this implies that $\Theta_\Delta$ is a symplectomorphism on the complement of the zero-section, and hence globally by continuity.

The Biran circle bundle construction $T_{BC}$ over the circle $\Lambda=\{x_1=0\}$ in $\Delta\cong S^2$ will then be the image under $\Theta_{\Delta}$ of the unit circle bundle over $\Lambda$.  (The radius $1$ is the radius necessary to guarantee monotonicity, as in \cite[Proposition 6.4.1]{BC09} after one adjusts for differences in normalization.)  We can now complete the proof of Theorem \ref{mains2s2}, as follows:

\begin{prop}\label{bcep} The tori $T_{BC}$ and $T_{EP}=\{(v,w)\in S^2\times S^2|(v+w)\cdot e_1=0,\,v\cdot w=-1/2\}$ are equal.
\end{prop}

\begin{proof}  We observe that if $(v,w)=\Theta_{\Delta}(x,y)$ for $(x,y)\in D_{\sqrt{2}}S^2$, then \[ (v+w)\cdot e_1 = (2-|y|^2)x_1\] and \[ v\cdot w = \left(1-\frac{|y|^2}{2}\right)^2-\left(1-\frac{|y|^2}{4}\right)|y|^2 = 1-2|y|^2+|y|^4/2 \]  Given that $|y|<\sqrt{2}$, we thus see that $(v,w)\in T_{EP}$ if and only if $x_1=0$ and $|y|=1$, \emph{i.e.} if and only if $(x,y)$ lies in the unit circle bundle over $\Lambda$, \emph{i.e.} if and only if $(v,w)=\Theta_{\Delta}(x,y)\in T_{BC}$.
\end{proof}

\subsection{Quadrics}\label{birquad}

We now discuss the Biran circle bundle construction where the ambient manifold is the quadric \[ Q_{n+1}(\sqrt{2})=\left\{[z_0:\cdots:z_{n+1}]\in \C P^{n+1}(\sqrt{2})\left|\sum z_{j}^{2}=0\right.\right\} \] and the hypersurface $\Sigma$ is $Q_{n}(\sqrt{2})$, identified as the set of points $[z_0:\cdots:z_{n+1}]\in Q_{n+1}(\sqrt{2})$ with $z_0=0$.  

There is a principal $S^1$-bundle $P_{Q}\to Q_n(\sqrt{2})$ given by \[ P_Q=\{w=u+iv\in \C^{n+1}|\|u\|=\|v\|=1,\,u\cdot v=0\},\] with bundle projection $w\mapsto[w]$.  We use the opposite of the complex $S^1$-action on $P_Q$  (of course $P_Q\subset S^{2n+1}(\sqrt{2})$ is just the preimage of $Q_n(\sqrt{2})$ under the Hopf projection, and the projection $P_Q\to Q_{n}(\sqrt{2})$ is the Hopf map).  $\beta=\frac{1}{4\pi}\sum(v_jdu_j-u_jdv_j)$ gives a connection form on $P_Q$.  We set $\tau=2\pi$ and form $\mathcal{D}^{n}_{2\pi}(P_Q)$ as a quotient of $P_Q\times D(\sqrt{2})$ just as earlier, with symplectic form that pulls back to $P_Q\times D(\sqrt{2})$ as $d((\pi|\zeta|^2-2\pi)\beta)+\frac{i}{2}d\zeta\wedge d\bar{\zeta}$.  Then where \[ \mathbb{S}_{0,n} =\{[1:ix_1:\cdots:ix_{n+1}] | x_i\in \R,\,\sum x_{j}^{2}=1\}\] we define the map \[ \Theta_Q\co \mathcal{D}^{n}_{2\pi}(P_Q)\to Q_{n+1}(a)\setminus \mathbb{S}_{0,n} \] by \[ \Theta_Q([w,\zeta])=\left[\sqrt{1-\frac{|\zeta|^2}{4}}\zeta: \left(1-\frac{|\zeta|^2}{4}\right)w-\frac{\zeta^2\bar{w}}{4}\right] \] and claim that $\Theta_Q$ is a symplectomorphism.    To see this, note that for $r\geq 0$ we have \[ \Theta_Q([u+iv,r])=[r\sqrt{1-r^2/4}:(1-r^2/2)u+iv]. \]  Now the map $\alpha\co P_Q\times (0,\sqrt{2})\to \mathcal{D}_{2\pi}^{n}(P_Q)$ defined by $(u+iv,r)\mapsto [u+iv,r]$ is a diffeomorphism to the complement of the zero-section in  
$\mathcal{D}_{2\pi}^{n}(P_Q)$, and the pullback of the symplectic form on $\mathcal{D}_{2\pi}^{n}(P_Q)$ by $\alpha$ is $d\left(\left(\frac{r^2}{4}-\frac{1}{2}\right)\sum(v_jdu_j-u_jdv_j)\right)$.   Meanwhile $\Theta_Q\circ\alpha$ lifts in obvious fashion to a map $P_Q\times (0,\sqrt{2})\to \C^{n+2}$ with image in the sphere of radius $\sqrt{2}$, which pulls back the primitive $\frac{1}{2}\sum (x_jdy_j-y_jdx_j)$ for the standard symplectic form on $\C^{n+2}$ precisely to $\left(\frac{r^2}{4}-\frac{1}{2}\right)\sum(v_jdu_j-u_jdv_j)$ (bearing in mind that $\sum u_jv_j=0$).  Thus $\Theta_Q$ pulls back the Fubini-Study symplectic form on $Q_{n+1}(\sqrt{2})$ to the Biran form $\Omega$ on $\mathcal{D}_{2\pi}^{n}(P_Q)$, as claimed.

Consequently, for any $k,m$ with $k+m=n-1$, we may construct a monotone Lagrangian submanifold $(\mathbb{S}_{k,m})^Q$ as the image under $\Theta_Q$ of an appropriate-radius circle bundle in $\mathcal{D}_{2\pi}^{n}(P_Q)$ over $\mathbb{S}_{k,m} = \{[ix:y]|x\in S^{k},\,y\in S^m\}$; these will be considered further in Section \ref{quadsubs}.

\subsection{Complex projective space}\label{bircpn}

Now define a new principal $S^1$-bundle $P_P\to Q_n(\sqrt{2})$ by \[ P_P=\frac{P_Q}{w\sim -w}, \] with $S^1$-action given by $e^{it}[w]=[e^{-it/2}w]$.  The one-form $\beta$ on $P_P$ defined by the property that its pullback to $P_Q$ is given by $\frac{1}{2\pi}\sum (v_jdu_j-u_jdv_j)$ (writing $w=u+iv$)   is a connection form on $P_P$, and so we set $\tau=\pi$ and form $\mathcal{D}^{n}_{\pi}(P_P)$ as a quotient of $P_P\times D(1)$ with symplectic form which pulls back by the map $\alpha_P\co P_P\times (0,1)\to \mathcal{D}^{n}_{\pi}(P_P)$ given by $([w],r)\mapsto [([w],r)]$ as $d\left(\pi(r^2-1)\beta\right)$.  Now define \[ \Theta_P\co \mathcal{D}^{n}_{\pi}(P_P)\to \C P^{n}(\sqrt{2})\setminus \R P^n \] by \[ \Theta_P([([w],\zeta)]) = \left[\sqrt{1-|\zeta|^2/2}w -\zeta\bar{w}/\sqrt{2}\right] \]  For $u,v\in \R^{n+1}$ with $|u|=|v|=1$ and $u\cdot v = 0$ we have \[ \Theta_P\circ\alpha_P([u+iv],r) = \left[\left(\sqrt{1-r^2/2}-r/\sqrt{2}\right)u+i\left(\sqrt{1-r^2/2}+r/\sqrt{2}\right)  v \right], 
\] which locally lifts to a map to $S^{2n+1}$ that pulls back the primitive $\frac{1}{2}\sum (x_jdy_j-y_jdx_j)$ to $\pi(r^2-1)\beta$.  As in the case of $\Theta_Q$, this implies that $\Theta_P$ is a symplectomorphism to $\mathbb{C}P^n(\sqrt{2})\setminus \R P^n$.

We may thus use $\Theta_P$ to apply the Biran circle bundle construction to the submanifolds $\mathbb{S}_{k,m}\subset Q_n(\sqrt{2})$ for $k+m=n-1$, yielding monotone Lagrangian submanifolds $(\mathbb{S}_{k,m})^{P}\subset \C P^{k+m+1}(\sqrt{2})$ that will be discussed further in Section \ref{cpnsubs}.  In particular, $(\mathbb{S}_{0,1})^{P}\subset \C P^2(\sqrt{2})$ is the submanifold denoted $T_{BC}^{P}$ in Theorem \ref{maincp2}.

\subsection{A general criterion for displaceability}\label{criterion}

One of our main results, Theorem \ref{maindisp}, asserts that certain of the monotone Lagrangian submanifolds $(\mathbb{S}_{k,m})^Q$ and $(\mathbb{S}_{k,m})^P$ are displaceable.  We will prove this in Sections \ref{quadsubs} and \ref{cpnsubs} based on explicit parametrizations of these specific submanifolds, but we would like to sketch here a more general context in terms of which the displaceability of these submanifolds can be understood.  

We consider a monotone K\"ahler manifold $(M,\omega,J)$, and assume that \begin{equation}\label{omalpha} [\omega]|_{\pi_2(M)}=\alpha c_1(TM)|_{\pi_2(M)}\neq 0 \qquad (\alpha>0) \end{equation} and that $\Sigma\subset M$ is a smooth complex hypersurface as in Biran's construction, with \[ [\omega] = \tau PD([\Sigma]) \qquad (\tau>0) \] where $PD$ denotes Poincar\'e duality, so that the complement of an isotropic CW complex in $M$ is symplectomorphic to a standard symplectic disk bundle $\mathcal{D}_{\tau}(P)\to \Sigma$ constructed from a principal $S^1$-bundle $\pi_P\co P\to \Sigma$ with connection form $\beta\in \Omega^1(P)$ obeying $-\tau d\beta = \pi_{P}^{*}\omega|_{\Sigma}$.  

Also let $\Lambda\subset \Sigma$ be a monotone Lagrangian submanifold, giving rise to Lagrangian submanifolds $\Lambda_{(r)}\subset M$ for $0<r<\sqrt{\tau/\pi}$ by taking the image of the restriction of the radius-$r$ circle bundle to $\Lambda$ under the inclusion of $\mathcal{D}_{\tau}(P)$ into $M$.  It is easy to see what the value of $r$ must be in order for $\Lambda_{(r)}\subset M$ to be monotone: in view of (\ref{omalpha}) we must have $I_{\omega}=\frac{\alpha}{2}I_{\mu}$ on $\pi_2(M,\Lambda_{(r)})$, while a radius-$r$ disk fiber of $\mathcal{D}_{\tau}(P)|_{\Lambda}$ yields a class in $\pi_2(M,\Lambda_{(r)})$ with area $\pi r^2$ and Maslov index $2$, so \[ r^2=\frac{\alpha}{\pi} \]
Conversely, if $\dim M\geq 6$ then \cite[Proposition 6.4.1]{BC09} implies that $\Lambda_{(\sqrt{\alpha/\pi})}\subset M$ will be monotone.

Now the fact that $\Lambda\subset \Sigma$ is Lagrangian implies that the connection given by $\beta$ on $P|_{\Lambda}$ is flat.  In certain circumstances, we can make the stronger statement that the connection given by $\beta$ on $P|_{\Lambda}$ is trivial (\emph{i.e.}, that there is a trivialization of $P|_{\Lambda}$ taking $\beta$ to the trivial connection form $\frac{1}{2\pi}d\theta$ on $\Lambda\times S^1$).  Specifically let us \[ \mbox{assume that }\partial\co \pi_2(\Sigma,\Lambda;\Z)\to \pi_1(\Lambda;\Z) \mbox{ is surjective} \] which of course is equivalent to the inclusion-induced map $\pi_1(\Lambda;\Z)\to \pi_1(\Sigma;\Z)$ being zero (which automatically holds when, as in all cases considered in this paper, $\Sigma$ is simply-connected).  Now if $\gamma\co S^1\to \Lambda$ is any loop, choose a map $u\co D^2\to \Sigma$ with boundary values $u|_{\partial D^2}=\gamma$.  A routine computation then shows that, since $d\beta=-\frac{1}{\tau}\pi_{P}^{*}\omega|_{\Sigma}$, the holonomy map induced by $\beta$-horizontal translation around $\gamma$ is given by rotation of the fiber by angle $\frac{2\pi}{\tau}\int_{D^2}u^*\omega$.

Thus, if the area homomorphism $I_{\omega}\co \pi_2(\Sigma,\Lambda;\Z)\to \R$ takes values only in $\tau\Z$ then the holonomy of the connection given by $\beta$ is trivial.  (Readers familiar with geometric quantization will recognize this as a Bohr-Sommerfeld condition on the Lagrangian submanifold $\Lambda$.)

In this case, the principal $S^1$-bundle $P|_{\Lambda}$ is trivial: choosing a basepoint $p_0\in P|_{\Lambda}$ we may define a bundle isomorphism $\Lambda\times S^1\to P|_{\Lambda}$ by mapping a pair $(x,e^{i\theta})\in \Lambda\times S^1$ to the horizontal translate of $e^{i\theta}p_0$ along any path in $\Lambda$ from $\pi(p_0)$ to $x$, which is independent of the choice of path by the triviality of the holonomy. This bundle isomorphism pulls back the connection form $\beta$ to the standard connection form $\frac{1}{2\pi}d\theta$ on $\Lambda\times S^1$. We then have a commutative diagram \[ \xymatrix{ 
  \Lambda\times S^1\times D(\sqrt{\tau/\pi})\ar[r] & P\times D(\sqrt{\tau/\pi}) \ar[d]
\\ \Lambda\times\{1\}\times  D(\sqrt{\tau/\pi})\ar[u] \ar[r]^>>>>>>>>>g & \mathcal{D}_{\tau}(P) 
} \]where the left map is the obvious inclusion and the bottom map pulls back the form $\Omega$ on $\mathcal{D}_{\tau}(P)$ to the pullback of the standard symplectic form $\omega_{\C}$ on $D(\sqrt{\tau/\pi})\subset \C$ by the projection $\Lambda\times\{1\}\times  D(\sqrt{\tau/\pi})\to D(\sqrt{\tau/\pi})$.  Consequently for any simple closed curve $C\subset D(\sqrt{\tau/\pi})$ we have a Lagrangian submanifold $g(\Lambda\times\{1\}\times C)\subset \mathcal{D}_{\tau}(P)$, and if the curves $C,C'\subset D(\sqrt{\tau/\pi})$ enclose the same area, an area-preserving isotopy of $D(\sqrt{\tau/\pi})$ mapping $C$ to $C'$ induces an exact Lagrangian isotopy from 
$g(\Lambda\times\{1\}\times C)$ to $g(\Lambda\times\{1\}\times C')$.

The Lagrangian submanifold $\Lambda_{(r)}$ is given by $g(\Lambda\times\{1\}\times C)$ where $C$ is the circle of radius $r$.  In particular if $r^2<\frac{\tau}{2\pi}$, then $C$ may be area-preservingly isotoped off of itself in $D(\sqrt{\tau/\pi})$, leading to an exact Lagrangian isotopy from $\Lambda_{(r)}$ to a Lagrangian submanifold of $\mathcal{D}_{\tau}(P)$ which is disjoint from $\Lambda_{(r)}$.  By \cite[Exercise 6.1.A]{P} this exact Lagrangian isotopy can be extended to a Hamiltonian isotopy of $\mathcal{D}_{\tau}(P)$ (and hence to a Hamiltonian isotopy of $M$), proving that $\Lambda_{(r)}$ is displaceable.

Recall that the special value of $r$ leading to monotonicity of $\Lambda_{(r)}$ in $M$ is given by $r^2=\frac{\alpha}{\pi}$.  Thus the above argument leads to a displaceable monotone Lagrangian submanifold of $M$ when $\alpha<\frac{\tau}{2}$.  We summarize the conclusion of this discussion as follows:

\begin{prop}\label{gendisp} Let $(M,\omega,J)$ be a K\"ahler manifold with $[\omega]|_{\pi_2(M)}=\alpha c_1(TM)|_{\pi_2(M)}\neq 0$ where $\alpha>0$, and let $\Sigma$ be a complex hypersurface of $M$ with $\tau PD([\Sigma]) = [\omega]$ where $\tau>0$.  Suppose that $L\subset M$ is a monotone Lagrangian submanifold obtained by the Biran circle bundle construction from a monotone Lagrangian submanifold $\Lambda\subset \Sigma$ such that the boundary map $\partial\co \pi_2(\Sigma,\Lambda)\to \pi_1(\Lambda)$ is surjective.  Assume moreover that: 
\begin{itemize}
\item $\alpha<\frac{\tau}{2}$
\item The area homomorphism $I_{\omega}\co \pi_2(\Sigma,\Lambda)\to \R$ takes values only in $\tau\Z$.
\end{itemize}
Then $L$ is displaceable.
\end{prop}

\begin{ex} If $M=\C P^{n}(\sqrt{2})$ and $\Sigma=Q_n(\sqrt{2})$, then $\tau=\pi$ since the area of a complex line in $\C P^{n}(\sqrt{2})$ is $2\pi$ and such a line has intersection number $2$ with $\Sigma$.  Meanwhile $c_1(TM)$ evaluates on a complex line as $n+1$, so $\alpha=\frac{2\pi}{n+1}$.  Thus the condition $\alpha<\frac{\tau}{2}$ is obeyed provided that $n\geq 4$. One can show that the submanifolds $\mathbb{S}_{k,m}\subset Q_n(\sqrt{2})$ also obey the last condition in Proposition \ref{gendisp}; we omit the proof since we will later prove the displaceability of the relevant submanifolds of $\C P^{n}(\sqrt{2})$ in a more concrete way in Section \ref{cpnsubs}.
\end{ex}

\section{Generalized Polterovich submanifolds}\label{afintro}

For natural numbers $m\geq k\geq 0$ not both zero,  and for a real number $r>0$, define a submanifold $\mathcal{P}_{k,m}^{r}\subset T^{*}S^{k+m+1}$ as the image of the two-to-one map \begin{align*} \iota\co S^1\times S^k\times S^m &\to T^{*}S^{k+m+1} \\ (e^{i\theta},x,y)&\mapsto ((-r\sin \theta\,x,r\cos\theta\,y)  ,(\cos \theta\, x,\sin\theta\,y)) \end{align*}

Here we view $S^k$ and $S^m$ as subsets of $\R^{k+1}$ and $\R^{m+1}$ respectively, and view elements of $T^*S^{k+m+1}$ as pairs $(p,q)\in \mathbb{R}^{k+m+2}\times\mathbb{R}^{k+m+2}$ where $p\cdot q=0$ and $|q|=1$.  (Also, elements of $\R^{k+m+2}$ are written as $(a,b)$ where $a\in \R^{k+1}$ and $b\in \R^{m+1}$.)

We see that $\iota(e^{i\theta},x,y)=\iota(e^{i\phi},u,v)$ if and only if $(e^{i\phi},u,v)=\pm(e^{i\theta},x,y)$, so that $\mathcal{P}_{k,m}^{r}$ is diffeomorphic to the quotient of $S^1\times S^k\times S^m$ by the action of $\mathbb{Z}/2\Z$ given by the antipodal map on each factor.  In particular if $k=0$ then $\mathcal{P}_{k,m}^{r}$ is diffeomorphic to $S^1\times S^m$ (by taking the slice for the $\Z/2\Z$ action consisting of points of form $(e^{i\theta},1,y)$).    

$\mathcal{P}_{k,m}^{r}$ may be characterized as the union of the lifts to $T^{*}S^{k+m+1}$ of the speed-$r$ geodesics that pass through points of the form $(x,0),(0,y)\in S^{k+m+1}$ where $x\in S^k$ and $y\in S^m$.  Equivalently, where for each nonnegative integer $j$ we write $e_{0,j}=(1,0,\ldots,0)\in S^j\subset \R^{j+1}$ and $\vec{0}_j$ for the zero vector in $\R^j$, we may construct $\mathcal{P}_{k,m}^{r}$ by beginning with the lift of the speed-$r$ geodesic passing through $(e_{0,k},\vec{0}_{m+1})$ and $(\vec{0}_{k+1},e_{0,m})$, and then taking the orbit of this lifted curve under the natural action of $SO(k+1)\times SO(m+1)$ on $T^{*}S^{k+m+1}$.  In the special case that $k=0$, yet another characterization of $\mathcal{P}_{0,m}^{r}\subset T^{*}S^{m+1}$ is as the image of the sphere of radius $r$ in the fiber  $T^{*}_{e_{0,m+1}}S^{m+1}$ under the geodesic flow.  In particular $\mathcal{P}_{0,1}^{r}$ is the torus considered in \cite{AF} (which attributes its introduction to Polterovich), and the Remark at the end of \cite{AF} also discusses $\mathcal{P}_{0,m}^{r}$ for $m>1$. Generalizing the argument of \cite{AF}, we prove:

\begin{prop}\label{pnd}
For all $k,m,r$, $\mathcal{P}_{k,m}^{r}$ is a monotone, nondisplaceable Lagrangian submanifold of $T^{*}S^{k+m+1}$.
\end{prop}

\begin{proof}
For $k=0$ this was proven in the remark at the end of \cite{AF}, so let us assume that $1\leq k\leq m$.  Also, the fiberwise dilation by the factor $2r$ is a conformal symplectomorphism of $T^*S^{k+m+1}$ that maps $\mathcal{P}_{k,m}^{1/2}$ to $\mathcal{P}_{k,m}^{r}$, so it suffices to prove the result when $r=1/2$.  

We first prove monotonicity.  Since $k+m+1\geq 3$, the boundary map $\partial \co \pi_2(T^{*}S^{k+m+1},\mathcal{P}_{k,m}^{1/2})\to \pi_1(\mathcal{P}_{k,m}^{1/2})$ is an isomorphism.  Of course, the map $\iota\co S^1\times S^k\times S^m\to \mathcal{P}_{k,m}^{1/2}$ embeds $\pi_1(S^1\times S^k\times S^m)$ as an index-two subgroup of $\pi_1(\mathcal{P}_{k,m}^{1/2})$.  So to establish monotonicity it is enough to prove that the area and Maslov homomorphisms are positively proportional on suitable preimages under $\partial$ of a set of generators for $\iota_*\pi_1(S^1\times S^k\times S^m)$.  Accordingly, define disks $u^{(p)}\co (D^2,\partial D^2)\to (T^{*}S^{k+m+1},\mathcal{P}_{k,m}^{1/2})$ for $p\in\{1,2,3\}$ by: 

\[ u^{(1)}(se^{i\theta})=\left(\frac{1}{2}\left(-s\sin\theta,\vec{0}_{k},s\cos\theta, \vec{0}_{m}   \right),\left(s\cos\theta,\vec{0}_{k},s\sin\theta,\vec{0}_{m-1},\sqrt{1-s^2}\right)\right) \]

 \[ u^{(2)}(se^{i\theta})=\left(\left(\frac{s}{2}\cos\theta,\frac{s}{2}\sin\theta,\vec{0}_{k+m}\right),\left(\vec{0}_{k+m+1},1\right)\right)  \]
 
 \[ u^{(3)}(se^{i\theta}) = \left(\left(\vec{0}_{k+1},\frac{s}{2}\cos\theta,\frac{s}{2}\sin\theta,\vec{0}_{m-1} \right),\left(1,\vec{0}_{k+m+1}\right) \right)\]
 
Here $\vec{0}_{j}$ denotes the zero vector in $\mathbb{R}^j$.

So the boundaries of $u^{(1)},u^{(2)},u^{(3)}$ coincide with the images under $\iota$ of great circles in the factors $S^1,S^k,S^m$, respectively, of $S^1\times S^k\times S^m$.  Of course, if $k>1$, then the corresponding great circle in $S^k$ is contractible and $u^{(2)}$ represents a trivial element in $\pi_2(T^{*}S^{k+m+1},\mathcal{P}_{k,m}^{1/2})$, and likewise if $m>1$ then $u^{(3)}$ is homotopically trivial.  Regardless of the values of $k$ and $m$, let us see that $u^{(2)}$ and $u^{(3)}$ both evaluate trivially under the area and Maslov homomorphisms.

Indeed, the statement that $u^{(2)}$ and $u^{(3)}$ have zero area is obvious, since they have image contained in the cotangent fibers $T^{*}_{(\vec{0}_{k+m},1)}S^{k+m+1}$ and $T^{*}_{(1,\vec{0}_{k+m})}S^{k+m+1}$, respectively, and these cotangent fibers are Lagrangian.  To see that $u^{(3)}$ has trivial Maslov index, it suffices to show that the Lagrangian subbundle $(u^{(3)}|_{\partial D^2})^{*}T\mathcal{P}_{k,m}^{1/2}$ of $(u^{(3)}|_{\partial D^2})^{*}T(T^{*}S^{k+m+1})$ is isotopic through Lagrangian subbundles to the pullback by $u^{(3)}|_{\partial D^2}$ of the tangent bundle to the cotangent fiber $T^{*}_{(1,\vec{0}_{k+m})}S^{k+m+1}$, since the latter Lagrangian subbundle extends over the interior of $u^{(3)}$.

We now confirm this latter statement.  Where we again write $e_{0,k} =(1,\vec{0}_{k})\in S^k$ and where $v_m(\theta)=(\cos\theta,\sin\theta,\vec{0}_{m-1})$, we have $u^{(3)}(e^{i\theta}) = ((\vec{0}_{k+1},v_{m}(\theta)/2),(e_{0,k},\vec{0}_{m+1}))$ and \begin{align*} T_{u^{(3)}(e^{i\theta})}\mathcal{P}_{k,m}^{1/2} = \left\{\left((-te_{0,k}/2,y),(x,tv_m(\theta))\right)|t\in\R,\,x\in e_{0,k}^{\perp}\subset\R^{k+1},\,y\in v_{m}(\theta)^{\perp}\subset\R^{m+1}\right\}
\end{align*}

Defining \[ \mathcal{L}_{a}(e^{i\theta}) = \left\{\left(\left(ax-\frac{(1-a)t}{2}e_{0,k},y+atv_{m}(\theta)\right),\left((1-a)x,(1-a)tv_m(\theta)\right)\right)\right.\qquad \qquad\] 

\[ \qquad \qquad\qquad\qquad\qquad\qquad\left.\left|t\in\R,\,x\in e_{0,k}^{\perp}\subset\R^{k+1},\,y\in v_{m}(\theta)^{\perp}\subset\R^{m+1}\right.\right\} \]for $0\leq a\leq 1$ then gives an isotopy of Lagrangian subbundles from $(u^{(3)}|_{\partial D^2})^{*}T\mathcal{P}_{k,m}^{1/2}$ to the pullback by $u^{(3)}|_{\partial D^2}$ of the vertical subbundle of $T(T^*S^{k+m+1})$.  Thus indeed $u^{(3)}$ has Maslov index zero.  The same argument shows that $u^{(2)}$ has Maslov index zero.

Thus since the homotopy classes of $u^{(1)},u^{(2)},$ and $u^{(3)}$ generate an index-two subgroup of $\pi_2(T^{*}S^{k+m+1},\mathcal{P}_{k,m}^{1/2})$, to see that $T^*S^{k+m+1}$ is monotone it suffices to check that the area and Maslov homomorphisms have the same sign on $u^{(1)}$.  Where $\lambda$ is the canonical one-form on $T^{*}S^{k+m+1}$ we have \[ (u^{(1)})^{*}\lambda = -\frac{s}{2}\sin\theta d(s\cos\theta)+\frac{s}{2}\cos\theta d(s\sin\theta) = \frac{s^2}{2}(\cos\theta d(\sin\theta)-\sin\theta d(\cos\theta)) = \frac{s^2}{2}d\theta \] and so $(u^{(1)})^{*}d\lambda$ is the standard area form $d\left(\frac{s^2}{2}d\theta\right)$ on the unit disk.  Thus the area of $u^{(1)}$ is $\pi$.  Meanwhile we will see below in Proposition \ref{futureprop} that $u^{(1)}$ has Maslov index $2(k+m)$, completing the proof of that $\mathcal{P}_{k,m}^{1/2}$ is monotone modulo Proposition \ref{futureprop}.

Continuing to assume that $k$ (and, as always, also $m$) is positive, $\pi_2(T^{*}S^{k+m+1},\mathcal{P}_{k,m}^{1/2})$ is generated by the classes $[u^{(1)}],[u^{(2)}],[u^{(3)}]$ of the disks just described together with another class $\delta$ which obeys $2\delta= [u^{(1)}]+[u^{(2)}]+[u^{(3)}]$ (the boundary of a representative of $\delta$ goes halfway around great circles in each of the $S^1,S^k,S^m$ factors and so defines a loop in $\mathcal{P}_{k,m}^{1/2}=(S^1\times S^k\times S^m)/(\Z/2\Z)$).  So our above Maslov index computations show that $\delta$ has Maslov index $k+m$, and so the minimal Maslov number of $\mathcal{P}_{k,m}^{1/2}$ is $k+m$ when $k>0$.  (On the other hand when $k=0$ there is no analogue of the class $\delta$ and the minimal Maslov number is $2m$.)  

We now turn to nondisplaceability.  For this purpose we can borrow the argument of \cite{AF}, which established nondisplaceability in the case that $k=0$; for the most part this extends straightforwardly to the case $k>0$, except perhaps in one respect.  Since $\mathcal{P}_{k,m}^{1/2}$ is monotone with minimal Maslov number at least $2$, its Floer homology (with $\Z/2\Z$ coefficients) can be defined as in \cite{Oh95}. The basic idea from \cite{AF} is that, if one makes suitable auxiliary choices in the definition of the Floer complex, including a Morse function $f\co \mathcal{P}_{k,m}^{1/2}\to\R$, then the Floer complex is spanned as a $(\Z/2\Z)$-module by the critical points of $f$, and the Floer differential  counts gradient flowlines of $f$ together with certain strips with positive Maslov index, but the positive-Maslov-index strips come in pairs due to a symmetry and hence do not contribute to the Floer differential over $\Z/2\Z$.

We can use essentially the same symmetry considered in \cite{AF}: define $I_0\co S^{k+m+1}\to S^{k+m+1}$ by \[ I_0(q_1,q_2,\ldots,q_{k+1},q_{k+2},q_{k+3},\ldots,q_{k+m+2}) = (q_1,-q_2,\ldots,-q_{k+1},q_{k+2},-q_{k+3},\ldots,-q_{k+m+2}) \] so that $I_0$ is a reflection through a great circle passing through the points $(e_{0,k},\vec{0}_{m+1})$ and $(\vec{0}_{k+1},e_{0,m})$, and then define $I\co T^*S^{k+m+1}\to T^*S^{k+m+1}$ as the cotangent lift of $I_0$.  Clearly $I(\mathcal{P}_{k,m}^{1/2}) = \mathcal{P}_{k,m}^{1/2}$.  The Morse function $f$ used by \cite{AF} when $k=0$ does not adapt well to the case that $k>0$, so we use a different one, as follows: define $f\co\mathcal{P}_{k,m}^{1/2}\to\R$ by the property that its pullback by the double cover $\iota \co S^1\times S^k\times S^m\to\mathcal{P}_{k,m}^{1/2}$ is given by \[ (f\circ\iota)(e^{i\theta},x,y) = 5\sin(2\theta)+(\cos\theta)(x_1+y_1) \] for $x=(x_1,\ldots,x_{k+1})\in S^k$ and $y=(y_1,\ldots,y_{m+1})\in S^m$.  (The right hand side above is clearly invariant under simultaneous application of the antipodal map on $S^1,S^k,S^m$, so does indeed define a function $f\co \mathcal{P}_{k,m}^{1/2}\to\R$.)  It is easy to check that $f$ is Morse, that $f\circ I = f$, and that at any critical point $(e^{i\theta},x,y)$ of $f\circ \iota$ we must have $\cos\theta\neq 0$, forcing $dx_1=dy_1=0$, so that $x=\pm e_{0,k}$ and $y=\pm e_{0,m}$.  In particular all critical points of $f$ are fixed by the involution $I$.

Now for suitably small $\ep>0$ we use the Morse function $\ep f$ to construct a Floer complex for $\mathcal{P}_{k,m}^{1/2}$ over $\Z/2\Z$ just as in \cite{AF}: where $X_{\ep f}$ denotes the Hamiltonian vector field of an $I$-invariant extension of $\ep f$ to a function on $T^*S^{k+m+1}$ the generators for the Floer complex are integral curves $\gamma\co [0,1]\to T^{*}S^{k+m+1}$ for $X_{\ep f}$ which begin and end on $\mathcal{P}_{k,m}^{1/2}$ (for sufficiently small $\ep$ these will be precisely the constant curves at critical points of $f$), and the Floer differential counts, in the standard way, solutions $u\co \R\times[0,1]\to T^{*}S^{k+m+1}$ to \begin{equation}\label{floereq} \frac{\partial u}{\partial s}+J_t\left(\frac{\partial u}{\partial t}-X_{\ep f}\right) = 0\end{equation} with boundary mapping to $\mathcal{P}_{k,m}^{1/2}$, where $J_t$ is a suitably generic family of tame almost complex structures on $T^*S^{k+m+1}$.  In order to apply the Albers-Frauenfelder symmetry argument we would like to require the almost complex structures $J_t$ to be $I$-invariant, which of course \emph{a priori} might be thought to conflict with genericity.  As in \cite{AF}, \cite[Propositon 5.13]{KS} shows that families of almost complex structures generic amongst those which are $I$-invariant will have the property that all solutions $u$ to (\ref{floereq}) whose images are not contained in $Fix(I)$ are regular; moreover, any solutions $u$ with image contained in $Fix(I)$ necessarily are contractible and so have Maslov index zero.  For small $\ep$, these Maslov-index-zero solutions will have very low energy and so will be contained in a Darboux--Weinstein neighborhood of $\mathcal{P}_{k,m}^{1/2}$, and so assuming that the extension of $f$ to $T^*S^{k+m+1}$ has been chosen appropriately, an argument as in \cite[Proposition 3.4.6]{Poz} implies that all of these Maslov-index-zero solutions will be given by $u(s,t)=\gamma(t)$ for some negative gradient flowline $\gamma$ of $\ep f\co \mathcal{P}_{k,m}^{1/2}\to\R$.  Also, such solutions are cut out transversely as solutions of the Floer equation if and only if the corresponding negative gradient flowlines $\gamma$ are cut out transversely as solutions of the negative gradient flow equation.

Now an argument similar to (but simpler than) the proof of the above-cited result of \cite{KS} shows that, for generic $I$-invariant metrics on $\mathcal{P}_{k,m}^{1/2}$, all negative gradient flowlines for $\ep f$ which are not contained in the fixed locus of $I|_{\mathcal{P}_{k,m}^{1/2}}$ are cut out transversely; indeed this result continues to hold if we restrict to $I$-invariant metrics which coincide with the standard metric (induced by the quotient $\iota\co S^1\times S^k\times S^m\to \mathcal{P}_{k,m}^{1/2}$) on the fixed locus.  The negative gradient flowlines (with respect to the standard metric) which \emph{are} contained in the fixed locus of $I|_{\mathcal{P}_{k,m}^{1/2}}$, on the other hand, are given by $\iota\circ\hat{\gamma}$ where $\hat{\gamma}\co\R\to S^1\times S^k\times S^m$ has the form $\hat{\gamma}(s)=(e^{i \theta(s)},\ep_1 e_{0,k},\ep_2 e_{0,m})$ for $\ep_1,\ep_2\in\{-1,1\}$ and $\theta\co\R\to\R$ obeying the differential equation $\theta'(s) = -(10\cos(2\theta(s))-(\sin\theta(s))(\ep_1+\ep_2))$.  For these special solutions it is a straightforward matter to see that the kernel of the linearization of the negative gradient flow equation has dimension equal to the difference of the indices of the critical points $\lim_{s\to\pm\infty}\gamma(s)$; hence the linearization is surjective and so the negative gradient flowlines that are contained in the fixed locus are indeed cut out transversely.  Consequently for a generic $I$-invariant metric on $\mathcal{P}_{k,m}^{1/2}$ which coincides with the standard metric along the fixed locus, all negative gradient flowlines are cut out transversely, and so just as in \cite{AF} we can find families of $I$-invariant almost complex structures for which all solutions to (\ref{floereq}) are cut out transversely.

We then construct the Floer complex using such a family of $I$-invariant almost complex structures.  Since $I$ acts freely on the set of positive-Maslov-index solutions to (\ref{floereq}), the $\Z/2\Z$-counts of such solutions are zero, and so the Floer differential only counts Maslov-index-zero solutions, all of which are $t$-independent and reduce to negative gradient flowlines.  Consequently the $\Z/2\Z$-Floer complex of $\mathcal{P}_{k,m}^{1/2}$, defined using these data, coincides with the $\Z/2\Z$-Morse complex of the function $\ep f$, and in particular the $\Z/2\Z$-Floer homology of $\mathcal{P}_{k,m}^{1/2}$ is nonzero, proving nondisplaceability.
\end{proof}

\section{Lagrangian submanifolds of quadrics}\label{quadsubs}

Where the quadric $Q_{k+m+2}(\sqrt{2})$ and the sphere $\mathbb{S}_{0,k+m+1}$ are as defined in Section \ref{birquad}, define a map $\Psi\co D_{1}^{*}S^{k+m+1}\to Q_{k+m+2}(\sqrt{2})$ by \[ \Psi(p,q)=[\sqrt{1-|p|^2}:p+iq]. \]  Here we write a general element of $Q_{k+m+2}(\sqrt{2})\subset \C P^{k+m+2}(\sqrt{2})$ as $[z_0:z]$ where $z_0\in\C$, $z\in \C^{k+m+2}$, and $|z_0|^2+|z|^2=2$.   The map $\Psi$, which is easily seen to be a symplectomorphism to its image,  gives a dense Darboux-Weinstein neighborhood of the sphere $\mathbb{S}_{0,k+m+1}\subset Q_{k+m+2}(\sqrt{2})$, with image equal to the complement of the hyperplane section \[ Q_{k+m+1}(\sqrt{2}) = \{[z_0:\cdots:z_{k+m+2}]\in Q_{k+m+2}(\sqrt{2})|z_0=0\} \]  Thus $\Psi$ should be seen as complementary to the map $\Theta_Q$ from Section \ref{birquad}, which gives a dense tubular neighborhood of the symplectic submanifold $Q_{k+m+1}(\sqrt{2})$ with complement $\mathbb{S}_{0,k+m+1}$.

Considering again the submanifolds $\mathcal{P}_{k,m}^{r}$ where now we always assume that $0<r<1$, we see that $\Psi\circ \iota\co S^1\times S^{k}\times S^m\to Q_{k+m+2}(\sqrt{2})$ is given by \[ \Psi\circ \iota(e^{i\theta},x,y) = \left[\sqrt{1-r^2}: ((-r\sin\theta+i\cos\theta)x,(r\cos\theta+i\sin\theta)y)\right] 
 \]  Meanwhile, where $\Theta_Q\co \mathcal{D}_{2\pi}^{k+m+1}(P_Q)\to Q_{k+m+2}(\sqrt{2})$ is the disk bundle embedding from Section \ref{birquad}, one easily computes that, for $(e^{i\theta},x,y)\in S^1\times S^k\times S^m$, \[ 
\Psi\circ \iota(e^{i\theta},x,y) = \Theta_Q([(ix,y),\sqrt{2-2r}e^{-i\theta}]).\]  From this we immediately see that:

\begin{prop}\label{maincorr} For $0<r<1$, the image under $\Psi\co D_{1}^{*}S^{k+m+1}\to Q_{k+m+2}(\sqrt{2})$ of the Lagrangian submanifold $\mathcal{P}_{k,m}^{r}$ is \textbf{equal} to the image under $\Theta_Q\co \mathcal{D}_{2\pi}^{k+m+1}(P_Q)\to Q_{k+m+2}(\sqrt{2})$ of the radius-$\sqrt{2-2r}$ circle bundle over the Lagrangian submanifold \[ \mathbb{S}_{k,m}=\{[ix:y]|x\in S^k,\,y\in S^m\}\subset Q_{k+m+1}(\sqrt{2}) \] of the zero section $Q_{k+m+1}(\sqrt{2})$ of $\mathcal{D}_{2\pi}^{k+m+1}(P_Q)$.
\end{prop}

We now turn to the one remaining loose end in the proof of Proposition \ref{pnd}, namely the computation of the Maslov index of the disk $u^{(1)}$ disk described near the start of the proof of that proposition. 
\begin{prop}\label{futureprop}
The Maslov index of $u^{(1)}$ is $2(k+m)$.
\end{prop}
\begin{proof}Consider  the map $\Psi\circ u^{(1)}\co D^2\to Q_{k+m+2}(\sqrt{2})$.  We see that the restriction $\Psi\circ u^{(1)}|_{\partial D^2}$ maps a general point $e^{i\theta}\in S^1$ to \begin{align*} &\left[ \frac{\sqrt{3}}{2}:\left(-\frac{1}{2}\sin\theta+i\cos\theta,\vec{0}_{k},\frac{1}{2}\cos\theta+i\sin\theta,\vec{0}_{m}\right)\right] 
\\&= \Theta_Q([(-\sin\theta+i\cos\theta,\vec{0}_{k},\cos\theta+i\sin\theta,\vec{0}_{m}),1])
\\&= \Theta_Q([(i,\vec{0}_{k},1,\vec{0}_{m}),e^{-i\theta}]) \end{align*}
Thus the boundary of the disk $\Psi\circ u^{(1)}$ in $Q_{k+m+2}(\sqrt{2})$ is, as an oriented loop, precisely the opposite of the image under $\Theta_Q$ of the boundary of the disk $D_{i,1}$ of radius $1$ in the fiber over $[(i,\vec{0}_{k},1,\vec{0}_{m})]$ in the bundle $\mathcal{D}_{2\pi}^{k+m+1}(P_Q)\to Q_{k+m+1}(\sqrt{2})$.  Thus, in $Q_{k+m+2}(\sqrt{2})$, we may glue the disk $\Psi\circ u^{(1)}$ to the disk $\Theta_Q(D_{i,1})$ along their common boundary to form an oriented sphere $S\subset Q_{k+m+2}(\sqrt{2})$.  Now $\Theta_Q(D_{i,1})$, considered as a disk with boundary on $\Psi(\mathcal{P}_{k,m}^{1/2})$ is easily seen in view of Proposition \ref{maincorr} to have Maslov index $2$.  Consequently (using of course that $\Psi$ is a symplectomorphism to its image) we conclude that the Maslov index of $u^{(1)}$ is given by \begin{equation}\label{mudiff} \mu(u^{(1)})=2\langle c_1(TQ_{k+m+2}(\sqrt{2})),[S]\rangle -2 \end{equation}

Now for any $n$ a standard computation shows that the first Chern class of $TQ_{n}(\sqrt{2})$ is equal to $n-1$ times the restriction of the Poincar\'e dual of a hyperplane in $\C P^{n}$ to $Q_{n}(\sqrt{2})$, \emph{i.e.}, to $n-1$ times the Poincar\'e dual to a transversely-cut-out hyperplane section of the hypersurface $Q_n(\sqrt{2})\subset \C P^{n}$.  For this transverse hyperplane section we may take $\{[\vec{z}]\in Q_{n}(\sqrt{2}) | z_0=0\}$, which is just the image under $\Theta_Q$ of the zero-section of the disk bundle $\mathcal{D}_{2\pi}^{n-1}(P_Q)\to Q_{n-1}(\sqrt{2})$.  Where $k+m+2=n$, the sphere $S$ described in the previous paragraph intersects this zero section once, positively and transversely; hence \[ \langle c_1(TQ_{k+m+2}(\sqrt{2})),[S]\rangle = n-1=k+m+1 \] and by (\ref{mudiff}) $\mu(u^{(1)}) = 2(k+m)$.
\end{proof}

\begin{prop}\label{qmon} For $0\leq k\leq m$, $m\geq 1$, and $0<r<1$, the Lagrangian submanifold $\Psi(\mathcal{P}_{k,m}^{r})\subset Q_{k+m+2}(\sqrt{2})$ is monotone if and only if $r=1-\frac{1}{k+m+1}$.
\end{prop}

\begin{proof} As mentioned earlier, for any $n\geq 2$, if we view $Q_{n}(\sqrt{2})$ as a hypersurface in $\C P^{n}$ then $c_1(TQ_{n}(\sqrt{2}))$ is $n-1$ times the restriction of the Poincar\'e dual to a hyperplane in $\C P^{n}$.  Meanwhile the symplectic form on $Q_{n}(\sqrt{2})$ is the restriction of the symplectic form on $\C P^{n}(\sqrt{2})$, which in turn is cohomologous to $2\pi$ times the Poincar\'e dual of a hyperplane.  Thus the symplectic form $\omega$ on $Q_{n}(\sqrt{2})$ obeys \[ [\omega] = \frac{2\pi}{n-1}c_1(Q_{n}(\sqrt{2}))\in H^2(Q_n(\sqrt{2});\R) \]

So in view of the exact sequence (for $n=k+m+2$) \[ \xymatrix{ 
\pi_2(Q_{n}(\sqrt{2}))\ar[r] & \pi_2(Q_n(\sqrt{2}),\Psi(\mathcal{P}_{k,m}^{r}))\ar[r]^>>>>{\partial} &\pi_1(\Psi(\mathcal{P}_{k,m}^{r}))
       } \] 
$\Psi(\mathcal{P}_{k,m}^{r})$ is monotone if and only if there is a set of elements of $\pi_2(Q_{k+m+2}(\sqrt{2}),\Psi(\mathcal{P}_{k,m}^{r}))$ whose images under $\partial$ generate a finite-index subgroup of $\pi_1(\Psi(\mathcal{P}_{k,m}^{r}))$ and on each of which the area homomorphism takes value $\frac{\pi}{k+m+1}$ times the Maslov homomorphism.

For such a set we could take $\{\Psi\circ u^{(1)},\Psi\circ u^{(2)},\Psi\circ u^{(3)}\}$ where $u^{(1)},u^{(2)},u^{(3)}$ are as in the proof of Proposition \ref{pnd} (but rescaled in the obvious way to have boundary on $\mathcal{P}_{k,m}^{r}$ rather than $\mathcal{P}_{k,m}^{1/2}$); however it is slightly more convenient to replace $\Psi\circ u^{(1)}$ in this set by the image of a radius $\sqrt{2-2r}$ disk fiber of $\mathcal{D}_{2\pi}^{k+m+1}$ under the embedding $\Theta_Q\co \mathcal{D}_{2\pi}^{k+m+1}\to Q_{k+m+2}(\sqrt{2})$ (as we have discussed earlier, this disk fiber has precisely the opposite boundary as $\Psi\circ u^{(1)}$).  Since $\Psi$ is a symplectomorphism to its image, the proof of Proposition \ref{pnd} shows that $\Psi\circ u^{(2)}$ and $\Psi\circ u^{(3)}$ both have area and Maslov index equal to zero, and so obey the desired proportionality.  Meanwhile the disk fiber has area $2\pi(1-r)$ and Maslov index $2$, and so we obtain monotonicity if and only if \[ 2\pi(1-r) = \frac{2\pi}{k+m+1} \] \emph{i.e.}, $r=1-\frac{1}{k+m+1}$
\end{proof}

Accordingly we define \[ L^{Q}_{k,m} = \Psi\left(\mathcal{P}_{k,m}^{1-\frac{1}{k+m+1}}\right); \] we have seen that this is a monotone Lagrangian submanifold of the $2(k+m+1)$-dimensional quadric $Q_{k+m+2}(\sqrt{2})$, diffeomorphic to $\frac{S^1\times S^k\times S^m}{\Z/2\Z}$ where $\Z/2\Z$ acts by the antipodal map on each factor.    Proposition \ref{maincorr} shows that $L_{k,m}^{Q}$ is equal to the submanifold of $Q_{k+m+2}(\sqrt{2})$ obtained by applying the Biran circle-bundle construction to $\mathbb{S}_{k,m}\subset Q_{k+m+1}(\sqrt{2})$, proving half of Theorem \ref{gendimcorr}.

Here is a different, very concrete, characterization of the submanifold $L^{Q}_{k,m}\subset Q_{k+m+2}(\sqrt{2})$.  Where $D(1)$ denotes the open unit disk in $\C$, $\mathbb{H}(1)=D(1)\cap \{z:Im(z)> 0\}$, and $\overline{\mathbb{H}(1)}$ is the closure of $\mathbb{H}(1)$ in $\C$, define \[  f\co D(1)\to \overline{\mathbb{H}(1)} \] by \[ f(a+ib) = -\frac{ab}{\sqrt{1-b^2}}+i\sqrt{1-b^2} \]  Also define \[ c\co D(1)\to (0,\infty) \] by \[ c(z)=\sqrt{2-|z|^2-|f(z)|^2} \]  Note that, for $z_1,z_2\in D(1)$, we have $f(z_1)=f(z_2)$ if and only if $z_1=\pm z_2$.  Moreover, by construction, for all $z\in D(1)$ we have \[ |z|^2+|f(z)|^2+c(z)^2 = 2\qquad \mbox{and}\qquad z^2+f(z)^2+c(z)^2=0,\] from which it follows that, for any $z\in D(1)$, $x\in S^k\subset\R^{k+1}\subset \C^{k+1}$ and $y\in S^m\subset\R^{m+1}\subset\C^{m+1}$, \[ [z:f(z)x:c(z)y]\in Q_{k+m+2}(\sqrt{2}) \]

Conversely, given $z\in D(1)$, $x\in S^k$, $y\in S^m$, one can easily check that a point of $\C P^{k+m+2}(\sqrt{2})$ having the form $[z:ux:vy]$ where $u\in \overline{\mathbb{H}(1)}$ and $v>0$ lies on the quadric $Q_{k+m+2}(\sqrt{2})$ only if $u=f(z)$ and $v=c(z)$.

\begin{prop}\label{lqkmprop} We have \begin{equation}\label{lqkm} L^{Q}_{k,m} = \left\{[z_0:f(z_0)x:c(z_0)y]\left| z_0\in \C,\,|z_0|^2=\frac{1}{k+m+1}\left(2-\frac{1}{k+m+1}\right),\,x\in S^k,\,y\in S^m\right.\right\} \end{equation} 
\end{prop}

\begin{proof} The subgroup $O(k+m+3)\subset U(k+m+3)$ acts naturally on $Q_{k+m+2}(\sqrt{2})$ by restriction of its action on $\C P^{k+m+2}(\sqrt{2})$, and the subgroup of $O(k+m+3)$ which preserves the hyperplane section $Q_{k+m+1}(\sqrt{2})=\{z_0=0\}$ is a copy of $O(k+m+2)$ which acts in Hamiltonian fashion on $Q_{k+m+2}(\sqrt{2})$ with moment map  $\mu_Q\co Q_{k+m+2}(\sqrt{2})\to \frak{o}(k+m+2)$  defined by 
\[ \mu_Q([z_0:z]) = Im(\bar{z}z^{\top}) \] for $z_0\in\C$ and $z\in \C^{k+m+2}$ with $|z_0|^2+\|z\|^2=2$ (where $\frak{o}(k+m+2)$ is identified with $\frak{o}(k+m+2)^*$ by the standard inner product $\langle A,B\rangle=\frac{tr(A^{\top}B)}{2}$).  From this and the remarks immediately before the proposition we see that the right hand side of (\ref{lqkm}) consists of those $[z_0:z]\in Q_{k+m+2}(\sqrt{2})$ such that $|z_0|^2=  \frac{1}{k+m+1}\left(2-\frac{1}{k+m+1}\right)$ and the upper left $(k+1)\times(k+1)$ and lower right $(m+1)\times (m+1)$ blocks of $\mu_Q([z_0:z])$  are zero.

Meanwhile the symplectic embedding $\Psi\co D_{1}^{*}S^{k+m+1}\to Q_{k+m+2}(\sqrt{2})$ pulls back $\mu_Q$ to the moment map $\mu_S$ for the natural $O(k+m+2)$-action on $T^{*}S^{k+m+1}$, which is given by $\mu_S(p,q)=pq^{\top}-qp^{\top}$, and pulls back the function $[z_0:z]\mapsto |z_0|^2$ to $(p,q)\mapsto 1-|p|^2$.  It is straightforward to see that, letting $r= 1-\frac{1}{k+m+1}$, $\mathcal{P}_{k,m}^{r}$ consists of those points $(p,q)$ such that $1-|p|^2=\frac{1}{k+m+1}\left(2-\frac{1}{k+m+1}\right)$ and the upper left $(k+1)\times(k+1)$ and lower right $(m+1)\times (m+1)$ blocks of $\mu_S(p,q)$ are zero.  So since the right hand side of (\ref{lqkm}) is contained in the image of $\Psi$ and since $\mu_S=\mu_Q\circ\Psi$, it follows that the right hand side of (\ref{lqkm}) is equal to $\Psi(\mathcal{P}_{k,m}^{1-\frac{1}{k+m+1}})$, \emph{i.e.}, to $L^{Q}_{k,m}$.
\end{proof}

\begin{prop}\label{qdisp}
For $m\geq 2$, the monotone Lagrangian submanifold $L^{Q}_{0,m}\subset Q_{m+2}(\sqrt{2})$ is displaceable.
\end{prop}

\begin{remark}
Combining Theorem \ref{mains2s2} with an appropriate symplectomorphism $Q_{3}(\sqrt{2})\to S^2\times S^2$ allows one to map $L^{Q}_{0,1}$ to the Chekanov--Schlenk twist torus in $S^2\times S^2$. Thus by \cite[Theorem 2]{CS}, $L^{Q}_{0,1}$ is, unlike all of the other $L^{Q}_{0,m}$, nondisplaceable.
\end{remark}

\begin{proof} For any embedded loop $\gamma\co S^1\to D(1)$ consider the map $G_{\gamma}\co S^1\times S^m\to Q_{m+2}(\sqrt{2})$ defined by \[ G_{\gamma}(e^{i\theta},y) = [\gamma(e^{i\theta}):f(\gamma(e^{i\theta})):c(\gamma(e^{i\theta}))y] \] 
where $f$ and $c$ are as defined just above Proposition \ref{lqkmprop}.  Since $f$ has image in the open upper half-plane, one sees easily that $G_{\gamma}$ is a Lagrangian embedding.  

For a (not necessarily embedded) cylinder $\Gamma\co [0,1]\times S^1\to D(1)$, let us likewise define $\tilde{G}_{\Gamma}\co [0,1]\times S^1\times S^m\to Q_{m+2}(\sqrt{2})$ by $\tilde{G}_{\gamma}(s,e^{i\theta},y) = G_{\Gamma(s,\cdot)}(e^{i\theta},y)$.  We see that, where $\omega_Q$ is the symplectic form on   $Q_{m+2}(\sqrt{2})$, the pullback $\tilde{G}_{\Gamma}^{*}\omega_Q$ is independent of the $S^m$ factor and restricts to each $[0,1]\times S^1\times \{y\}$ as the pullback $\Gamma^{*}\omega'_{D}$ of the nonstandard symplectic form \[ \omega'_D=du\wedge dv+f^*du\wedge dv = \frac{1}{1-v^2}du\wedge dv \] where $u,v$ are the standard real coordinates on $D(1)\subset \C$.  In particular $\tilde{G}_{\Gamma}$ defines an \emph{exact} Lagrangian isotopy from $G_{\Gamma(0,\cdot)}$ to $G_{\Gamma(1,\cdot)}$   if and only if $\Gamma\co  [0,1]\times S^1\to D(1)$ is an exact Lagrangian isotopy in the symplectic manifold $(D(1),\omega'_{D})$, \emph{i.e.}, if and only if the $\omega'_D$-area enclosed by the image of $\Gamma(s,\cdot)$ is independent of $s$.

In particular, if an embedded loop $\gamma\co S^1\to D(1)$ encloses a region having less than one-half of the $\omega'_D$-area of $D(1)$, then  the corresponding Lagrangian submanifold $G_{\gamma}(S^1\times S^m)$ will be displaceable.  Indeed, we can take a smooth isotopy $\Gamma\co [0,1]\times S^1\to D(1)$ such that the $\omega'_D$-area enclosed by $\Gamma(\{s\}\times S^1)$ is independent of $s$, such that $\Gamma(0,\cdot)=\gamma$, and such that $\Gamma(\{1\}\times S^1)$ is disjoint from $\gamma(S^1)$.  The corresponding Lagrangian isotopy $\tilde{G}_{\Gamma}$ will then be an exact Lagrangian isotopy which disjoins $G_{\gamma}(S^1\times S^m)$ from itself, and this exact Lagrangian isotopy may be extended to a Hamiltonian isotopy of $Q_{m+2}(\sqrt{2})$ by \cite[Exercise 6.1.A]{P}.

Now it follows directly from the $k=0$ case of Proposition \ref{lqkmprop} that $L^{Q}_{0,m}$ is the image of a map $G_{\gamma}$ as above where $\gamma$ has image equal to the circle of radius $\sqrt{\frac{1}{m+1}\left(2-\frac{1}{m+1}\right)}$ in $D(1)$.  Now the $\omega'_D$-area of the disk of radius $a$ is, for $0<a<1$, given by  \begin{align*} \int_{0}^{a}\int_{0}^{2\pi}\frac{r}{1-r^2\sin^2\theta}d\theta dr = \int_{0}^{a}\frac{2\pi r}{\sqrt{1-r^2}} dr = 2\pi\left(1-\sqrt{1-a^2}\right) \end{align*}  (For the first equality, use that, by symmetry,\[ \int_{0}^{2\pi}\frac{1}{1-r^2\sin^2\theta}d\theta=4\int_{0}^{\pi/2}\frac{1}{1-r^2\sin^2\theta}d\theta \] and  observe that on the interval $[0,\pi/2)$, $\theta\mapsto \frac{1}{1-r^2\sin^2\theta}$ has antiderivative $\theta\mapsto \frac{\tan^{-1}(\sqrt{1-r^2}\tan\theta)}{\sqrt{1-r^2}}$.)

In particular if $a<\frac{\sqrt{3}}{2}$ (for instance, if $a=\sqrt{\frac{1}{m+1}\left(2-\frac{1}{m+1}\right)}$ where $m\geq 2$) then the disk of radius $a$ has strictly less than half of the $\omega'_D$-area of $D(1)$.  Thus the above argument shows that $L^{Q}_{0,m}$ is displaceable for $m\geq 2$.

\end{proof}

\begin{remark} The proof of the above proposition does not apply to the Lagrangian submanifolds $L^{Q}_{k,m}$ with $k>0$.  Imitating the proof, one could consider Lagrangian immersions $G_{\gamma}\co S^1\times S^k\times S^m\to Q_{k+m+2}(\sqrt{2})$ defined by \[ G_{\gamma}(e^{i\theta},x,y) = [\gamma(e^{i\theta}):f(\gamma(e^{i\theta}))x:c(\gamma(e^{i\theta}))y] \] where $\gamma\co S^1\to D(1)$ is an embedded curve.  Such immersions are not generally embeddings, since (bearing in mind that $f(z)=f(-z)$ and $c(z)=c(-z)$ for all $z$)  if $\gamma(e^{i\theta_1})=-\gamma(e^{i\theta_2})$ then $G_{\gamma}(e^{i\theta_1},x,y)=G_{\gamma}(e^{i\theta_2},-x,-y)$ for all $x\in S^k, y\in S^m$.  In particular if $\gamma_0$ is the circle of radius $\sqrt{\frac{1}{k+m+1}\left(2-\frac{1}{k+m+1}\right)}$ around the origin then $G_{\gamma_0}$ is a two-to-one cover of $L^{Q}_{k,m}$.  To use the method of the proof to displace $L^{Q}_{k,m}$ for $k>0$, one would need to isotope this circle off of itself passing only through loops which bound the same $\omega'_D$-area \emph{and} are symmetric about $0$, which of course is impossible. Relatedly,  in terms of the general condition for displaceability discussed in Section \ref{criterion}, it can be shown that the area homomorphism on $\pi_2(Q_{k+m+1}(\sqrt{2}),\mathbb{S}_{k,m})$ has image $\tau\Z$ when $k=0$ but $\frac{\tau}{2}\Z$ when $k>0$, so that in the latter case the connection used in the Biran circle bundle construction has nontrivial holonomy.

If we instead choose the loop $\gamma$ to be contained in the upper-half disk $\mathbb{H}(1)$, so that in particular $\gamma$ does not pass through any pair of antipodal points, then $G_{\gamma}\co S^1\times S^k\times S^m\to Q_{k+m+2}(\sqrt{2})$ will be an embedding, which is monotone if and only if $\gamma$ bounds $\omega'_D$-area $\frac{2\pi}{k+m+1}$.  If the area enclosed by $\gamma$ is less than one-half of the $\omega'_D$-area of $\mathbb{H}(1)$ (\emph{i.e.}, less than $\frac{\pi}{2}$), then the same method used in the proof of Proposition \ref{qdisp} shows that $G_{\gamma}(S^1\times S^k\times S^m)$ is displaceable.  Thus for $k+m\geq 4$ we obtain a monotone Lagrangian $S^1\times S^k\times S^m$ in $Q_{k+m+2}(\sqrt{2})$ which is displaceable (though, unlike $L^{Q}_{k,m}$, this does not have a direct relationship to the submanifold $\mathcal{P}_{k,m}^{r}\subset T^*S^{k+m+1}$).
\end{remark}

\section{Lagrangian submanifolds of $\mathbb{C}P^n$}\label{cpnsubs}

The submanifolds $\mathcal{P}_{k,m}^{r}\subset T^{*}S^{k+m+1}$ defined in Section \ref{afintro} are invariant under the antipodal involution of $T^*S^{k+m+1}$ and hence descend to submanifolds $\underline{\mathcal{P}}_{k,m}^{r}\subset T^{*}\R P^{k+m+1}$.  The composition of the previously-defined $2$-to-$1$ immersion $\iota\co S^1\times S^k\times S^m\to T^{*}S^{k+m+1}$ having image $ \mathcal{P}_{k,m}^{r}$ with the projection $T^{*}S^{k+m+1}\to T^{*}\R P^{k+m+1}$ identifies $\underline{\mathcal{P}}_{k,m}^{r}$ with up to diffeomorphism with the quotient \[ \frac{S^1\times S^{k}\times S^{m}}{(e^{i\theta},x,y)\sim (-e^{i\theta},-x,-y)\sim (e^{i\theta},-x,-y)}\cong \R P^{1}\times \left(\frac{S^k\times S^m}{(x,y)\sim (-x,-y)}\right) \]  Thus $\underline{\mathcal{P}}_{k,m}^{r}$ is diffeomorphic to $S^1\times\left(\frac{S^k\times S^m}{\Z/2\Z}\right)$ where $\Z/2\Z$ acts by the antipodal map on both factors.  In particular $\underline{\mathcal{P}}_{0,m}^{r}$ is diffeomorphic to $S^1\times S^m$.

\begin{prop}
For all $k,m,r$, $\underline{\mathcal{P}}_{k,m}^{r}$ is a monotone, nondisplaceable Lagrangian submanifold of $T^{*}\R P^{k+m+1}$.
\end{prop}

\begin{proof} Since the quotient projection $\pi\co T^{*}S^{k+m+1}\to T^{*}\R P^{k+m+1}$ is isosymplectic (\emph{i.e.}, pulls back the standard symplectic form on $T^{*}\R P^{k+m+1}$ to the standard symplectic form on $T^{*}S^{k+m+1}$) and $\mathcal{P}_{k,m}^{r}=\pi^{-1}(\underline{\mathcal{P}}_{k,m}^{r})$, the proposition follows rather quickly from the corresponding properties of $\mathcal{P}_{k,m}^{r}$ from Proposition \ref{pnd}.  Indeed, if the Hamiltonian flow of some function $H\co [0,1]\times T^{*}\R P^{k+m+1}\to \R$ were to disjoin $\underline{\mathcal{P}}_{k,m}^{r}$ from itself, then it is easy to check that the Hamiltonian flow of the function $(t,x)\mapsto H(t,\pi(x))$ on $[0,1]\times T^{*}S^{k+m+1}$ would disjoin $\mathcal{P}_{k,m}^{r}$ from itself, contradicting Proposition \ref{pnd}.  Thus $\underline{\mathcal{P}}_{k,m}^{r}$ is nondisplaceable.

As for monotonicity, the quotient projection $\pi$ induces a map $\pi_*\co \pi_2(T^{*}S^{k+m+1},\mathcal{P}_{k,m}^{r})\to \pi_2(T^{*}\R P^{k+m+1},\underline{\mathcal{P}}_{k,m}^{r})$ whose image is an index-two subgroup of $\pi_2(T^{*}\R P^{k+m+1},\underline{\mathcal{P}}_{k,m}^{r})$.  Since $\pi$ is isosymplectic, the map $\pi_*$ intertwines the respective area and Maslov homomorphisms on  $\pi_2(T^{*}S^{k+m+1},\mathcal{P}_{k,m}^{r})$ and $\pi_2(T^{*}\R P^{k+m+1},\underline{\mathcal{P}}_{k,m}^{r})$.  So the fact that (by Proposition \ref{pnd}) the area and Maslov homomorphisms are proportional on $\pi_2(T^{*}S^{k+m+1},\mathcal{P}_{k,m}^{r})$ directly implies the corresponding fact on $\pi_2(T^{*}\R P^{k+m+1},\underline{\mathcal{P}}_{k,m}^{r})$.\end{proof}

Where $f\co [0,1)\to \R$ is defined by $f(x)=\frac{1-\sqrt{1-x^2}}{x^2}$ (smoothly extended to take the value $\frac{1}{2}$ at $0$), the map $\Psi^P\co D^{*}_{1}\R P^{k+m+1}\to \C P^{k+m+1}(\sqrt{2})$ defined by \[ \Psi^{P}([(p,q)]) = \left[ \sqrt{f(|p|)}p+\frac{i}{\sqrt{f(|p|)}}q\right] \]has been shown in Lemma \ref{rpncpn} to be a symplectomorphism to its image, which is the complement $\C P^{k+m+1}(\sqrt{2})\setminus Q_{k+m+1}(\sqrt{2})$ of the quadric.  $\Psi^P$ is manifestly equivariant with respect to the natural $O(k+m+2)$-actions on domain and range, and pulls back the moment map for the $O(k+m+2)$ action on $\C P^{k+m+1}(\sqrt{2})$ (namely $\Phi^C\co [z]\mapsto Im(\bar{z}z^{\top})$) to the moment map for the $O(k+m+2)$-action on $D^{*}_{1}\R P^{k+m+1}$ (namely $\Phi^R\co [(p,q)]\mapsto pq^{\top}-qp^{\top}$).  The norms of these moment maps (with respect to the standard inner product $\langle A,B\rangle=\frac{1}{2}A^{\top}B$ on $\frak{o}(k+m+2)$) are given, respectively, by 

\[ \|\Phi^C([z])\| =\frac{1}{2}\sqrt{\|z\|^4-\left|\sum z_{j}^{2}\right|^2}=\frac{1}{2}\sqrt{4- \left|\sum z_{j}^{2}\right|^2}\qquad \mbox{ and } \|\Phi^R([(p,q)])\| = \|p\| \]

In particular, since the Hamiltonian $\frac{1}{2}\|\Phi^R\|=\frac{\|p\|}{2}$ generates an effective $S^1$-action on $D_{1}^{*}\R P^{k+m+1}\setminus 0_{\R P^{k+m+1}}$ (given by the unit speed geodesic flow) and since $\Psi^P$ pulls back $\|\Phi^C\|$ to $\|\Phi^P\|$, we infer that $\frac{1}{2}\|\Phi^C\|$ generates an effective $S^1$-action on $\C P^{k+m+1}(\sqrt{2})\setminus \R P^{k+m+1}$, as can also be verified by direct computation.

For $0<r<1$, the Lagrangian submanifold $\underline{\mathcal{P}}_{k,m}^{r}$ consists of those points $[(p,q)]$ of $D_{1}^{*}\R P^{k+m+1}$ such that the upper left $(k+1)\times (k+1)$ and lower right $(m+1)\times (m+1)$ blocks of $\Phi^R([(p,q)])$ are zero, and such that $\|\Phi^R([(p,q)])\| = r$.  Consequently the Lagrangian submanifold $\Psi^P(\underline{\mathcal{P}}_{k,m}^{r})\subset \C P^{k+m+1}(\sqrt{2})$ consists of those points $[z]\in \C P^{k+m+1}$ such that 
the upper left $(k+1)\times (k+1)$ and lower right $(m+1)\times (m+1)$ blocks of $\Phi^C([z])=Im(\bar{z}z^{\top})$ are zero and $\|\Phi^C([z])\|=r$.  Writing $z\in \C^{k+m+2}$ and $z=(z_1,z_2)$ where $z_1\in \C^{k+1}$ and $z_2\in \C^{m+1}$, the first of these conditions amounts to the statement that $Im(\bar{z}_1z_{1}^{\top})=0$ and $Im(\bar{z}_2 z_{2}^{\top})=0$, or equivalently that, for some $x\in S^k\subset \R^{k+1}$, $y\in S^m\subset \R^{m+1}$, and $v,w\in\C$ with $|v|^2+|w|^2=2$ we have $z_1=vx$ and $z_2=wy$.  By modifying $z=(z_1,z_2)$ within its equivalence class under the Hopf projection we may as well assume that $w$ is a nonnegative real number (and hence that $w=\sqrt{2-|v|^2}$), and then by possibly switching $x$ to $-x$ we may assume that $v$ lies in the closed upper half-disk $\overline{\mathbb{H}(\sqrt{2})}$ of radius $\sqrt{2}$.  So  $\Psi^P(\underline{\mathcal{P}}_{k,m}^{r})$ is contained in the set of points of form $[(vx,\sqrt{2-|v|^2}y)]$ where $x\in S^k,y\in S^m$, and $v\in\overline{\mathbb{H}(2)}$. 

Such a point in fact lies in $\Psi^P(\underline{\mathcal{P}}_{k,m}^{r})$ if and only if $\|\Phi^C([z])\|=r$.  Given our earlier formula for $\|\Phi^C\|$ this holds if and only if $\left|v^2+2-|v|^2\right|^2=4(1-r^2)$, or equivalently, writing $v=a+ib$, $(1-b^2)^2+a^2b^2=1-r^2$. This equation for $v$ in particular forces $v$ to be in the interior $\mathbb{H}(\sqrt{2})$ of the closed upper half-disk, which implies that the point's representation in the form $[vx:y]$ for $v\in\overline{\mathbb{H}(\sqrt{2})}$ is unique. Thus:

\begin{prop}\label{cpnchar} For $0<r<1$, \[ \Psi^P(\underline{\mathcal{P}}_{k,m}^{r}) = \left\{\left[vx:\sqrt{2-|v|^2}y \right]\left|v\in\mathbb{H}(\sqrt{2}), \left|v^2+2-|v|^2\right|^2=4(1-r^2),\,x\in S^k,y\in S^m  \right.\right\}\]
\end{prop}

As in Proposition \ref{maincorr} we consider the Lagrangian submanifold \[ \mathbb{S}_{k,m}=\{[ix:y]|x\in S^k,\,y\in S^m\}\subset Q_{k+m+1}(\sqrt{2}) \] of the hypersurface $Q_{k+m+1}(\sqrt{2})$.  Recall from Section \ref{bircpn} that where \[ P_P=\frac{\{w\in \C^{k+m+2}|\sum w_{j}^{2}=0,\,\sum |w_j|^2=2\}}{w\sim -w} \] and \[ \mathcal{D}_{\pi}^{k+m+1}(P_P)=\frac{P_P\times D(1)}{([e^{i\theta/2}w],e^{i\theta}\zeta)\sim ([w],\zeta)} \] we have a disk bundle projection $\mathcal{D}_{\pi}^{k+m+1}(P_P)\to Q_{k+m+1}(\sqrt{2})$ given by $[([w],\zeta)]\mapsto [w]$ and a symplectic tubular neighborhood $\Theta_P\co \mathcal{D}_{\pi}(P_P)\to \mathbb{C}P^{k+m+1}(\sqrt{2})$ defined by \[ \Theta_P([([w],\zeta)]) = \left[\sqrt{1-|\zeta|^2/2}w-\zeta\bar{w}/\sqrt{2}\right]  \]  

Similarly to Proposition \ref{maincorr}, we have:

\begin{prop}\label{cpncorr} The Lagrangian submanifold $\Psi^P(\underline{\mathcal{P}}_{k,m}^{r})$ is equal to the image under $\Theta_P$ of the restriction of the radius-$\sqrt{1-r}$ circle bundle in $\mathcal{D}_{\pi}^{k+m+1}(P_P) $ to $\mathbb{S}_{k,m}\subset Q_{k+m+1}(\sqrt{2})$.
\end{prop}

\begin{remark} The case $(k,m)=(0,1)$ of Proposition \ref{cpncorr} proves the final equivalence in Theorem \ref{maincp2}, since $\Psi^P(\underline{\mathcal{P}}_{0,1}^{1/3}) = L_{0,1}^{P}$ and the image of the radius-$\sqrt{2/3}$ circle bundle over $\mathbb{S}_{0,1}$ under $\Theta_P$ is $T_{BC}^{P}$.\end{remark}

\begin{proof} The restriction of the radius-$\sqrt{1-r}$ circle bundle in $\mathcal{D}_{\pi}^{k+m+1}(P_P)$ to $\mathbb{S}_{k,m}$ consists of points of the form $\left[\left(\left[\left(ix,y\right)\right],\zeta\right)\right]$ where $x\in S^k$, $y\in S^m$, and $|\zeta|^2=1-r$.  Such a point is mapped by $\Theta_P$ to $[(\sqrt{1-|\zeta|^2/2}+\zeta/\sqrt{2})ix:(\sqrt{1-|\zeta|^2/2}-\zeta/\sqrt{2})y]\in \C P^{k+m+1}(\sqrt{2})$.  In particular (by multiplying the above expression in brackets by an appropriate phase), we see that for any $x\in S^k$, $y\in S^m$, and $\zeta\in D(1)$,   $\Theta_P\left(\left[\left(\left[\left(ix,y\right)\right],\zeta\right)\right]\right)$ may be expressed as $[vx:\sqrt{2-|v|^2}y]$ for some $v\in \overline{\mathbb{H}(\sqrt{2})}$.  Now  \[
\left\|\Phi^C\left(\left[\left(\left[\left(ix,y\right)\right],\zeta\right)\right]\right)\right\| = \frac{1}{2}\sqrt{4-\left|-\zeta\sqrt{2}\left(2\sqrt{1-\frac{|\zeta|^2}{2}}\right)\right|^2}\]
\[  =\frac{1}{2}\sqrt{4-8|\zeta|^2+4|\zeta|^4} = 1-|\zeta|^2 \]
So since $\Psi^P(\underline{\mathcal{P}}_{k,m}^{r})$ consists of points of form $[vx:\sqrt{2-|v|^2}y]$ such that $\|\Phi^C([vx:\sqrt{2-|v|^2}y])\|=r$, we see that $\Theta_P\left(\left[\left(\left[\left(ix,y\right)\right],\zeta\right)\right]\right)\in \Psi^P(\underline{\mathcal{P}}_{k,m}^{r})$ precisely when $|\zeta|=\sqrt{1-r}$

This proves that the image under $\Theta_P$ of the radius-$\sqrt{1-r}$ circle bundle over $\mathbb{S}_{k,m}$ is contained in $\Psi^P(\underline{\mathcal{P}}_{k,m}^{r})$.  Since $\Theta_P$ is an embedding and since the two manifolds in question are closed, connected, and of the same dimension this suffices to prove the proposition.\
\end{proof}

Likewise we have an analogue of Proposition \ref{qmon}:

\begin{prop} \label{pmon}  The Lagrangian submanifold $\Psi^P(\underline{\mathcal{P}}_{k,m}^{r})$ is monotone if and only if $r=1-\frac{2}{k+m+2}$.
\end{prop}

\begin{proof} A finite index subgroup of $\pi_2(\C P^{k+m+1}(\sqrt{2}),\Psi^P(\underline{\mathcal{P}}_{k,m}^{r}))$ is spanned by the homotopy classes of:
\begin{itemize}\item A complex projective line in $\C P^{k+m+1}(\sqrt{2})$, which has area $2\pi$ and Chern number $k+m+2$, hence Maslov index $2(k+m+2)$.
\item The image under $\Theta_P$ of a disk of radius $\sqrt{1-r}$ in a fiber of the bundle $\mathcal{D}_{\pi}^{k+m+1}(P_P)\to \mathbb{S}_{k,m}$, which has area $\pi(1-r)$ and Maslov index $2$.
\item The images under $\Psi$ of the disks $\pi\circ u^{(2)}$ and $\pi\circ u^{(3)}$, where $\pi\co D_{1}^{*}S^{k+m+1}\to D_{1}^{*}\R P^{k+m+1}$ is the quotient projection and $u^{(2)}$ and $u^{(3)}$ are the obvious rescalings of the disks appearing in the proof of Proposition \ref{pnd}.  The areas and Maslov indices of both of these disks are zero.
\end{itemize}
Thus $\Psi^P(\underline{\mathcal{P}}_{k,m}^{r})$ will be monotone precisely when $\frac{2\pi}{2(k+m+2)}=\frac{\pi(1-r)}{2}$, \emph{i.e.}, $r=1-\frac{2}{k+m+2}$.
\end{proof}

Accordingly we define \[ L^{P}_{k,m}=\Psi^P(\underline{\mathcal{P}}_{k,m}^{1-\frac{2}{k+m+2}}), \] which we have shown to be a monotone Lagrangian $S^1\times\left(\frac{S^k\times S^m}{\Z/2\Z}\right)$ in $\C P^{k+m+1}(\sqrt{2})$.  The two-torus $L_{0,1}^{P}$, in particular, coincides by Theorem \ref{maincp2} with the Chekanov-Schlenk torus in $\mathbb{C}P^2(\sqrt{2})$ (and also the Wu torus from \cite{Wu}), which is nondisplaceable (\cite[Theorem 2]{CS}, \cite[Theorem 1.1]{Wu}, \cite[Section 8.7]{BC12}).  On the other hand:

\begin{prop}\label{highdisp} For $k+m\geq 3$ the monotone Lagrangian submanifold $L^{P}_{k,m}\subset \C P^{k+m+1}(\sqrt{2})$ is displaceable.
\end{prop}

\begin{proof} We argue similarly to the proof of Proposition \ref{qdisp}.  To any simple closed curve $\gamma\co S^1\to \mathbb{H}(\sqrt{2})$ in the open upper half-disk of radius $\sqrt{2}$ we may associate the Lagrangian embedding $G_{\gamma}\co S^1\times\left(\frac{S^k\times S^m}{\Z/2\Z}\right)\to \C P^{k+m+1}(\sqrt{2})$ defined by $G_{\gamma}(e^{i\theta},x,y) = [\gamma(e^{i\theta})x:y]$. (As always, $\Z/2\Z$ acts on $S^k\times S^m$ by the antipodal map on both factors.)  In the special case that the image of $\gamma$ is the curve \[ C_{k,m}=\left\{v\in \mathbb{H}(\sqrt{2}) \left| |v^2+2-|v|^2|^2=4(1-r_{k+m}^{2})\right.\right\}\] where $r_{k+m}=1-\frac{2}{k+m+2}$, the image of $G_{\gamma}$ is our monotone Lagrangian $L^{P}_{k,m}$.  

Likewise, to an isotopy $\Gamma\co [0,1]\times S^1\to \mathbb{H}(\sqrt{2})$ we may associate the Lagrangian isotopy $\tilde{G}_{\Gamma}\co [0,1]\times S^1\times\left(\frac{S^k\times S^m}{\Z/2\Z}\right)\to\C P^{k+m+1}(\sqrt{2})$ defined by $\tilde{G}_{\Gamma}(s,e^{i\theta},x,y)=G_{\Gamma(s,\cdot)}(e^{i\theta},x,y)$; this Lagrangian isotopy will be exact precisely when the image of each $\Gamma(s,\cdot)$ encloses the same area\footnote{Unlike in the proof of Proposition \ref{qdisp}, we compute this area with respect to the standard area form on $\C$} in $\mathbb{H}(\sqrt{2})$, and in this case $\tilde{G}_{\Gamma}$ can be extended to an ambient Hamiltonian isotopy.  In particular if the image of $\gamma\co S^1\to \mathbb{H}(\sqrt{2})$ encloses area less than $\frac{\pi}{2}$ (\emph{i.e.}, less than half of the area of $\mathbb{H}(\sqrt{2})$) then the image of $G_{\gamma}$ will be displaceable, since in this case we can find an isotopy $\Gamma$ through curves all enclosing the same area which disjoins the image of $\gamma$ from itself, giving rise to a Lagrangian isotopy $\tilde{G}_{\Gamma}$ which disjoins the image of $G_{\gamma}$ from itself.

By Lemma \ref{arealemma}, the area enclosed by our curve $C_{k,m}$ is $\pi(1-r_{k+m}) = \frac{2\pi}{k+m+2}$, which is strictly less than $\frac{\pi}{2}$ when $k+m\geq 3$.  Thus the argument of the previous paragraph shows that $L_{k,m}^{P}$ is displaceable when $k+m\geq 3$.

%There remains the matter of determining the area enclosed by the above  curve $C_{k,m}$ in $\mathbb{H}(\sqrt{2})$, which seems nontrivial to do directly.  However we can use our earlier work in the proof of Proposition \ref{cpncorr} for this purpose as follows.  Denoting by $U_{k,m}$ the region enclosed by $C_{k,m}$,  for any $x\in S^k,y\in S^m$, the area of $U_{k,m}\subset \mathbb{H}(\sqrt{2})$ is evidently equal to the area of the disk $\{[vx:\sqrt{2-|v|^2}y]:v\in U_{k,m}\}\subset \C P^{k+m+1}(\sqrt{2})$.  In turn, it follows from the proof of Proposition \ref{cpncorr} that this latter disk is the symplectomorphic image under $\Theta_P$ of the disk of radius $\sqrt{1-r_{k+m}}$ in the fiber of $\mathcal{D}_{\pi}^{k+m+1}(P_P)$ over $[ix:y]\in Q_{k+m+1}(\sqrt{2})$, and of course this disk has area $\pi(1-r_{k+m})=\frac{2\pi}{k+m+2}$.  So provided that $k+m\geq 3$, the region $U_{k,m}$ will have area less than $\pi/2$, which suffices to show that $L^{P}_{k,m}$ is displaceable using the argument in the previous paragraph.
\end{proof}

\section{Floer homology computations in $\C P^3$}\label{cp3sect}

Thus our monotone Lagrangian submanifolds $L_{k,m}^{P}$ are nondisplaceable for the unique case that $k+m+1=2$, but are displaceable for $k+m+1\geq 4$; there remains the matter of the displaceability of the Lagrangian $S^1\times S^2$, $L_{0,2}^{P}$, and the Lagrangian $T^3$, $L_{1,1}^{P}$, in $\C P^{3}$.  We will investigate this matter using Floer homology, as reformulated in terms of the Biran--Cornea pearl complex \cite{BC09} and its twisted-coefficient version \cite{BC12}.  This complex is constructed using a Morse function $f\co L\to \R$ on a monotone Lagrangian submanifold $L$ of a symplectic manifold $(M,\omega)$, a suitable almost complex structure $J$ on $M$, and a homomorphism $\chi\co H_1(L;\Z)\to \C^{*}$; at the level of vector spaces the complex coincides with the $\C$-Morse complex of $f$, while its differential counts ``pearly trajectories'' consisting of a sequence of negative gradient flowlines for $f$ and $J$-holomorphic disks with boundary on $L$, with the count of each pearly trajectory weighted by the image under $\chi$ of the homology class of the $1$-cycle on $L$ traced out by the boundaries of the holomorphic disks that appear in the pearly trajectory.  As noted in \cite[Section 2.4.1]{BC12}, the resulting homology $QH(L;\chi)$ is isomorphic to the self-Floer homology of the pair consisting of $L$ together with a  flat line bundle with holonomy $\chi\co H_1(L;\Z)\to \C^{*}$ (as defined in, \emph{e.g.}, \cite{Cho}).  In particular, if there exists a homomorphism $\chi\co H_1(L;\Z)\to \C^{*}$ such that $QH(L;\chi)$ is nonzero then $L$ is nondisplaceable. 

In the case that, as with both $L_{0,2}^{P}$ and $L_{1,1}^{P}$, the Lagrangian submanifold $L$ is spin and has minimal Maslov number $2$, and $H_1(L;\Z)$ is torsion-free, let us now recall a method (\cite{FOOOtoric}, \cite[Section 3.3]{BC12}) for finding choices of $\chi$ such that $QH(L;\chi)\neq 0$.  Let $\mathcal{E}_2$ denote the subset of $H_2(M,L;\Z)$ consisting of classes with Maslov index two.  To each $A\in \mathcal{E}_2$ we may associate an integer-valued Gromov--Witten-type invariant $\nu(A)\in\Z$ that, for generic compatible almost complex structures $J$, counts $J$-holomorphic disks with boundary on $L$ that pass through a generic point of $L$ on their boundaries (since $L$ is monotone and $A$ has minimal positive Maslov number, no bubbling is possible, and the spin structure on $L$ allows the relevant spaces of $J$-holomorphic disks to be oriented).   Where $\partial\co H_2(M,L;\Z)\to H_1(L;\Z)$ is the standard boundary map, define the superpotential $\mathcal{W}\co Hom(H_1(L;\Z),\C^*)\to \mathbb{C}$ by \[ \mathcal{W}(\chi) = \sum_{A\in \mathcal{E}_2}\nu(A)\chi(\partial A) \]  (Of course, choosing a basis of $H_1(L;\Z)$ (which we have assumed to be torsion-free) identifies $Hom(H_1(L;\Z),\C^*)$ with $(\C^*)^{b_1(L)}$.)
Then \cite[Proposition 3.3.1]{BC12} asserts that if $QH(L;\chi)\neq 0$ then $\chi$ must be a critical point of $\mathcal{W}$, and that the converse holds in the case that (as with $L_{1,1}^{P}$) $H^*(L;\R)$ is generated as a ring by $H^1(L;\R)$.

Accordingly we now set about determining the Maslov-index-two holomorphic disks with boundary on $L_{0,2}^{P}$ and $L_{1,1}^{P}$.  In fact for the first of these it is just as easy to give an argument for applying to all $L_{0,m}^{P}$ for all $m\geq 2$.

For any $k,m$, Proposition \ref{cpncorr} identifies $L_{k,m}^{P}$ with the image under the symplectic embedding $\Theta_P\co \mathcal{D}^{k+m+1}_{\pi}(P_P)\to \C P^{k+m+1}(\sqrt{2})$ of the radius-$\sqrt{2/(k+m+2)}$ circle bundle in $\mathcal{D}^{k+m+1}_{\pi}(P_P)$ over the Lagrangian submanifold $\mathbb{S}_{k,m}\subset Q_{k+m+1}(\sqrt{2})$.  Let us denote by $D\in H_2(\C P^{k+m+1}(\sqrt{2}),L_{k,m}^{P};\Z)$ the relative homology class of (the image under $\Theta_P$ of) the closed radius-$\sqrt{2/(k+m+2)}$ disk in the fiber of $\mathcal{D}^{k+m+1}_{\pi}(P_P)$ over a point $[ix:y]$ of $\mathbb{S}_{k,m}$ (where $x\in S^k,y\in S^m$).   Just as in the proof of Proposition \ref{highdisp}, the image of this disk is equivalently given as the set $\{[vx:\sqrt{2-|v|^2}y]|v\in \bar{U}\}$ where $\bar{U}$ is a certain compact region in the open half-disk $\mathbb{H}(\sqrt{2})$.

\begin{prop}\label{fibercount} For any $k,m$ such that the minimal Maslov number of $L_{k,m}^{P}$ is $2$, the Gromov--Witten invariant $\nu(D)$ associated to $D\in H_2(\C P^{k+m+1}(\sqrt{2}),L_{k,m}^{P})$ obeys  $\nu(D)\in \{-1,1\}$.
\end{prop}
(The condition on the minimal Maslov number is included to prevent disk-bubbling and so to ensure that $\nu(D)$ is well-defined.)

\begin{proof}  
Where $J$ is the standard complex structure on $\C P^{k+m+1}(\sqrt{2})$, we will show that there is up to conformal automorphism a unique,  $J$-holomorphic representative of $D$ with boundary passing through any point in the fiber of $L_{k,m}^{P}$ over the point $[i:0:\cdots:0:1]\in \mathbb{S}_{k,m}$.

Let $\bar{U}\subset \mathbb{H}(\sqrt{2})$ be the domain mentioned just above the statement of the Proposition, and define $\alpha\co \mathbb{H}(\sqrt{2})\to \C$ by $\alpha(z)=\frac{z}{\sqrt{2-|z|^2}}$.  Then where $\phi\co D^2\to \alpha(\bar{U})$ is a biholomorphism (given by the Riemann mapping theorem),  define $u_0\co D^2\to \C P^{k+m+1}(\sqrt{2})$ by \[ u_0(z)=\left[\frac{\phi(z)\sqrt{2}}{\sqrt{1+|\phi(z)|^2}}:0:\cdots:0:\frac{\sqrt{2}}{\sqrt{1+|\phi(z)|^2}}\right] \]  Then $u_0$ is a holomorphic map (it is the composition of the holomorphic quotient projection $\C^{k+m+2}\setminus\{0\}\to \C P^{k+m+1}(\sqrt{2})$ with the map $z\mapsto (\phi(z),0,\cdots,0,1)$) whose image is precisely the image under $\Theta_P$ of the radius$-\sqrt{2/(k+m+2)}$ disk in the fiber over $[i:0:\cdots:0:1]$.  By working in the affine chart corresponding to $z_{k+m+1}=1$ it is straightforward to check that $u_0$ is cut out transversely.

We will now show that any other representative $w\co D^2\to \C P^{k+m+1}(\sqrt{2})$ of $D$ whose boundary passes through a point on the fiber of $L_{k,m}^{P}$ over $[i:0:\cdots:0:1]$ must be a reparametrization of $u_0$.  In this direction, note first that if $V\subset \C P^{k+m+1}(\sqrt{2})$ is any holomorphic subvariety that does not intersect $L_{k,m}^{P}$, then there is a well-defined intersection number $D\cdot V$; if additionally the image of the map $u_0$ from the previous paragraph does not intersect $V$ then it must be that $D\cdot V=0$ and hence that, by positivity of intersections, our other hypothetical representative $w$ of $D$ must also not intersect $V$.   

For $0\leq i<j\leq k+m+1$ let $V_{i,j}$ denote the hyperplane $\{z_i+\sqrt{-1}z_j=0\}$ in $\C P^{k+m+1}(\sqrt{2})$.  When either $0\leq i<j\leq k$, or $k+1\leq i<j\leq k+m+1$, we see that $V_{i,j}$ is disjoint from $L_{k,m}^{P}$.  Also, our disk $u_0$ is disjoint from $V_{0,j}$ for $j\leq k$ and from $V_{i,k+m+1}$ for $i\geq k+1$, and therefore so too must be any other representative $w$ of $D$.  So the image of $w$ is contained in the domain of the partially-defined holomorphic map $\C P^{k+m+1}(\sqrt{2})\dashrightarrow \C^{k+m}$ given by \[ [z_0:\cdots:z_{k+m+1}]\mapsto \left(\frac{z_1}{z_1+\sqrt{-1}z_0},\ldots,\frac{z_k}{z_k+\sqrt{-1}z_0},\frac{z_{k+1}}{z_{k+1}+\sqrt{-1}z_{k+m+1}},\ldots,\frac{z_{k+m}}{z_{k+m}+\sqrt{-1}z_{k+m+1}}\right) \]  Composing $w$ with the above map gives a holomorphic map $\tilde{w}\co D^2\to \C^{k+m}$.  By the characterization of $L_{k,m}^{P}$ from Proposition \ref{cpnchar}, the fact that $w$ has boundary on $L_{k,m}^{P}$ implies that $\tilde{w}|_{\partial D^2}$ takes values only in $\frac{1}{1+i}\R^{k+m}$.  But of course this implies that $\tilde{w}$ is constant.  Also, the fact that the boundary of $w$ passes through a point in the fiber of over $[i:0:\cdots:0:1]$ implies that $\tilde{w}$ takes the value $\vec{0}$ somewhere, and hence everywhere.  Thus our map $w$ must take the form $w(z)=[w_0(z):0:\cdots:0:w_{k+m+1}(z)]$ for some nonvanishing functions $w_0,w_{k+m+1}\co D^2\to \C$ such that $|w_0|^2+|w_{k+m+1}|^2=2$ and $\frac{w_0}{w_{k+m+1}}$ is holomorphic.  The boundary condition on $w$ implies that $\frac{w_0}{w_{k+m+1}}$ has boundary values on $\alpha(\partial\bar{U})$, and hence $\frac{w_0}{w_{k+m+1}}$ is the composition of the earlier biholomorphism $\phi\co D^2\to \alpha(\partial\bar{U})$ with a holomorphic map $D^2\to D^2$.  The homological condition on $w$ readily implies that this latter map is an automorphism of $D^2$ and consequently that $w$ is a reparametrization of $u_0$.

We have thus shown that $u_0$ is the unique $J$-holomorphic representative of $D$ up to reparametrization, and so since $u_0$ is cut out transversely we have $\nu(D)\in \{-1,1\}$.
\end{proof}

\begin{prop}\label{0mdisks} Let $m\geq 2$.  Where $J$ denotes the standard complex structure, the class $D$ is the the \textbf{only} class in $H_2(\C P^{m+1}(\sqrt{2}),L_{0,m}^{P};\Z)$ that has Maslov index two and is represented by a $J$-holomorphic disk.
\end{prop}

\begin{proof}  The long exact sequence for relative homology shows that $H_2(\C P^{m+1}(\sqrt{2}),L_{0,m}^{P};\Z)$ has basis $\{D,\ell\}$ where  $\ell$ is the class of a projective line in $\C P^{m+1}(\sqrt{2})$.  Now if $C\in H_2(\C P^{m+1}(\sqrt{2}),L_{0,m}^{P};\Z)$ and if $V\subset M$ is a compact oriented codimension-two submanifold such that $V\cap L_{0,m}^{P}=\varnothing$ we have a well-defined $\Z$-valued intersection number $A\cdot V$.  If $A$ is represented by a $J$-holomorphic disk and if $V$ is a $J$-complex hypersurface then it will hold that $A\cdot V\geq 0$ by positivity of intersections.  

In particular, where $Q=Q_{m+1}(\sqrt{2})$ denotes the quadric hypersurface, we have $\ell\cdot Q=2$ and $D\cdot Q=1$.  So for a general class $A=\lambda\ell+\delta D$ to be represented by a $J$-holomorphic curve it is necessary that \[ A\cdot Q=2\lambda+\delta\geq 0 \]

Meanwhile the Maslov indices of $\ell$ and $D$ are given by $\mu(\ell)=2(m+2)$ and $\mu(D)=2$.  So for $\mu(A)=2$ 
we must have \[
(m+2)\lambda +\delta = 1,
\]
and hence, since $2\lambda+\delta\geq 0$, \[ m\lambda=1-(2\lambda+\delta)\leq 1,\] forcing $\lambda\leq 0$ since $\lambda\in\Z$ and $m\geq 2$.  

Also, where $H=\{z_{m}+iz_{m+1}=0\}$, the hyperplane $H$ is disjoint from the image under $\Theta_P$ of a disk fiber of $\underline{D}_{\pi}$ over a point of $\mathbb{S}_{0,m}$, as follows from the description of this disk just before the statement of  Proposition \ref{fibercount}.  So $D\cdot H=0$, while of course $\ell\cdot H=1$.  So for the class $A=\lambda\ell+\delta D$ to have nonnegative intersection with $H$ we must have $\lambda\geq 0$.  But we have already shown that $\lambda\leq 0$, so $\lambda=0$ and, in order for $\mu(C)$ to be $2$, it must be that $\delta=1$.
\end{proof}

\begin{cor}\label{floervanish} For $m\geq 2$ and all homomorphisms $\chi\co H_1(L_{0,m}^{P};\Z)\to \C^*$ we have $QH(L;\chi)=0$ 
\end{cor}

(Of course, when $m\geq 3$ this already follows from the fact that $L_{0,m}^{P}$ is displaceable.)

\begin{proof} The basis $\{\ell,D\}$ for $H_1(L_{0,m}^{D};\Z)$ used in the proof of Proposition \ref{0mdisks} identifies $Hom(H_1(L_{0,m}^{P};\Z),\C^{*})$ with $(\C^{*})^2$, by identifying $\chi\in  Hom(H_1(L_{0,m}^{P};\Z),\C^{*})$ with the pair $(\chi(\ell),\chi(D))$.  In terms of this identification, it follows directly from Propositions \ref{fibercount} and \ref{0mdisks} that the superpotential $\mathcal{W}$ is given by \[ \mathcal{W}(t_{\ell},t_{D}) = \pm t_D \] (for some sign $\pm$ which we have no particular need to determine).  Thus $\mathcal{W}$ has no critical points, and the Corollary follows from \cite[Proposition 3.3.1]{BC12}. 
\end{proof}

We now focus entirely on the torus $L_{1,1}^{P}\subset \C P^{3}(\sqrt{2})$. 

 Let us describe $H_2(\C P^{3}(\sqrt{2}),L_{1,1}^{P};\Z)$.  As with the other $L_{k,m}^{P}$ we have the disk fiber class $D$ and the class $\ell$ of a complex projective line.  In addition, we introduce two classes $C_1,C_2\in H_2(\C P^{3}(\sqrt{2}),L_{1,1}^{P};\Z)$ as follows.  By Proposition \ref{cpnchar}, where  $s_{\pm}=\sqrt{1\pm\frac{\sqrt{3}}{2}}$ (so that $s_{\pm}=\sqrt{2-s_{\mp}^{2}}$) the point $[is_{\pm}x:s_{\mp}y]$ lies on $L_{1,1}$ for each $x,y\in S^1$.  The classes $C_1$ and $C_2$ are represented respectively by the maps $v_1,v_2\co D^2\to\C P^3(\sqrt{2})$ defined by \begin{align*} v_1(re^{i\theta}) &= \left[is_{-}r\cos\theta:is_{-}r\sin\theta:\sqrt{2-r^2s_{-}^{2}}:0\right] \\
\\ v_2(re^{i\theta}) &=  \left[i\sqrt{2-r^2s_{-}^{2}}:0:rs_-\cos\theta:rs_-\sin\theta\right] \end{align*}

If $\gamma\co S^1\to \mathbb{H}(\sqrt{2})$ is a counterclockwise parametrization of the curve $\{|v^2+2-|v|^2|^2=3,\,Im(v)>0\}$, there is a double cover $\psi\co S^1\times S^1\times S^1\to L_{1,1}^{P}$ defined by \[ \psi(e^{i\alpha},e^{i\theta},e^{i\phi}) = [\gamma(e^{i\alpha})\cos\theta: \gamma(e^{i\alpha})\sin\theta:\sqrt{2-|\gamma(e^{i\alpha})|^2}\cos\phi:\sqrt{2-|\gamma(e^{i\alpha})|^2}\sin\phi],\] with the unique nontrivial automorphism of the cover given by simultaneous negation on the last two $S^1$ factors.  The images under $\psi$ of the three $S^1$ factors represent the classes of, respectively, $\partial D$, $\partial C_1$, and $\partial C_2$.  So $H_1(L_{1,1}^{P};\Z)$ has basis $\{\partial D,\partial C_1,\beta\}$ where $\beta$ is a class with $2\beta=\partial(C_1+C_2)$.   The long exact relative homology sequence shows that there is a class $B\in H_2(\C P^{3}(\sqrt{2}),L_{1,1}^{P};\Z)$, unique up to addition of a multiple of $\ell$, such that $\partial B=\beta$.  For some $n\in\Z$ we must have $2B=C_1+C_2+n\ell$, and in fact we can (and do)  take $n\in\{0,1\}$ by adding a suitable multiple of $\ell$ to $B$.

To determine whether $n$ is $0$ or $1$, note that where $H_{0}^{+}$ denotes the hyperplane $\{z_0+iz_1=0\}$, we have $\ell\cdot H_{0}^{+}=1$, $C_1\cdot H_{0}^{+}=1$, and $C_2\cdot H_{0}^{+}=0$.  So the fact that $2B=C_1+C_2+n\ell$ implies that $B\cdot H_{0}^{+}=\frac{1+n}{2}$, and hence that $n=1$ since $B\cdot H_{0}^{+}\in \Z$.

We thus conclude that \[ \{B,C_1,C_2,D\} \mbox{ is a basis for }H_2(\C P^{3}(\sqrt{2}),L_{1,1}^{P};\Z) \] and that \begin{equation}\label{lbc} \ell=2B-C_1-C_2 \end{equation}

\begin{prop}\label{onlyclasses} If $A\in H_2(\C P^{3}(\sqrt{2}),L_{1,1}^{P};\Z)$ is a class with Maslov index two that is represented by a $J$-holomorphic disk, where $J$ is the standard almost complex structure on $\C P^{n}(\sqrt{2})$, then \[ A\in \{D,B-D,B-D-C_1, B-D-C_2,B-D-C_1-C_2\}\]
\end{prop}

\begin{proof}  The maps $v_1$ and $v_2$ defined above that represent $C_1$ and $C_2$ clearly have area zero, so by the monotonicity of $L_{1,1}^{P}$ we have $\mu(C_1)=\mu(C_2)=0$ where $\mu$ is the Maslov index.  Meanwhile $\mu(D)=2$, and $\mu(\ell)=2\langle c_1(T\C P^{3}),\ell\rangle=8$, so by (\ref{lbc}) we have $\mu(B)=4$.

Thus if we write a general class $A\in H_2(\C P^{3}(\sqrt{2}),L_{1,1}^{P};\Z)$ as \[ A= bB+c_1C_1+c_2 C_2+dD \] then $\mu(A)=2$ if and only if \begin{equation}\label{mas} 2b+d=1
\end{equation}
Let $Q=Q_{3}(\sqrt{2})$,  $H_{0}^{\pm}=\{z_0\pm iz_1=0\}$, and $H_{2}^{\pm}=\{z_2\pm iz_3=0\}$.  The following table gives the intersection numbers of various classes in $H_2(\C P^{3}(\sqrt{2}),L_{1,1}^{P};\Z)$ with $Q,H_{0}^{\pm},H_{2}^{\pm}$; the first four rows are easily verified by examining explicit representatives, while the last row is obtained from the first three using (\ref{lbc}):\[ \begin{tabular}{c||c|c|c|c|c}
$\cdot$ & $Q$ & $H_{0}^{+}$ & $H_{0}^{-}$ & $H_{2}^{+}$ & $H_{2}^{-}$ \\  \hline
$\ell$ &  $2$ &   $1$       &   $1$       &   $1$       &   $1$  \\
$C_1$  &  $0$ &  $1$        &   $-1$      &   $0$       &   $0$  \\
$C_2$  &  $0$ &  $0$        &   $0$      &   $1$       &   $-1$  \\
$D$    &  $1$ &  $0$        &   $0$      &   $0$       &   $0$   \\
$B$    &  $1$ &  $1$        &   $0$      &   $1$       &   $0$   
\end{tabular}\]

If the Maslov-index-two class $A=bB+c_1C_1+c_2 C_2+dD$ is to be represented by a $J$-holomorphic disk, its intersections with $Q$,$H_{0}^{+},H_{0}^{-}$, $H_{1}^{+}$, and $H_{1}^{-}$ must be nonnegative.  The fact that $A\cdot Q\geq 0$ implies that $b+d\geq 0$, which combined with (\ref{mas}) yields $0\leq b+(1-2b)=1-b$, so $b\leq 1$.

The fact that $A\cdot (H_{0}^{+}+H_{0}^{-})\geq 0$, on the other hand, implies that $b\geq 0$.  Hence \[ \mbox{either }b=0\mbox{ and }d=1\quad \mbox{or}\quad b=1\mbox{ and }d=-1 \]  In the first case, positivity of intersections with both $H_{0}^{+}$ and $H_{0}^{-}$ combine to force $c_1=0$, while positivity of intersections with both $H_{2}^{+}$ and $H_{2}^{-}$ combine to force $c_2=0$.  So the only possibility for the class $A$ when $b=0$  is $A=D$.  

Meanwhile if $b=1$ and $d=-1$ then for $j=0,1$ positivity of intersections with $H_{j}^{+}$ and $H_{j}^{-}$ combine to give $-1\leq c_j\leq 0$.  This yields precisely the rest of the possibilities asserted in the statement of the proposition.
\end{proof}

\begin{prop}\label{fiveclasses}  For each $A\in \{D,B-D,B-D-C_1, B-D-C_2,B-D-C_1-C_2\}$ the corresponding Gromov--Witten invariant $\nu(A)$ is $\pm 1$.
\end{prop}

\begin{proof}  We have already established this result for $A=D$ in Proposition \ref{fibercount}.  For the remaining classes, we use the dense embedding $\Theta_P\co \underline{\mathcal{D}}_{\pi}\to \C P^{3}(\sqrt{2})$ in order to appeal to the results of \cite{BK}.
In particular we compute $\nu(A)$ using an ``admissible'' almost complex structure $J$ as defined in \cite[p. 19]{BK} (this involves first choosing a regular almost complex structure $J_Q$ on the quadric $Q_{3}(\sqrt{2})$, which we take to be the standard complex structure).  These almost complex structures make the quadric $Q_3(\sqrt{2})$ holomorphic, so since all of the classes in the proposition other than $D$ have intersection number zero with $Q_3(\sqrt{2})$, any $J$-holomorphic representative of one of these classes will be contained in $\C P^{3}(\sqrt{2})\setminus Q_3(\sqrt{2})$. But then \cite[Proposition 5.0.2]{BK} implies that these disks must all be contained in (the image under $\Theta_P$ of) a small neighborhood of the radius-$1/\sqrt{2}$-circle bundle in $\mathcal{D}_{\pi}^{k+m+1}(P_P)$, and then \cite[Section 6.1]{BK} shows that such disks are cut out transversely.  

Also, by the definition of admissibility in \cite{BK}, such a disk is contained in a region on which the projection $\underline{\pi}\co\mathcal{D}_{\pi}^{k+m+1}(P_P)\to Q_{3}(\sqrt{2})$ is $(J,J_Q)$-holomorphic, and therefore maps to a $J_Q$-holomorphic disk in $Q_{3}(\sqrt{2})$, which has Maslov index two by \cite[Proposition 2.4.2]{BK}.  Conversely,   for any $p\in L_{1,1}^{P}$ and any $J_Q$-holomorphic disk $u_0\co (D^2,\partial D^2)\to (Q_{3}(\sqrt{2}),\mathbb{S}_{1,1})$ passing through $\underline{\pi}(p)$, by \cite[Lemma 7.1.1]{BK} there is a unique-up-to-automorphism $J$-holomorphic disk $u\co (D^2,\partial D^2)\to (\C P^{n}\setminus Q_{3}(\sqrt{2}),L_{1,1}^{P})$ that passes through $p$, and again by \cite[Proposition 2.4.2]{BK} $u$ and $u_0$ have the same Maslov indices.

Now there is a K\"ahler isomorphism between $Q_{3}(\sqrt{2})$ and $S^2\times S^2$ that takes $\mathbb{S}_{1,1}$ to the product of equators (one can see this, by, for instance, considering the moment map of the action of a maximal torus in $O(4)$ on $Q_{3}(\sqrt{2})$, which has $\mathbb{S}_{1,1}$ as its central fiber, and then appealing to the uniqueness theorem for toric manifolds with a given moment polytope).  In view of this, there are up to automorphism precisely four $J_Q$-holomorphic disks in $Q_{3}(\sqrt{2})$ with boundary on $\mathbb{S}_{1,1}$ passing through a given point of $\mathbb{S}_{1,1}$ on their boundary; moreover the boundaries of these disks represent four distinct classes in $H_{1}(\mathbb{S}_{1,1};\Z)$.  Consequently the lifts of these four disks must have boundaries that represent distinct classes in $H_{1}(L_{1,1}^{P};\Z)$.  So since the only four classes in 
$H_2(\C P^{3}(\sqrt{2}),L_{1,1}^{P};\Z)$ of Maslov index two and intersection number zero with $Q_{3}(\sqrt{2})$ that can be represented by $J$-holomorphic disks are the classes $B-D$, $B-D-A_1$, $B-D-A_2$, and $B-D-A_1-A_2$, we conclude that there must be precisely one disk up to automorphism in each of these four classes with boundary passing through any given point on $L_{1,1}^{P}$.  As mentioned earlier, for the admissible almost complex structure $J$ all disks in $\C P^{3}(\sqrt{2})\setminus Q_{3}(\sqrt{2})$ will be cut out transversely, so we indeed obtain that all four of these classes have Gromov--Witten invariant in $\{-1,1\}$.
\end{proof}

\begin{cor}\label{l11}
There is  $\chi\in Hom(H_1(L_{1,1}^{P};\Z),\C^*)$ such that $QH(L_{1,1}^{P};\chi)\neq 0$.  Thus $L_{1,1}^{P}$ is nondisplaceable.
\end{cor}

\begin{proof}
We identify $Hom(H_1(L_{1,1};\Z),\C^*)$ with $(\C^*)^{3}$ by the map $\chi\mapsto(\chi(\partial B),\chi(\partial C_1),\chi(\partial D))$.   Since $2\partial B=\partial C_1+\partial C_2$ we see from Proposition \ref{fiveclasses} that the superpotential is given in terms of this identification by \[ \mathcal{W}(x,y,z)=\frac{1}{z}\left(\ep_1 x+\ep_2 xy^{-1}+ \ep_3 x^{-1}y +\ep_4 x^{-1}   \right)+\ep_5 z
\] where each $\ep_j\in \{-1,1\}$.  Routine calculation shows that any triple $(x,y,z)$ with the properties that $x^2=\ep_1\ep_4$, $y^2=\ep_1\ep_2\ep_3\ep_4$, $y\neq -\ep_1\ep_2$, and $z^2=\frac{2\ep_4\ep_5(y+\ep_1\ep_2)}{xy}$ will be a critical point of $\mathcal{W}$.  Thus the corollary follows from \cite[Proposition 3.3.1]{BC12}.
\end{proof}

\end{document}